\documentclass[11pt]{article}
\usepackage{amsfonts}
\usepackage{amssymb,amsmath,amsthm, enumerate}
\usepackage{latexsym}
\usepackage{fullpage}
\usepackage{hyperref}
\usepackage{graphicx}
\usepackage{color}

\newtheorem{theorem}{Theorem}[section]

\newtheorem{prop}[theorem]{Proposition}

\newtheorem{lemma}[theorem]{Lemma}
\newtheorem{cor}[theorem]{Corollary}
\newtheorem{defn}[theorem]{Definition}

\theoremstyle{definition}
\newtheorem{remark}[theorem]{Remark}

\newcounter{tenumerate}

\def \B {{\cal B}}

\def \E {{\mathbb{E}}}

\def \P {{\mathbb{P}}}

\def \R {{\mathbb{R}}}

\def \Var {{\rm Var}}
\def \cov{{\rm Cov}}

\newcommand{\pe}{\mathsf{p}}
\begin{document}

\title{{\bf Heat kernel for Liouville Brownian motion and Liouville graph distance }}

\author{ Jian Ding\thanks{Partially supported by an NSF grant DMS-1757479, an Alfred Sloan fellowship, and NSF of China 11628101.}  \\ University of Pennsylvania \and Ofer Zeitouni\thanks{Partially supported by the ERC advanced grant LogCorrelatedFields and by the
Herman P. Taubman professorial chair at the Weizmann Institute.} \\
Weizmann Institute \\ Courant Institute \and Fuxi Zhang\thanks{Partially supported by NSF of China 11771027.} \\ Peking University
}

\date{June 28, 2018}

\maketitle
\begin{abstract}
We show the existence of the scaling exponent $\chi\in (0,4[(1+\gamma^2/4)-
\sqrt{1+\gamma^4/16}]/\gamma^2]$ of the graph distance associated with subcritical 
two-dimensional Liouville quantum gravity of paramater $\gamma<2$ on $\mathbb V =[0,1]^2 $. 
We also show that the Liouville heat kernel satisfies, for any fixed $u,v\in \mathbb V^o$,
 the short time estimates 
$$ \lim_{t\to 0} \frac{\log |\log \pe_t^\gamma(u,v)| }{|\log t|}=\frac{\chi}{2-\chi}, 
\ \mbox{\rm a.s.}$$
  \end{abstract}

\section{Introduction}

Let $\mathbb V =[0,1]^2 \subseteq \mathbb R^2$ and let ${\mathbb V}^o$ denote its interior. Let  $h$ be an instance of the
Gaussian free field (GFF) on $\mathbb V$ with Dirichlet boundary condition. For an introduction
 to the theory of the GFF including various formal constructions, see, e.g., \cite{S07, berestycki16}.  Fix $\gamma\in(0,2)$ and let
 $M_\gamma$ denote the $\gamma$
 Liouville quantum gravity (LQG) 
 given by formally exponentiating the GFF $h$ 
 \cite{DS11}\footnote{ 
 Thus, in our terminology, the LQG is the Gaussian Multiplicative Chaos (GMC) 
 built from the Gaussian free field. 
 As pointed out to us by Remi Rhodes, in the physics literature the
 LQG is often meant to represent a modification of this measure, e.g.
 by normalizition with respect to the total mass of the GMC.
 In this paper we follow the terminology established in \cite{DS11},
 and only note that global, absolutely continuous modifications such
 as a normalization by the area would not change the value of the
 exponents in Theorem \ref{thm-main} below.}. One can then introduce
the positive continuous additive functional (PCAF) with respect to $M_\gamma$
as
\begin{equation}\label{eq-def-PCAF}
F (t) := \int_0^t e^{\gamma h(X_s) - \frac {\gamma^2}2 \E h (X_s)^2} d s, 
\end{equation}
where
$\{ X_t \}$
denotes a standard Brownian motion (SBM) on $\mathbb V$ killed upon exiting $\mathbb V$, independent of $h$.
The Liouville Brownian motion (LBM) is then defined formally as $Y_t:=X_{F^{-1}(t)}$, and
the Liouville heat kernel (LHK) $\pe_t^\gamma (x,y)$ is  the density of the Liouville semigroup with respect to $M_\gamma$, i.e.
\begin{equation}\label{eq-def-LHK}
E^x f(Y_t) = \int \pe_t^\gamma (x,y) f(y) M_\gamma (d y),
 \end{equation}
where the superscript $x$ is to recall that $Y_0=X_0=x$. We refer to  Section~\ref{sec:prelim} for pointers to the (non-trivial) precise construction and 
properties of these objects.

For $\delta>0$ and any two distinct points $u, v\in {\mathbb V}^o$, we define the Liouville graph distance $D_{\gamma, \delta}(u, v)$ to be the minimal number of Euclidean balls with rational centers\footnote{so that $D_{\gamma,\delta}(u,v)$ is a measurable random variable} 
and LQG measure at most $\delta^2$, whose union contains a path from $u$ to $v$.

\begin{theorem}\label{thm-main}
Fix $\gamma\in (0,2)$. There  exists $\chi=\chi(\gamma)\in \big(0,\frac{4[(1+\gamma^2/4) - \sqrt{1+\gamma^4/16}]}{\gamma^2}\big]$ such that the following holds. 
For any   $\iota>0$ and any fixed points  
 $u \neq v\in {\mathbb V}^o$,  there exists a random variable $C=C(\iota,u,v)$ measurable with respect to $h$ such that for all $\delta, t\in (0,1]$,
\begin{align}
C^{-1} \delta^{-\chi+\iota} &\leq D_{\gamma, \delta}(u, v) \leq C \delta^{-\chi - \iota}\,, \label{eq-main-thm-1}\\
C^{-1} \exp\Big\{-t^{-\frac{\chi}{2 - \chi} - \iota}\Big\} &\leq \pe_t^{\gamma}(u, v)   \leq C \exp\Big\{-t^{-\frac{\chi}{2 - \chi} + \iota}\Big\}\,. \label{eq-main-thm-2}
\end{align}
\end{theorem}
As we now discuss, Theorem~\ref{thm-main} is an amalgamation of several results, proved in different sections of the paper.
\begin{itemize}
\item The Liouville graph distance exponent $\chi$ is well defined (see Proposition~\ref{prop-existence-exponent}) and the (log of the) distance
concentrates around its mean (see Proposition~\ref{prop-concentration}).
\item The distance exponent $\chi$ does not depend on the particular choice of $u$ and $v$ as long as they are fixed and away from the boundary (see Proposition~\ref{prop-existence-exponent}). 
\item Both lower bounds and upper bounds on the Liouville heat kernel can be obtained from the distance exponent (see \eqref{eq-heat-kernel-lower-bound-fancy} and \eqref{eq-heat-kernel-upper-bound-fancy}): such bounds are sharp in terms of the power on $t$ in the exponential as in \eqref{eq-main-thm-2}. 
\item The lower bound that $\chi>0$ is a relatively obvious result (see Lemma~\ref{lem-obvious-bounds}); the upper bound on $\chi$ is a reading from the KPZ relation established in \cite{DS11}, which is applied to bound the minimal number of Euclidean balls of LQG 
  measure at most $\delta^2$  required in order to  cover the line segment joining $u$ and $v$. Evaluating
$\chi$ is a major open problem and is not the focus of the present article. 
We record the bounds here only to show that $\chi$ is nontrivial (i.e., $0<\chi<1$), and therefore the heat kernel in
\eqref{eq-main-thm-2} is not diffusive.
\item For $\gamma$ small, non-trivial upper bounds on
  $\chi$  appear in \cite{DG16}. In particular, combining 
  Theorem \ref{thm-main}, \cite[Theorem 1.2]{DG16} and \cite{MRVZ14},
one obtains that there exist constants $c^*,c'>0$ so that 
$\chi\in (1-c'\gamma,1-c^*\gamma^{4/3}/|\log \gamma|)$ for small $\gamma$.
In particular,
as discussed in \cite{DG16}, this is incompatible with
Watabiki's conjecture.
For some work toward bounding exponents for a 
related distance, see \cite{GHS16}. 
\item It is a consequence of \cite{DZ16} and \cite{DG16} 
  that the Liouville graph distance 
  is not universal across different log-correlated fields.
  Because of Theorem \ref{thm-main} and \cite{DZZ17}, the same holds
  for the Liouville heat kernel exponent.
\end{itemize}

\subsection{Background and related results}
Making a rigorous sense of the metric associated with the LQG is
a well known major open problem, see
\cite{RV16} for an up-to-date review. 
In a recent series of works of Miller and Sheffield, 
the special case $\gamma=\sqrt{8/3}$ is treated; one of their achievements is
to produce candidate scaling limits 
and to establish a deep connection to the Brownian map, see  \cite{MS15b, MS15, MS16, MS16b} and references therein.
In a recent work \cite{GHS16}, upper and lower bounds 
have been obtained for a distance associated with the
LQG (which is presumably related to Liouville graph distance considered in the present article), and for that distance
the existence of the scaling
exponent was established. 

From another perspective,
the LBM has also drawn much
interest recently, after
it was constructed in
\cite{GRV13, B14}. In particular, the LBM heat kernel 
was constructed in \cite{GRV14}, and on-diagonal bounds were derived 
in \cite{RV14b}, implying that the
spectral dimension of LBM equals $2$.  Estimates on the 
off-diagonal behavior are more challenging, and some (weak, but non-trivial) 
bounds were
established  in  \cite{MRVZ14} and \cite{AK}, with a significant gap 
in the exponent between the upper and
lower bounds.  Building on \cite{DZ15},  we have computed in \cite{DZZ17} the 
exponent for the Liouville heat kernel on a so-called 
coarse modified branching random walk, 
and showed that the exponent 
is not universal among log-correlated Gaussian fields.
The present article focuses on the GFF set-up and establishes that 
in the precision of the exponent, the off-diagonal LHK is closely
related to the Liouville graph distance.

Another  distance  that has been considered in
the literature is the Liouville first passage percolation (FPP), whose discrete version is the shortest distance metric where each vertex is given a weight of the
exponential of the GFF value there.
In  \cite{DD16}, it was shown that at high temperatures the appropriately normalized Liouville FPP converges subsequentially in the Gromov-Hausdorff sense to a random distance on the unit square, where all the (conjecturally unique) limiting metrics are homeomorphic to the Euclidean distance.  In  \cite{DZ16}, it was shown that the dimension of the geodesic for Liouville FPP is strictly larger than 1. 

Finally, we  mention two random walk models on the environment generated by GFF: in \cite{CGL17} a discrete analog of LBM was considered, where the holding times for the random walk at each vertex are exponential distributions with means given by the exponentials of the GFF --- some scaling limit results were obtained for this model; in \cite{BDG16} a random walk
 on the random network generated by discrete GFF was considered, where in the random network each edge $(u, v)$ is assigned a resistance exponential in the sum of the GFF values at $u$ and $v$ --- the return probability for this random walk was computed via a computation of the effective resistance of this random network.

\subsection{A word on proof strategy and organization of the paper}
Before describing our proof strategy, we discuss some of the basic 
objects that we work with. The first  object 
is the Gaussian free field. There are many approaches for its construction, which we quickly review in Section \ref{subsec-gff}. Of importance to us is its construction 
in terms of integral over space-time white noise, where the `time' coordinate denotes scale. This allows naturally the split of the GFF into an independent sum of 
a `coarse' field, consisting of contributions down to a cutoff scale, and a `fine' field,
consisting of the rest.

Next, the Gaussian multiplicative chaos built from the Gaussian free field, which we refer to as the Liouville quantum gravity, can be constructed as a martingale
limit of the exponential of the coarse field associated with
the GFF, see e.g. \eqref{eq-limit-LQG}.
In particular, it can be described as a product of a function depending
only on the coarse field of the GFF, by an independent 
 measure determined by the fine field; this yields a natural separation of scales, which however as we explain below is not quite sufficient for our analysis.
For this reason, we often work 
with appropriate approximations of the LQG, see for example 
\eqref{eq:WND_decomposition-approximation}. In this sketch, we only mention such details when they are crucial to the argument.

We can now begin to discuss our proof strategies, starting with the Liouville graph distance. As is often the case, the proof of a scaling statement as in 
\eqref{eq-main-thm-1} is based on sub-additivity, which in this case will be with respect to the scale parameter. However, the Liouville
graph distance
 from the introduction is not convenient to work with, because of the lack of scale-separation properties that are crucial for sub-additivity. Therefore, our first step is to relate the Liouville graph distance to an approximate
 Liouville graph distance, obtained through a specific partitioning procedure of the square according to the LQG content of dyadic squares, see Section \ref{subsec-appgraphdist} for details of the construction. Since the approximation involves a sequence of refinements, sub-additivity for the approximate Liouville graph distance is almost built in. However, we need to show that the approximate distance is indeed a good proxy for the distance. This is done in Proposition \ref{prop-approximate-LGD}.
Most of Section \ref{sec:LGD} is devoted to its proof, which employs appropriate 
approximations of the LQG and a-priori estimates of fluctuations of the coarse
field of the GFF. A particularly annoying fact is that the coarse field fluctuations, which typically are well behaved, cannot be well controlled uniformly, and
at places one needs to replace the actual minimizing sequence by 
a proxy, bypassing some bad regions of large fluctuations. This is done in
Lemmas \ref{lem-regularity} and  \ref{lem-partition-independence}, which employ 
percolation arguments.

The approximate graph distance thus constructed also has better
continuity  properties in terms of the underlying GFF, and is instrumental in proving 
that the (logarithm of the) graph distance concentrates around its mean, see 
Proposition \ref{prop-concentration}.
 
Once these preliminary tasks are complete, we turn in Section \ref{sec-LHK}
to the study of off-diagonal short time 
Liouville heat kernel estimates. (We study the 
LHK before showing the convergence of the distance exponent in order to emphasize that the study of the LHK is independent of the latter.) Recall that the 
Liouville Brownian motion is constructed from simple Brownian motion by a time
change that depends on the Liouville quantum gravity.
In Section \ref{subsec-LHKLB}, we prove a lower bound on the LHK, by a technique introduced in \cite{DZZ17}. We construct boxes according 
to the partition yielding the approximate Liouville graph distance. (In reality,
we construct smaller sub-boxes in order to handle differing sizes of blocks in the partition, and bypass some bad regions in the geodesic, using Lemmas
\ref{lem-regularity} and \ref{lem-partition-independence}). In order to control the behavior of the LBM, we introduce the notion of 
`fast boxes', which are boxes in which, from many starting points,
the LBM does not accumulate more time change than typical. Boxes are fast with high probability, and using a Peierls argument, we show that they percolate; the lower bound on the LHK is
obtained by forcing the LBM to follow such a path. For the upper bound, we introduce a parallel notion of `slow boxes', which are cells in which, for enough starting points, the LBM typically accumulates at least a small fraction of the typical time-change. Most cells in the partition determining the approximate Liouville graph distance are slow, and by tracking the accumulated time change, we obtain a lower
bound on the total accumulated time-change, which translates to a LHK upper bound.
We emphasize that the upper bound is obtained in terms of a liminf of the 
Liouville graph distance exponent, while the lower bound is obtained in terms of a limsup.

Finally, in Section \ref{sec:exponent}, we return to the Liouville graph distance.  Using concentration inequalities, it is enough to  prove convergence for the 
rescaled expectation of (the logarithm of) the approximate Liouville graph distance.
Separation of scales is built into the definition, however translation invariance is not (due 
to boundary effects). Further, even though the approximate Liouville 
graph distance 
uses refinements in its construction and thus separation of scales, it still 
suffers from lack of independence across scales. These two factors 
prevent the direct use of sub-additivity. To obtain the latter, we introduce yet another version of the Liouville graph distance, which does possess the required invariance property and, while at a given scale, does not depend on the fine field in slightly smaller scales. A coupling argument allows us to couple the two distances, and sub-additivity can then be employed to give a point-to-point convergence 
of the rescaled log-distance (see Lemma \ref{lem-existence-exponent}), for points 
near the center of the box. 
This is already enough to give an upper bound for arbitrary points. 
To give a lower bound, it is not enough to control point-to-point distances, and we 
need to control point to boundary distances for small enough sub-boxes. The latter
estimate involves the point-to-point estimate and a percolation argument, see
Lemma \ref{lem-exponent-point-to-boundary}. 

Various preliminaries are collected for the convenience of the reader
in Section \ref{sec:prelim}. We also include, in Section \ref{subsec-nonopt}, a derivation of rough
estimates on the distance exponent. These estimates are not expected to be sharp.

\subsection{Notation convention}\label{sec:notation}
We say that the events $E=E_\delta$ occur
with high probability (with respect to $\delta$)  if there exists a constant $c>0$, depending
on $\gamma, \{E_\delta\}$ only, so that 
$\P (E_\delta)\geq 1-\delta^c$ for all small $\delta>0$.
For $\alpha>0$, we say that  the events $E=E_\delta$ occur with 
$\alpha$-high probability, 
if $\P (E_\delta)
\geq 1- \delta^{\alpha}$ for all small $\delta>0$.

 For (nonnegative) functions $F(\cdot)$ and $G(\cdot)$ we write $F = O(G)$ (alternatively, $\Omega(G)$) if there exists an absolute constant $C > 0$ such that $F \leq C G$ (respectively $\geq C G$) everywhere in their domain. We write $F = \Theta(G)$ if $F$ is both $O(G)$ and  $\Omega(G)$. If the constant depends on variables $x_1, x_2, \ldots, x_n$, we change these notations to $O_{x_1,x_2, \ldots, x_n}(G)$ and $\Omega_{x_1,x_2, \ldots, x_n}(G)$ respectively.
We denote by $C, c, C',c_i$ etc positive universal  constants. For parameters
or variables $p_i$, we write $C = C(p_1,\ldots,p_k)$ if $C$ is a positive constant that depends only on $p_1,\ldots,p_k$. For example, $C(\gamma)$ is a positive constant that may depend on $\gamma$. 

For $v\in \mathbb R^2$ and $r>0$, we denote by 
$B_r(v)$ the (open) Euclidean ball centered at $v$ of radius $r$.
For $i\geq 1$ we denote by 
$\mathfrak C_i$ the collection of centers for all dyadic squares of side length  $2^{-i}$ contained in $\mathbb V$. That is,
with $o_{\mathrm{LB}}=(0,0)$,
\begin{equation}\label{eq-def-mathfrak-C}
\mathfrak C_i = \{o_{\mathrm{LB}} + (2^{-i-1}, 2^{-i-1}) + (j\cdot 2^{-i}, k \cdot 2^{-i}):  0\leq j, k\leq 2^i - 1\}.
\end{equation}
Note that $|\mathfrak C_i|=2^{2i}$.

A box $B$ is a square in $\mathbb R^2$. We denote by $s_B$ the side of $B$ and
 by $c_B$ its center. We say that a box $B$ is a dyadic box if, for some $i\in \mathbb N$,
 $s_B=2^{-i}$ and  $c_B\in \mathfrak C_i$.
We say that  a Euclidean ball $B$ is a dyadic ball  if, for 
some $i\in \mathbb N$, the radius of $B$ is $2^{-i}$ and the center of $B$ is in $\mathfrak C_i$.  Finally, we use $|\cdot|$ to denote the Euclidean distance
and $|\cdot|_\infty$ to denote the $\ell_\infty$ norm.

\section{Preliminaries} \label{sec:prelim}

\subsection{General Gaussian inequalities}
The next lemma is a consequence of the the Borell--Sudakov-Tsirelson Gaussian
isoperimetric inequality (\cite{Borell75,ST74}). 
\begin{lemma}\label{lem-Gaussian-concentration}
For any constant $c>0$ there exists $C>0$ such that the following holds.
Let $\mathbf X = (X_1, \ldots, X_n)$ be a centered Gaussian process with $\max_{1\leq i\leq n} \Var X_i= \sigma^2$. Let $B\subseteq \mathbb R^n$ such that $\P(\mathbf X\in B)\geq c$. Then for $\lambda \geq C \sigma$,
$$\P(\min_{\mathbf x\in B} |\mathbf X-\mathbf x|_\infty \geq \lambda) \leq  C e^{-\frac{(\lambda - C\sigma)^2}{2\sigma^2}}\,.$$
\end{lemma}
\begin{proof}
Let $\mathbf X = A \mathbf Z$ where $\mathbf Z$ is a Gaussian vector whose components are i.i.d.\ standard Gaussian variables. Set 
$\tilde B = \{\tilde {\mathbf x}: A \tilde {\mathbf {x}} \in B\}$. By the Cauchy-Schwarz inequality and
the fact that the $\ell_2$-norm for any row vector in $A$ is at most $\sigma$, we 
obtain that
$$
|A z - B|_\infty \ge \lambda \mbox{ implies } |z - \tilde B| \ge \lambda / \sigma \mbox{ for all } z \in \R^n .
$$
Therefore, 
\begin{equation}
\label{eq-BST}
\P(\min_{\mathbf x\in B} |\mathbf X-\mathbf x|_\infty \geq \lambda) \leq \P(\min_{\tilde {\mathbf x}\in \tilde  B} |\mathbf Z- \tilde {\mathbf x}|\geq \lambda/\sigma).
\end{equation} On the other hand,
by  assumption,  $\P(\mathbf Z\in \tilde B) \geq c$. 
Combining this with \eqref{eq-BST} and 
the standard Borell--Sudakov-Tsirelson inequality \cite{ST74, Borell75}, see also \cite[(2.9)]{L01}, yields the lemma.
\end{proof}

The next lemma is a consequence of Lemma~\ref{lem-Gaussian-concentration}. See,  e.g., \cite[(7.4), (2.26)]{L01} as well as discussions in \cite[Page 61]{L01}.
 \begin{lemma} \label{Lem.concentration}
Let $\{ G_z : z \in B \}$ be a Gaussian field on a (countable) index set $B$. Set $\sigma^2 = \sup_{z \in B} \Var (G_z)$. Then, for all $a > 0$,
 $$
 \P (| \sup_{z \in B} G_z - \E \sup_{z \in B} G_z | \ge a) \le 2 e^{-\frac {a^2}{2 \sigma^2}}\,.
 $$
\end{lemma}

We will often need to control the expectation of the
maximum of a Gaussian field in terms
of its covariance structure. This is achieved by Fernique's criterion
\cite{Fernique75}.
We quote a version suited to our needs, which follows straightforwardly
from the version in
\cite[Theorem 4.1]{A90}.
\begin{lemma}
  \label{lem-ferniquecriterion}
  There exists a universal constant $C_F>0$ with the following property.
  Let $B \subset \mathbb V$ denote a box of side length
  $b$ and assume $\{G_v\}_{v\in B}$
  is a mean zero Gaussian field satisfying
  $$\E(G_v-G_u)^2\leq |u-v|/b\,, \mbox{ for all } u, v \in B\,.$$
  Then there exists a version of $\{G_v\}$ which is spatially continuous such that
$\E\max_{v\in B} G_v \leq C_F$.
\end{lemma}
\begin{remark}
When the condition of Lemma~\ref{lem-ferniquecriterion} holds, we always in the sequel consider the continuous version of the underlying Gaussian process.
This allows us to consider the maximum of the process over various
subsets, with the maximum being a bona fide random variable. We use
below this convention without further comment.
\end{remark}

\subsection{Gaussian free field}
\label{subsec-gff}
The GFF $h$ is not defined pointwise, however as a distribution
it is regular enough so that its circle averages are bona fide 
Gaussian variables. In particular, if $|v - \partial \mathbb V| > \delta$ let 
$h_\delta(v)$ denote the average of $h$ along
a circle of radius $\delta$ around $v$. 
Then, the \emph{circle average process} $\{h_\delta (v): v \in \mathbb V, |v - \partial\mathbb V| > \delta \}$  is a centered Gaussian process 
with covariance 
\begin{equation}
\label{eq:GFF_cov}
\cov(h_\delta(v), h_{\delta'}(v')) =  \pi\int_{\partial B_\delta(v) \times \partial B_{\delta'}(v')}G_{\mathbb V}(z, z')\mu_\delta^v(dz)\mu_{\delta'}^{v'}(dz')\,,
\end{equation}
where the normalization
factor of $\pi$ is chosen  to conform with the literature and ensure 
that the GFF is log-correlated.
Here  $\mu_r^v$ is the uniform probability measure on $\partial B_r(v)$, the boundary of $B_r(v)$,
and $G_{\mathbb V}(z, z')$ is the Green function for $\mathbb V$, which is defined by
\begin{equation}
\label{eq:Green_fxn}
G_{\mathbb V}(z, z') = \int_{(0, \infty)}p_{\mathbb V}(s; z, z')ds\,.
\end{equation}
Here and henceforth, for any $A\subset 
\mathbb R^2$,
 $p_{A}(s; z, z')$ is the transition probability 
density of Brownian motion killed upon exiting $A$.  More precisely, $p_{A}(s; z, \cdot)$ is the unique (up to sets of Lebesgue measure 0) nonnegative measurable function satisfying
\begin{equation}
\label{eq:heat_kernel}
\int_{B} p_{A}(s; z, z')dz' = P^z(B_s \in B, \tau_{A} > s)\,,
\end{equation}
for all Borel measurable subsets $B$ of $\R^2$ where $P^z(\cdot)$ is the law of the two-dimensional standard Brownian motion $\{B_t\}_{t \geq 0}$ starting from $z$ and $\tau_{A}$ is the exit time of $\{B_t\}_{t \geq 
0}$ from $A$. It was shown in \cite{DS11} that there exists a version of the circle average process which is jointly H\"{o}lder continuous in $v$ and $\delta$ of order $\vartheta < 1/2$ on all compact subsets of $\{(v, \delta): v \in \mathbb V, |v - \partial \mathbb V| > \delta\}$.
In particular, the LQG measure 
can be defined as the limit of 
\begin{equation}
M^\circ_{\gamma,\delta} (dv)=e^{\gamma h_\delta(v)-\frac{\gamma^2}{2} \log (1/\delta)} 
{\mathcal L}_2(dv), \end{equation}
where $\mathcal L_2$ denotes the two-dimensional Lebesgue measure 
(restricted to $\mathbb V$), and the superscript $\circ$ indicates a circle average approximation is taken. Similarly, the
 functional in \eqref{eq-def-PCAF} can be defined by replacing there $h$ with $h_\delta$ and then taking the limit as 
$\delta \to 0$ (see \eqref{eq-limit-LQG} below).

We will also use the white noise decomposition of the
GFF. A white noise $W$ distributed on $\R^2 \times \R^+$ refers to a centered Gaussian process $\{(W, f): f \in L^2(\R^2 \times \R^+)\}$ whose covariance kernel is given by $\E 
(W, f) (W, g) = \int_{\R^2 \times \R^+}fgdzds$. An alternative and suggestive notation for 
$(W, f)$, which we will use in the sequel,  is $\int_{\R^2 \times 
\R^+}f W(dz, ds)$. For any $B \in \B(\mathbb R^2)$ and $I \in \B(\mathbb R^+)$, we let $\int_{B \times I}f W(dz, ds)$ denote the variable $\int_{\R^2 \times \R^+}f_{B \times I} W(dz, ds)$, where $f_{B \times 
I}$ is the restriction of $f$ to $B \times I$. Now define the Gaussian process $\{\tilde h_\delta^{\tilde \delta}(v) : v \in  \mathbb V, \tilde \delta>\delta >0\}$ 
by
\begin{equation}
\label{eq:WND_decomposition}
\tilde h_\delta^{\tilde \delta}(v) = \sqrt{\pi}\int_{\mathbb V \times (\delta^2, \tilde \delta^2)}p_{\mathbb V}(s/2; v, w)W(dw, ds)
\end{equation}
(for notation convenience, we will drop the superscript $\tilde \delta$ when $\tilde \delta = \infty$). Then $\tilde h_{\delta}$ is another approximation of the GFF
 as $\delta \to 0$, known as the white noise decomposition. The LQG measure as well as the functional in \eqref{eq-def-PCAF} can also be approximated by taking a limit with the white noise decomposition, and it has been shown in \cite[Theorem 5.5]{RV14} and \cite{Shamov16} that the limiting law is the same as with the circle average approximation. For future reference we note that for $u,v\in \mathbb V$, the
Chapman-Kolmogorov equations give that
\begin{equation}
\label{eq-cov-tildeh}
\E(\tilde h_\delta^{\tilde \delta}(u)\tilde h_\delta^{\tilde \delta}(v))= \pi \int_{\delta^2}^{\tilde \delta^2} p_{\mathbb V}(t;u,v) dt.
\end{equation}

We will, in fact, consider an approximation of the white noise decomposition. To this end, we define for $0<\delta < \tilde \delta \leq \infty$
\begin{equation}\label{eq:WND_decomposition-approximation}
\eta_\delta^{\tilde \delta}(v) = \sqrt{\pi}\int_{\mathbb V \times (\delta^2, \tilde\delta^2)}p_{\mathbb V \cap B_{ 4^{-1} s^{1/2} |\log s^{-1}| \wedge 10^{-1}}(v)}(s/2; v, w)W(dw, ds)\,,
\end{equation}
where we recall that
$B_r(v)$ is the Euclidean ball of radius $r$ centered at $v$. Here we truncate the transition density upon exiting $B_{4^{-1} s^{1/2} |\log s^{-1}| \wedge 10^{-1}}(v)$ (or exiting $\mathbb V$) so that each scale in the hierarchical structure of the process $\eta_{\delta}^{\tilde \delta}$ (that is, the process $\{\eta_{\delta'}^{2\delta'}(v): v\in \mathbb V\}$ for some $\delta \leq \delta' \leq\tilde \delta/2$) only has local dependence --- the ``$\wedge 10^{-1}$'' in the definition is to ensure \eqref{eq-translation-invariant} in Section~\ref{sec:exponent} and is
otherwise not important. Again, for notation convenience, we will drop the superscript $\tilde \delta$ when $\tilde \delta = \infty$.

\begin{lemma}\label{lem-variance-continuity}
With notation as above, we have that
\begin{equation*}
\Var(\tilde h_\delta(u) - \tilde h_{\delta}(v)) + \Var(\eta_\delta(u) - \eta_{\delta}(v)) = O(\frac{|u-v|}{\delta}), \mbox{ uniformly in } \delta>0, u, v\in \mathbb V\,.
\end{equation*}
\end{lemma}
\begin{proof}
We will give a proof for the  bound on 
$\Var(\tilde h_\delta(u) - \tilde h_{\delta}(v))$.  The 
bound on $\Var(\eta_\delta(u) - \eta_{\delta}(v))$ follows 
from a similar argument. Our proof follows \cite[Appendix A]{RV14}, where a version of Lemma~\ref{lem-variance-continuity} is proved, with $|u- v|=O(\delta^2)$ and 
where both $u,v$
 are away from $\partial \mathbb V$. We will adapt their arguments and show that these restrictions are not needed.
Because of \eqref{eq-cov-tildeh}, estimates on $p_{\mathbb V}(t;u,v)$ will play an
important role.  Note that
$$p_{\mathbb V}(t; u, v) = \frac{e^{-\frac{|u-v|^2}{2t}}}{2\pi t} q(t; u, v) \mbox{ where } q(t; u, v) = P(B_s - \frac{s}{t}B_t + u + \frac{s}{t}(v-u)\in \mathbb V \mbox{ for all } s\leq t)\,.$$
Therefore, we get that
\begin{align*}
\pi \int_{\delta^2}^\infty (p_{\mathbb V}(t; u, u) - p_{\mathbb V}(t; u, v)) dt  \leq  \int_{\delta^2}^\infty \frac{1}{2t}( q(t; u, u) - q(t; u, v) )dt + \int_{\delta^2}^\infty \frac{1}{2t}q(t; u, v) (1-e^{-\frac{|u-v|^2}{2t}}) dt\,.
\end{align*}
Using the fact that $1 - e^{-x} \leq \sqrt{x}$ for $x>0$,  we get that 
\begin{equation}\label{eq-RV-1}
\int_{\delta^2}^\infty q(t; u, v) \frac{1}{2t}(1-e^{-\frac{|u-v|^2}{2t}}) dt \leq \int_{\delta^2}^\infty \frac{|u-v|}{t^{3/2}} dt \leq 2\frac{|u-v|}{\delta}\,.
\end{equation}
Let $\tau = \min\{s\leq t: B_s - \frac{s}{t}B_t + u \not\in \mathbb V\}$ and $\tau' = \min\{s\leq t: B_s - \frac{s}{t}B_t + u +\frac{s}{t} (v-u) \not
\in \mathbb V\}$ where we use the convention that $\min\emptyset = \infty$. Then we see that 
\begin{equation}\label{eq-tau's-decomposition}
|q(t; u, u) - q(t; u, v)| \leq P(\tau \leq t, \tau' >t) + P(\tau' \leq t, \tau >t).
\end{equation}
The two terms on the right hand side of \eqref{eq-tau's-decomposition} can be bounded in a similar way. As a result, we just bound 
$P(\tau \leq t, \tau' >t)$. To this end, we denote by $\mathbb L_1, \ldots, \mathbb L_4$ the four boundary segments of $\mathbb V$, and let $\tau_i = \min\{s\leq t: B_s - \frac{s}{t}B_t + u \in \mathbb L_i\}$ for $i=1, \ldots, 4$. It is clear that 
$$P(\tau \leq t, \tau' >t)  \leq \mbox{$\sum_{i=1}^4$} P(\tau_i \leq t, \tau' >t)\,.$$
Assume that $\mathbb L_1$ is the left boundary of $\mathbb V$. The event
$\tau_1\leq t$ implies that $\min_{s\in [0,t]} (B_s-\frac{s}{t} B_t)_1\leq -u_1$,
while the event $\tau'>t$ implies that  $\min_{s\in [0,t]} (B_s-\frac{s}{t} B_t)_1\geq -(1-\frac{s}{t}) u_1 - \frac{s}{t}v_1$ for some $0 < s \leq t$. Here we use the notation $w_1$ for the
$x$-coordinate of some $w\in \mathbb R^2$. Thus, the intersection is possible only if $v_1>u_1$, and in that case we obtain that
$$P(\tau_1 \leq t, \tau' >t) \leq P(\min_{s\in [0,t]} (B_s - \tfrac{s}{t}B_t)_1 \in [-v_1,-u_1])=
P(\max_{s\in [0,t]} (B_s - \tfrac{s}{t}B_t)_1 \in [u_1,v_1])\,.$$
By the reflection principle,
for $v_1>u_1$ we have that 
\begin{eqnarray*}
&&P(\max_{s\in [0,t]} (B_s - \frac{s}{t}B_t)_1 \in [u_1,v_1])=
\int_{u_1}^{v_1}- \frac{d}{dx}\left(\frac{p(t;0,2x)}{p(t;0,0)} \right)
dx\\
&&\quad
=e^{-2u_1^2/t}-e^{-2v_1^2/t}\leq C\frac{|u_1-v_1|}{\sqrt{t}}.
\end{eqnarray*}
Repeating this argument for $i=1,\ldots,4$, we conclude that
$$P(\tau \leq t, \tau' >t)\leq 4C\frac{|u-v|}{\sqrt{t}},$$
which gives, using 
\eqref{eq-tau's-decomposition}, that $q(t; u, u) - q(t; u, v) = O(|u-v|/\sqrt{t})$.
Therefore,  
$$\int_{\delta^2}^\infty \frac{1}{2t} [q(t; u, u) - q(t; u, v)] dt = O(\frac{|u-v|}{\delta}).$$
Combined with \eqref{eq-RV-1} we get that 
\begin{equation}
\label{eq-feb23}\pi \int_{\delta^2}^\infty [p_{\mathbb V}(t; u, u) - p_{\mathbb V}(t; u, v)] dt  = O(\frac{|u-v|}{\delta})\,.
\end{equation}
Interchanging the roles of $u$ and $v$, we obtain the same estimate for 
$\pi \int_{\delta^2}^\infty [p_{\mathbb V}(t; v, v) - p_{\mathbb V}(t; u, v)] dt $.
Recalling \eqref{eq-cov-tildeh}, we have 
$$\Var (\tilde h_\delta(u) - \tilde h_\delta(v)) = \pi \int_{\delta^2}^\infty [p_{\mathbb V}(t; u, u) - p_{\mathbb V}(t; u, v) ]dt + \pi \int_{\delta^2}^\infty[ p_{\mathbb V}(t; v, v) - p_{\mathbb V}(t; u, v) ]dt \,,$$
and substituting \eqref{eq-feb23},
we complete the proof of the lemma.
\end{proof}
\begin{lemma}\label{lem-continuity-h-eta}
Uniformly in $\delta>0$,  $a>0$ and $k\geq 1$, we have
\begin{align*}
&\sup_{u\in \mathbb V}\P\left(\max_{v: |v-u| \leq k \delta} |\eta_\delta(v) - \eta_\delta(u)| \geq a \log (k+1)\right)= O(1) e^{-\Omega(a^2)}\,.\\
&\E \max_{u, v\in \mathbb V, |u-v|\leq \delta} \left(|\tilde h_\delta(u) - \tilde h_\delta(v)| + |\eta_\delta(v) - \eta_\delta(u)|\right) = O(\sqrt{\log \delta^{-1}}).
\end{align*}
\end{lemma}
\begin{proof}
By Lemma~\ref{lem-variance-continuity}, we can apply Lemma~\ref{lem-ferniquecriterion} and deduce that for all $u\in \mathbb V$
$$\E \max_{v\in \mathbb V: |u-v|\leq \delta} \left(|\tilde h_\delta(u) - \tilde h_\delta(v)| + |\eta_\delta(v) - \eta_\delta(u)|\right) = O(1)\,.$$
Combined with Lemma~\ref{Lem.concentration}, this yields the second inequality by considering a union bound over $u\in \mathfrak C_{\lceil \log_2 \delta^{-1} \rceil +1}$ (recall the definition of  $\mathfrak C_i$ in \eqref{eq-def-mathfrak-C}). In addition, by a similar argument, we get that  uniformly in $a,k,\delta$,
\begin{equation}\label{eq-boring-1}
\sup_{u\in \mathbb V}\P\left(\max_{v: v\in \mathfrak C_{\lceil \log_2 \delta^{-1} \rceil +1}, |v-u| \leq k\delta} \max_{x: |x-v|\leq \delta}   |\eta_\delta(v) - \eta_\delta(x)| \geq a \log (k+1)/2\right) \leq e^{-\Omega(a^2)}\,.
\end{equation}
Since $\Var (\eta_\delta(v) - \eta_\delta(u)) = O(\log (k+1))$ for all $|v-u| \leq k\delta$, a union bound yields that uniformly in the same parameters,
$$\sup_{u\in \mathbb V}\P\left(\max_{v: v\in \mathfrak C_{\lceil \log_2 \delta^{-1} \rceil +1}, |v-u| \leq k\delta}  |\eta_\delta(v) - \eta_\delta(u)| \geq a \log (k+1)/2\right) \leq O(1) e^{-\Omega(a^2)}\,.$$
Combined with \eqref{eq-boring-1} and the fact that 
\begin{align*}
\max_{v: |v-u| \leq k \delta} |\eta_\delta(v) - \eta_\delta(u)| \leq &\max_{v: v\in \mathfrak C_{\lceil \log_2 \delta^{-1} \rceil +1}, |v-u| \leq k\delta} \max_{x: |x-v|\leq \delta}   |\eta_\delta(v) - \eta_\delta(x)| \\
&+ \max_{v: v\in \mathfrak C_{\lceil \log_2 \delta^{-1} \rceil +1}, |v-u| \leq k\delta}  |\eta_\delta(v) - \eta_\delta(u)|\,,
\end{align*}
this yields the first inequality of the lemma. 
\end{proof}

Recall the definition of $\mathfrak C_i$ in \eqref{eq-def-mathfrak-C}. By a simple union bound, we get that
$$\E \max_{v\in \mathfrak C_{\lfloor \log_2 \delta^{-1} \rfloor}} \tilde h_\delta(v)\leq 2 \log \delta^{-1} + O(1) \mbox{ for all } \delta>0\,.$$
Combined with Lemma~\ref{Lem.concentration} and Lemma~\ref{lem-continuity-h-eta}, we obtain that uniformly in $\lambda > 0$ and small $\delta>0$,
\begin{equation}\label{eq-max-white-noise-process}
\P(\max_{v\in \mathbb V} \tilde h_\delta(v) \geq 3 \log \delta^{-1} + \lambda) \leq O(1) e^{-\frac{\lambda^2}{2\log \delta^{-1}+O(1)}}\,.
\end{equation}

\begin{lemma}\label{lem-tilde-h-eta}
 We have $\P( \max_{v\in \mathbb V} \max_{j\geq 0} |\tilde h_{2^{-j}}(v) - \eta_{2^{-j}}(v)|\geq \lambda)  \leq O(1)e^{-\Omega(\lambda^2)}$.
\end{lemma}
\begin{proof}
We may and will assume that $\lambda>C$ for some constant $C$ large enough. 
For $i \geq 1$, write $\Delta_i(v) = \tilde h_{2^{-i}}^{2^{-i+1}}(v) - \eta_{2^{-i}}^{2^{-i+1}}(v)$ and write $\Delta_0(v) = \tilde h_1(v) - \eta_1(v)$. 
Let $\tau_i=\min\{t>0: |B_t-t2^{2i}B_{2^{-2i}}|_\infty \geq i2^{-i}/8\}$ where $\{ B_t \}$ is a standard Brownian motion.
Uniformly in  $v\in \mathbb V$ and $i$ we have 
\begin{equation}\label{eq-variance-truncation}
\Var \Delta_i(v) = O(1) P(\tau_i\leq 2^{-2i})=
O(1) e^{-\Omega(i^2)}\,.
\end{equation}
By Lemma~\ref{lem-variance-continuity} and \eqref{eq-variance-truncation}, we get that uniformly in $u, v\in \mathbb V$
\begin{equation}
\label{eq-feb25}
\Var (\Delta_i(v) - \Delta_i(u)) \leq  O(1) \min \{ e^{-\Omega((i+1)^2)},  2^{i} |u-v| \} \,.
\end{equation}
Combined with Lemmas~\ref{Lem.concentration} and \ref{lem-ferniquecriterion}, this gives that
\begin{equation}\label{eq-boring-2}
\P(\max_{u\in \mathfrak C_{ i+ \lfloor 4\log_2 i\rfloor}} \max_{v: |v-u| \leq 4i^{-4} \cdot 2^{-i}} |\Delta_i(u) - \Delta_i(v)| \geq \lambda (i+1)^{-2}) \leq O(1)e^{-\Omega(1)\lambda^2 (i+1)^2} .
\end{equation}
In addition, by \eqref{eq-variance-truncation} and a  union bound, we get that
\begin{equation}\label{eq-boring-3}
\P(\max_{u\in \mathfrak C_{ i+ \lfloor 4\log_2 i \rfloor}} |\Delta_i(u)| \geq \lambda (i+1)^{-2}) \leq O(1)e^{-\Omega(1)\lambda^2 (i+1)^2}\,.
\end{equation}
Note that for any $j\geq 0$ one has 
$$ \max_{v\in \mathbb V} \max_{j\geq 0} |\tilde h_{2^{-j}}(v) - \eta_{2^{-j}}(v)| \leq \sum_{i\geq 0} ( \max_{u\in \mathfrak C_{ i+  \lfloor 4\log_2 i \rfloor}} \max_{v: |v-u| \leq 4i^{-4} \cdot 2^{-i}} |\Delta_i(u) - \Delta_i(v)| +\max_{u\in \mathfrak C_{ i+ \lfloor 4\log_2 i \rfloor}} |\Delta_i(u)|)\,.$$
Combined with \eqref{eq-boring-2} and \eqref{eq-boring-3}, this completes the proof of the lemma.
\end{proof}

Define \begin{equation}
\label{eq:WND_decomposition-stationary}
\hat h_\delta^{\tilde \delta}(v) = \sqrt{\pi}\int_{\mathbb R^2 \times (\delta^2, \tilde \delta^2)}p(s/2; v, w)W(dw, ds)\,.
\end{equation}
The 
process $\hat h_\delta^{\tilde \delta}$ has better invariance properties than the process
$\tilde h_\delta^{\tilde \delta}$ from 
\eqref{eq:WND_decomposition}.
By a direct computation we obtain 
that for all $\tilde \delta > \delta > 0$ and $v, w\in \mathbb V$,
\begin{equation}\label{eq-hat-h-continuity}
\Var (\hat h_\delta^{\tilde \delta}(v) - \hat h_\delta^{\tilde \delta}(w)) \leq
 \int_{\delta^2}^{\infty}\frac{1 - e^{-\frac{|v - w|^2}{2s}}}{s}ds \leq \int_{\delta^2}^{\infty}\frac{|v - w|^2}{2s^2}ds \leq  \frac{|v - w|^2}{\delta^2}\,. 
\end{equation}

For $\xi >0$, write $\mathbb V^\xi = \{v\in \mathbb V: |v-\partial \mathbb V| \geq \xi\}$.
\begin{lemma}\label{lem-hat-h-eta}
For any $\xi>0$, there exists 
a constant $C = C(\xi)>0$ so that for all $\lambda>0$
\begin{equation}\label{eq-coupling-hat-h-eta-1}
\P( \max_{v\in \mathbb V^\xi} \max_{j\geq 0}
|\hat h_{2^{-j}}^{1}
(v) - \eta_{2^{-j}}(v)|\geq \lambda)  \leq Ce^{-C^{-1}\lambda^2} \,.
\end{equation}
\end{lemma}
\begin{proof}
The proof is very similar to that  of Lemma \ref{lem-tilde-h-eta}. Define $\Delta_i(v) = \hat h_{2^{-i}}^{2^{-i+1}}(v) - \eta_{2^{-i}}^{2^{-i+1}}(v)$ for $i\geq 1$ and write $\Delta_0(v) = \eta_1(v)$. Similarly to \eqref{eq-feb25}, we
obtain that uniformly in $i,u,v\in \mathbb V^\xi$,
$$\Var (\Delta_i(v)-\Delta_i(u))\leq O(1) \min(e^{-\Omega(i^2)}, 2^{i}|u-v|),$$
where the $O(1)$ and the $\Omega$ terms depend on $\xi$ only.
Thus, following the derivation as in  Lemma \ref{lem-tilde-h-eta}, we obtain  an analogue of \eqref{eq-feb25} and \eqref{eq-boring-2} in our setting, and then  conclude the proof of the current lemma.
\end{proof}

\begin{lemma}\label{lem-scaling-coupling}
For $0<\xi, \kappa_2\leq \kappa_1<1$, let $\mathbb V_1, \mathbb V_2 \subseteq 
\mathbb V^\xi$ be two boxes with side lengths $\kappa_1$ and $\kappa_2$ respectively. Let $\theta:\mathbb V_1\to \mathbb V_2$ be such that
$\theta v = av+b$ for $a=\kappa_2/\kappa_1$ and some $b \in \mathbb R$ so that $\theta$ maps $\mathbb V_1$ onto $\mathbb V_2$.
Then, there exists a coupling of $\zeta^{(1)} = \{\zeta_{\delta}^{(1)}(v): v\in \mathbb V_1, 0 < \delta \le 1 \}$ and $\zeta^{(2)} = \{\zeta_{a\delta}^{(2)}(v): v\in \mathbb V_2, 0 < \delta \le 1 \}$
such that the following hold.

\noindent (1) The marginal laws of $\zeta^{(1)}$ and $\zeta^{(2)}$ are respectively the same as $\{\eta_{\delta}(v): v\in \mathbb V_1, 0 < \delta \le 1 \}$ and $\{\eta_{a \delta}(v): v\in \mathbb V_2, 0 < \delta \le 1 \}$.

\noindent (2) There exists $C = C(\xi, \kappa_1, \kappa_2)>0$ such that
$$\P(\max_{v\in \mathbb V_1} \max_{j\geq 0} |\zeta_{2^{-j}}^{(1)}(v) - 
\zeta_{a 2^{-j}}^{(2)}(\theta v)| \geq \lambda) \leq C e^{-C^{-1}\lambda^2}\,.$$
\end{lemma}
\begin{proof}
By \eqref{eq-hat-h-continuity} we see that $\Var(\hat h_{a}^1(u) - \hat h_a^1(v)) = O(|u-v|)$  for all $u, v\in \mathbb V^\xi$ where the $O(1)$ depends only on $(\xi, a)$.  In addition, by a straightforward computation we get that $\Var (\hat h_a^1 (u)) = O(1)$.
Therefore, Lemmas~\ref{Lem.concentration} and \ref{lem-ferniquecriterion} imply that 
$$\P(\max_{u\in \mathbb V^\xi} |\hat h_a^1(u)| \geq \lambda) \leq Ce^{-C^{-1}\lambda^2}\,,$$
where again $C$ is a positive constant depending on $(\xi, \kappa_1, \kappa_2)$.
Combined with \eqref{eq-coupling-hat-h-eta-1}, this gives that
\begin{equation}\label{eq-coupling-hat-h-eta-2}
\P( \max_{v\in \mathbb V_1} \max_{j\geq 0} 
|\hat h_{ a 2^{-j}}^{a}(\theta v) - \eta_{a2^{-j}}(\theta v)|\geq \lambda)  \leq Ce^{-C^{-1}\lambda^2}\,.
\end{equation}
By the translation invariance and scaling invariance property of the
$\hat h$-process we see that 
$$\{\hat h_{2^{-j}}^1(v): v\in \mathbb V_1, j\geq 0\}\mbox{ has the same law as } \{\hat h_{ a 2^{-j}}^a (\theta v): v\in \mathbb V_1, j\geq 0\}\,.$$
Therefore, we can construct a coupling of $((\hat h^{(1)}, \zeta^{(1)}), (\hat h^{(2)}, \zeta^{(2)}))$ such that 
\begin{itemize}
\item $(\hat h^{(1)})_{2^{-j}}^{1}(v) =  (\hat h^{(2)})_{ a 2^{-j}}^a (\theta v)$ for all $v\in \mathbb V_1, j\geq 0$;
\item  for $i \in\{1, 2\}$ the pair $(\hat h^{(i)}, \zeta^{(i)})$ is identically distributed as the pair $(\hat h, \eta)$.
\end{itemize}
Combined with \eqref{eq-coupling-hat-h-eta-1} (noting that
$\mathbb V_1 \subseteq \mathbb V^\xi$) and \eqref{eq-coupling-hat-h-eta-2}, this completes the proof of the lemma.
\end{proof}

\subsection{Liouville quantum gravity} \label{sec:LQG}

For any $\gamma < 2$, $M_\gamma$ is defined in  \cite{DS11} as the almost sure weak limit of the sequence of measures $M^\circ_{\gamma, n}$ given by
\begin{equation}\label{eq-limit-LQG}
M^\circ_{\gamma, n} = e^{\gamma h_{2^{-n}}(z)}2^{-n\gamma^2/2}\mathcal L_2(dz)\,,
\end{equation}
where $\mathcal L_2$ is the Lebesgue measure on $\mathbb R^2$. 
The LQG measure is by now well understood (see e.g., \cite{Kahane85, DS11, RV11, RV14, Shamov16, Berestycki17}), and in particular one has 
the existence of the limit in \eqref{eq-limit-LQG}, the uniqueness in law for the limiting measure via different approximation schemes, as well as a KPZ correspondence
through a uniformization of the random lattice seen as a Riemann surface.
In particular, it follows from martingale convergence that the sequence
\begin{equation}
\label{eq-03172018-a}
e^{\gamma \tilde h_{2^{-n}}(z)}2^{-n\gamma^2/2}\mathcal L_2(dz)
\end{equation}
almost surely weakly 
converges  to a Gaussian Multiplicative Chaos, 
and then it follows e.g. from \cite{DS11,Shamov16} that
the limit is precisely 
$M_\gamma$. 
This approximation of the LQG measure via the 
white noise decomposition will be particularly useful to us.

 Of particular relevance to the present article is  the following boundedness result on the positive and negative moments of the 
LQG measure, proved in \cite{Kahane85, RV10} 
(see also \cite[Theorems 2.11, 2.12]{RV14}). 
\begin{lemma}\label{lem-LQG-moment}
For any $0<p<4/\gamma^2$, we have $\E (M_\gamma(\mathbb V))^p<\infty$. For any non-empty Euclidean ball $A\subseteq \mathbb V$, we have $\E (M_\gamma(A))^p < \infty$ for all $p<0$.
\end{lemma}
We will need a slightly stronger version of Lemma~\ref{lem-LQG-moment}. Let $B\subseteq \mathbb V$ be a square or a Euclidean ball of diameter $\xi>0$, and  define
\begin{equation}\label{eq-def-tilde-M}
\begin{split}
\tilde M_{\gamma, \delta}(B) = \lim_{n\to \infty} \int_B e^{\gamma \tilde h_{2^{-n}}^{\delta}(z)}e^{-\frac{\gamma^2}{2}\Var (\tilde h_{2^{-n}}^{\delta}(z)) }\mathcal L_2(dz)\,,\\
\tilde M_{\gamma, \delta, \eta}(B) = \lim_{n\to \infty} \int_B e^{\gamma \tilde \eta_{2^{-n}}^{\delta}(z)}e^{-\frac{\gamma^2}{2}\Var (\tilde \eta_{2^{-n}}^{\delta}(z)) }\mathcal L_2(dz)\,,
\end{split}
\end{equation}
where the existence of the almost sure limit follows from the fact that  $\tilde M_{\gamma, \delta}(B)$ (respectively $\tilde M_{\gamma, \delta, \eta}(B )$) forms a sequence of martingales (c.f. \cite{RV14}). By a straightforward adaption of
the proof of Lemma~\ref{lem-LQG-moment}, we obtain that 
\begin{align}
\E (\xi^{-2}\tilde M_{\gamma, \delta}(B))^p \leq C_{\gamma, p} \mbox{ for all } 0<p<4/\gamma^2 \mbox{ and } \delta \leq \xi\,, \label{eq-LQG-positive-moment}\\
\E (\xi^{-2}\tilde M_{\gamma, \delta}(B))^p \leq C_{\gamma, p} \mbox{ for all } p<0 \mbox{ and } \delta \leq \xi\,,\label{eq-LQG-negative-moment}
\end{align}
where $C_{\gamma, p}$ is a positive constant depending only on $(\gamma, p)$.
(Tail estimates for $\tilde M_{\gamma,\delta,\eta}$ will be provided in the course
of the proof of Proposition \ref{prop-approximate-LGD} below.)

\subsection{Liouville Brownian motion}

To precisely define the 
Liouville Brownian Motion, we revisit \eqref{eq-def-PCAF}. We define the positive continuous additive functional (PCAF) with respect to $M_\gamma$
as
\begin{equation}\label{eq-def-PCAF-precise}
F (t) := \lim_{n\to \infty} \int_0^t e^{\gamma \tilde h_{2^{-n}}(X_s) - \frac {\gamma^2}2 \Var (\tilde h_{2^{-n}} (X_s))} d s, 
\end{equation}
where the limit exists almost surely due to \cite{GRV13, B14}. It is not hard
to check, using the a.s. convergence discussed in Section~\ref{sec:LQG},
that the limit in \eqref{eq-def-PCAF-precise} 
does not depend on whether circle averages or white noise approximations are used.
With $F(t)$ well-defined, the LBM is defined as $Y_t:=X_{F^{-1}(t)}$, and
the LHK  $\pe_t^\gamma (x,y)$ is then constructed in \cite{GRV14} as the density of the Liouville semigroup with respect to $M_\gamma$ as in \eqref{eq-def-LHK}. The LBM
and its heat kernel capture geometric information encoded in $M_\gamma$; for example,
the KPZ formula was derived from the Liouville heat kernel in 
\cite{FM09, BGRV14}. 

We will need the following lemma, which is essentially proved in \cite{MRVZ14}. We remark that in \cite{MRVZ14} the authors work with GFF on a torus but their proofs adapt to our case with minimal change and we omit further details on such adaption.
See also \cite{AK} for related estimates.
\begin{lemma}\label{lem-Liouville-hitting-probability}
For any constants $\alpha_1, \alpha_2>0$ there exists a constant
$\alpha_3=\alpha_3(\alpha_1,\alpha_2,\gamma)>0$ 
and random variables $c_1, c_2,c_3>0$ measurable with respect to the GFF, so that for all $t>0$,
$$\pe^\gamma_{t}(u, v) \leq c_3 (t^{-2\alpha_3} + 1) P_u(|Y_{t - t^{\alpha_3}} - v| < t^{\alpha_1})  + \frac{c_1}{t^{2\alpha_3 + 2}} e^{-c_2 t^{\alpha_2}} \mbox{ for all } |u-v|\leq t^{-\alpha_1}\,.$$
\end{lemma}
\begin{proof}
With quantifiers as in the statement of the lemma, we have from
\cite[Theorem 4.2]{MRVZ14} that
\begin{equation}\label{eq-MRVZ-thm4.2}
\pe^\gamma_{t^{\alpha_3}}(x, y) \leq \frac{c_1}{t^{2\alpha_3 + 2}} e^{-c_2 t^{\alpha_2}} \mbox{ for all } |x-y|\geq t^{\alpha_1}\,.
\end{equation}
In addition, by \cite[Lemma 4.3]{MRVZ14}, 
$$\sup_{x,y\in \mathbb V} \pe^\gamma_{t}(x, y) \leq c_3 (t^{-2} + 1)\,. $$
The lemma follows from the last two displays and the decomposition
\begin{equation*}\pe^\gamma_{t}(u, v) = \int_{B(v, t^{\alpha_1})} \pe^\gamma_{t- t^{\alpha_3} }(u, x) \pe^\gamma_{t^{\alpha_3}}(x, v) M_\gamma(dx)  +\int_{\mathbb V \setminus B(v, t^{\alpha_1})} \pe^\gamma_{t-t^{\alpha_3}}(u, x) \pe^\gamma_{t^{\alpha_3}}(x, v) M_\gamma(dx)\,. \qedhere
\end{equation*}
\end{proof}

\subsection{Non-optimal bounds on the Liouville graph distance}
\label{subsec-nonopt}
The following are
non-optimal bounds on the Liouville graph distance. 
Our main goal in recording the following lemma is to illustrate that the distance exponent is non-trivial (i.e., strictly between 0 and 2).
\begin{lemma}\label{lem-obvious-bounds}
For $0<\gamma<2$ there exists $c>0$ depending only on $\gamma$ such that for all fixed $u, v\in \mathbb V$ we have   $c - o(1)< \frac{\E \log D_{\gamma, \delta}(u, v)}{\log \delta^{-1}} \leq \frac{4[(1+\gamma^2/4) - \sqrt{1+\gamma^4/16}]}{\gamma^2} +o(1)$ where the $o(1)$ term tends to 0 as $\delta$. In addition, $D_{\gamma, \delta}(u, v) \geq \delta^{-c}$ with high probability.
\end{lemma}
\begin{proof}
The upper bound on $D_{\gamma, \delta}(u, v)$ follows from
the KPZ relation derived in \cite[Proposition 1.6]{DS11}, which is used to bound the number of Euclidean balls of LQG measure at most $\delta^2$  required in order to  cover the line segment joining $u$ and $v$ (that is, set $X$ as the line segment joining $u$ and $v$ in \cite[Equation (5)]{DS11}, and adjust $\delta$ to $\delta^2$). 

To prove the lower bound,
it suffices to show that, for some constant $c=c(\gamma)>0$,  $D_{\gamma, \delta}(u, v) \geq \delta^{-c}$ with high probability.  To this end, fix $c=c(\gamma)$.  Let $k_\delta$ be the smallest integer so that $2^{-k_\delta} \leq \delta^{c}$ and let $\mathfrak C_{k_\delta}$ be defined as in \eqref{eq-def-mathfrak-C}. By \eqref{eq-max-white-noise-process}, we have that with high probability,  
\begin{equation}\label{eq-assump-max-white-noise}
\max_{v\in \mathbb V} \tilde h_{2^{-k_\delta}}(v) \leq 3 k_\delta\,.
\end{equation}
From \eqref{eq-LQG-negative-moment} and a union bound
we have that with high probability,
$$\tilde M_{\gamma, 2^{-k_\delta}}(B(v, 2^{-k_\delta})) \geq 2^{-2.5 k_\delta},
\mbox{for all $v\in \mathfrak C_{k_\delta}$},$$
where $\tilde M_{\gamma, 2^{-k_\delta}}$ is as in \eqref{eq-def-tilde-M}.
Combined with \eqref{eq-assump-max-white-noise}, we see that if we choose $c$ small enough we have that $M_{\gamma}(B(v, 2^{-k_\delta})) \geq \delta^2$ for all $v\in \mathfrak C_{k_\delta}$. This implies that any Euclidean ball with LQG measure at most $\delta^2$ has radius at most $2^{-k_\delta+2}$. This implies the claimed
lower bound on the Liouville graph distance.
\end{proof}

\section{Liouville graph distance: approximation and concentration}
\label{sec:LGD}
In this section, we introduce an approximation for the Liouville graph distance,
which will play a key role throughout the paper. The key technical advantage of the approximate Liouville graph distance is on a version of ``separation of randomness'', as codified in Lemma~\ref{lem-partition-independence}.

\subsection{Liouville graph distance via approximate Liouville Quantum Gravity}
\label{subsec-appgraphdist}
For each box $B$ of side length $s_B=\epsilon>0$ and center $c_B=v$, we define the approximate LQG to be 
\begin{equation}\label{eq-def-approximate-LQG}
M_{\gamma, \epsilon}(B) = \epsilon^2 e^{\gamma \eta_{\epsilon}(v) - \frac{\gamma^2}{2} \Var (\eta_{\epsilon}(v))}\,,
\end{equation}
compare with \eqref{eq-limit-LQG} and \eqref{eq-03172018-a}; the main point in 
\eqref{eq-def-approximate-LQG} is that one only considers the 
value of $\eta_\epsilon$ at the center of $B$.  Note also that 
$M_{\gamma,\epsilon}$ does not define a measure, due to the
lack of additivity. Fixing $\delta>0$,
we introduce a random $\delta$-partition of $\mathbb V$ as in the following
iterative procedure.
Call a  box (which may be closed, open, or neither closed or open)
that has not been partitioned yet a cell. 
Whenever $M_{\gamma, s_B}(B) \geq \delta^2$ for a cell $B$,
diadically partition $B$ into four sub-boxes. The iterative procedure halts
when all  cells $B$ satisfy $M_{\gamma, s_B}(B) < \delta^2$. We denote by
 $\mathcal V_\delta$ the final collection of cells obtained in this procedure.
Note that closures of cells may intersect only along their boundary.
We view $\mathcal V_\delta$ as a graph, with vertices consisting of the cells 
in $\mathcal V_\delta$ and edges between
cells such that their closures have intersection with non-empty relative
interior (i.e., a nontrivial line segment).
For each $v\in \mathbb V$, we denote by $\mathsf{C}_{v, \delta}$ the unique cell in $\mathcal V_\delta$ which contains $v$. For two distinct $u, v\in \mathbb V$ define the approximate Liouville graph distance $D'_{\gamma, \delta}$ to be the graph distance between $\mathsf C_{v, \delta}$ and $\mathsf C_{u, \delta}$ in $\mathcal V_\delta$. In addition, we denote by $s_{v, \delta}$ the side length of $\mathsf C_{v, \delta}$. Finally, recall the definitions of events of high probability and of $\iota$-high probability, see Section \ref{sec:notation}. The following proposition justifies our terminology of approximate LGD. For a fixed $\xi>0$, denote $\mathbb V^\xi = \{v\in \mathbb V: |v-\partial \mathbb V| \geq \xi\}$. We say 
that $(A_\delta, B_\delta) \subseteq \mathbb V^\xi \times \mathbb V^\xi$ is a sequence of $\xi$-admissible pairs if 
\begin{itemize}
\item $A_\delta$ (respectively $B_\delta$) is a single point, or a connected set of diameter at least $\delta^\xi$.
\item The distance between $A_\delta$ and $B_\delta$ is at least 
$\xi$ for all $\delta$.
\end{itemize}
The following lemma, whose proof is postponed, gives an a-priori, coarse bound on
the cells in $\mathcal V_\delta$.
\begin{lemma}\label{lem-partition-minimal-cell}
For any $\gamma\in (0, 2)$, there exist constants $C_{\mathrm{mc}}, C_{\mathrm{Mc}}>0$ (depending only on $\gamma$) such that with high probability, each cell $\mathsf C_{v,\delta}\in
\mathcal V_\delta$ has side length $ \delta^{C_{\mathrm{mc}}} \leq s_{v,\delta}\leq \delta^{C_{\mathrm{Mc}}}$.  
\end{lemma}
The subscript $\mathrm{mc}$ in $C_{\mathrm{mc}}$ stands  for ``minimal cell'',  and $\mathrm{Mc}$ stands for ``maximal cell''. The values of $C_{\mathrm{mc}}$ and $C_{\mathrm{Mc}}$ are kept fixed  throughout the paper.
A first approximation step for the LGD is contained in the next proposition.
\begin{prop}\label{prop-approximate-LGD}
Fix $0<\xi< C_{\mathrm{Mc}}/3$. Then, there
exists a constant $c=c(\gamma, \xi)$ so that for any sequence of $\xi$-admissible pairs  $(A_\delta, B_\delta )$, we have with $c$-high probability
$$ \min_{x\in A_\delta, y\in B_\delta}D'_{\gamma, \delta}(x, y) \cdot e^{-(\log \delta^{-1})^{0.9}}\leq\min_{x\in A_\delta, y\in B_\delta}D_{\gamma, \delta}(x, y) \leq  \min_{x\in A_\delta, y\in B_\delta}D'_{\gamma, \delta}(x, y) \cdot e^{(\log \delta^{-1})^{0.9}}\,.$$
\end{prop}
The proof of  Proposition~\ref{prop-approximate-LGD} follows roughly  the following
outline.
\begin{enumerate}
\item In order to get an upper bound on the LGD, we  take the geodesic in $D'_{\gamma, \delta}$ and construct an efficient covering of this geodesic by Euclidean balls with bounded LQG measure. 
\item In order to get a lower bound on  LGD, we  show that any  path achieving the
 LGD will have to place at least one Euclidean ball in each cell of a path which is candidate for  $D'_{\gamma, \delta}$.
\end{enumerate}
Item 2 is easier to achieve, since we can apply a more or less straightforward union bound (essentially due to the fact that all negative moments exist for LQG measure). In order to prove (the more challenging) Item 1 (as well as later
showing the lower bound on the Liouville heat kernel), it would be ideal if in each cell of $\mathcal V_\delta$, the ``fine field'' within that cell (roughly speaking the integration over white noise within that cell) were almost independent of $\mathcal V_\delta$. While this property holds for a typical cell, it unfortunately cannot hold uniformly for all cells, for the reason that occasionally some cell will be neighboring to cells that are of much smaller side lengths (this, roughly speaking, is due to the fact that LQG measure only has finite positive moment up to a fixed, $\gamma$ dependent, order). In order to address this issue, we employ a technique influenced by percolation theory.

Some remark is in order concerning
 the definition of $\xi$-admissible pairs. The somewhat strange condition there is that if $A_\delta$ (or $B_\delta$) is not a single vertex, then it has to be a connected set  that is moderately large. This assumption is related to the regularity of the random partition $\mathcal V_\delta$ --- it is possible (though typically the case) that in some places, the random partition is highly irregular but yet these locations serve as endpoints for the geodesic between $A_\delta$ and $B_\delta$ in $D'_{\gamma, \delta}$. The high irregularity will prevent us
 from building efficient path in $D_{\gamma, \delta}$. Under our admissibility assumption, it becomes tractable (via a percolation-type  argument) since 
\begin{itemize}
\item If $A_\delta$ is a single vertex, then with high probability it has to be somewhat regular around $A_\delta$;
\item If $A_\delta$ is a connected set of moderately large diameter, then when it is irregular around $u\in A_\delta$, there exists a regular $u'\in A_\delta$ which is close to $u$.
\end{itemize}

Before providing the proof of Proposition  \ref{prop-approximate-LGD}, we prove a few preparatory lemmas. 
We begin with the proof of Lemma \ref{lem-partition-minimal-cell}.
\begin{proof}[Proof of Lemma \ref{lem-partition-minimal-cell}]
For $\epsilon>0$ with $\log_2 \epsilon^{-1} \in \mathbb Z$,  we have $|\mathfrak C_{\log_2 \epsilon^{-1}}| = \epsilon^{-2}$ (recall \eqref{eq-def-mathfrak-C}). 
Fix $\beta\in (\gamma,1+\gamma^2/4)$, noting that the last interval is non empty if $\gamma\in (0,2)$. A straightforward union bound implies that 
\begin{equation}
\label{eq-010518a}
\P(\max_{v\in \mathfrak C_{\log_2 \epsilon^{-1}}} \eta_\epsilon(v) \geq 
\tfrac {\beta}{\gamma} \log  \epsilon^{-2} ) \leq C \epsilon^{\frac{2\beta^2}{\gamma^2}-2}\leq \epsilon^c\,,
 \end{equation}
for some $c=c(\beta)>0$. 
On the complement of the event in \eqref{eq-010518a}, we have, using Lemma \ref{lem-continuity-h-eta}, that, with high probability,
 for any box $B$ with side $\epsilon$ centered at $\mathfrak C_{\log_2 \epsilon^{-1}}$,
 we have that
$M_{\gamma,\epsilon}(B)\leq \epsilon^{2(1+\gamma^2/4-\beta)}\leq \epsilon^{c'}$ for some $c'=c'(\beta)>0$. The bound on the side length for the maximal cell follows from a similar (simple) computation, and we omit further details.
\end{proof}
We note that an argument similar to that employed in  the proof of
Lemma~\ref{lem-partition-minimal-cell} shows that the tail of the distribution of
$\log (S_\delta)/\log \delta$ decays at least exponentially, where
 $S_\delta$ is the side length  of the minimal cell in $\mathcal V_\delta$. This implies that for any $u,v\in\mathbb V$,
\begin{equation}\label{eq-very-crude-prime}
\E  (\frac{\log D'_{\gamma, \delta}(u, v)}{\log \delta^{-1}} )^2 = O_\gamma(1)\,.
\end{equation} In addition, a simple adaption of the argument in \cite[Proposition 1.6]{DS11} (see also \cite[Proposition 6.2]{DG16}) gives that 
\begin{equation}
\E  \Big(\frac{\log D_{\gamma, \delta}(u, v)}{\log \delta^{-1}} \Big)^2 = O_\gamma(1) \label{eq-very-crude}
\end{equation}
(we remark that these are extremely crude bounds). Thus, combined with (the yet unproven) Proposition~\ref{prop-approximate-LGD}, we obtain the following corollary.
\begin{cor}\label{cor-approximate-LGD-expectation}
 For any $u, v\in \mathbb V$, we have that $\left| \E \frac{\log D_{\gamma, \delta}(u, v)}{\log \delta^{-1}}  - \E \frac{\log D'_{\gamma, \delta}(u, v)}{\log \delta^{-1}} \right| \le e^{-(\log \delta^{-1})^{0.9}}$.
\end{cor}

For $\alpha>0$, we define 
 \begin{eqnarray} \label{eq-def-E-delta-alpha}
&&\mathcal E_{\delta, \alpha} 
 : =\\
&&
\!\!\!\!\!\{ \delta^{C_{\mathrm{mc}}} \leq s_{\mathsf C} \leq \delta^{C_{\mathrm{Mc}}}
\mbox{\rm \ for all cells in $\mathcal V_\delta$}\}
\cap \cap_{m,j;x,y} \{ |\eta_{2^{-m}}(x) - \eta_{2^{-m-j}}(y)| \leq \alpha \sqrt{\log \delta^{-1}} \log \log \delta^{-1} \},\nonumber
\end{eqnarray}
where the last intersection is taken over $m,j,x,y$ such that $1 \le 2^m \le \delta^{-  C_{\mathrm{mc}} }$, $1 \leq  2^j  \leq (\alpha \log \delta^{-1})^2$, and $|x - y| \le 2^{-m+3}$.

\begin{lemma}\label{lem-neighboring-cell}
There exists $\alpha_0>0$ such that for all $\alpha > \alpha_0$, $\mathcal E_{\delta, \alpha}$ occurs
 with high probability.
\end{lemma}
\begin{proof}
Denote by $m_0 =  \lfloor C_{\mathrm{mc}} \log_2 \delta^{-1} \rfloor$, $j_0 = \lfloor 2 \log_2 (\alpha \log \delta^{-1}) \rfloor$. Denote by $\tilde x$ the center of the dyadic box of side length $2^{-m}$ containing $x$, and $\tilde y$  the center of the dyadic box of side length $2^{-m-j}$ containing $y$. By the triangle inequality,
 \begin{eqnarray*}
|\eta_{2^{-m}}(x) - \eta_{2^{-m-j}}(y)|  
 & \le &
|\eta_{2^{-m}}(x)  -  \eta_{2^{-m}}(\tilde x)| +  |\eta_{2^{-m}}(\tilde x)  - \eta_{2^{-m}}(\tilde y)| 
 \\ & &
+  |\eta_{2^{-m-j}}(\tilde y) - \eta_{2^{-m-j}}(y)| +  |\eta_{2^{-m}}(\tilde y) - \eta_{2^{-m-j}}(\tilde y)|.
 \end{eqnarray*}
Next, we will bound the four terms on the right hand side above.

For the first three terms, by Lemma~\ref{lem-continuity-h-eta} and a union bound, there exists $\alpha>0$ such that  with high probability
 $$ 
\cap_{i=1}^{2 m_0} \cap_{x\in \mathfrak C_i} \{ \max_{y:|x-y|\leq 11 \times 2^{-i} } |\eta_{2^{-i}}(x) - \eta_{2^{-i}}(y)| \leq \frac{\alpha}{10} \sqrt{\log \delta^{-1}} \}\,,
 $$ 
where $i$ is set as $m$ for the first two terms (note $|\tilde x - \tilde y| \le 11 \times 2^{-m}$ if $|x-y| \le 2^{-m+3}$) and is set as $m+j$ for the third term (note $m+j \le 2m_0$). 

For the fourth term, adjusting the value of $\alpha$ if needed, we obtain from a union bound 
over the choice of $j$ and $y \in \mathfrak C_{m + j}$ that with high probability
 $$
\cap_{m=1}^{m_0}  \cap_{j=0}^{j_0}  \cap_{y \in \mathfrak C_{m+j}}  \{ |\eta_{ 2^{-m-j }}(y) - \eta_{2^{-m}}(y)| \leq \frac{\alpha}{10} \sqrt{\log \delta^{-1} }\log\log \delta^{-1}\}\,,
 $$

Collecting the above results, we conclude that \eqref{eq-def-E-delta-alpha} holds with high probability. This, combined with Lemma~\ref{lem-partition-minimal-cell}, completes the proof.
 \end{proof}

The next lemma, whose proof is deferred,
compares the approximate LGD with two different parameters. 
\begin{lemma}\label{lem-approximate-LGD-two-deltas}
Fix $0<\xi< C_{\mathrm{Mc}}/3$ where $C_{\mathrm{Mc}}$ is specified in Lemma~\ref{lem-partition-minimal-cell}. For 
any sequence of $\xi$-admissible pairs  $(A_\delta, B_\delta )$  and any function $\delta' = \delta' (\delta) < \delta$,  it holds 
with high probability that
\begin{equation}
\label{eq-280318}
 \min_{u\in A_\delta, v\in B_\delta}D'_{\gamma, \delta'}(u, v) \leq   \min_{u\in A_\delta, v\in B_\delta} D'_{\gamma, \delta}(u, v) (\delta/\delta')^3 e^{(\log \delta^{-1})^{0.8}}\;.
\end{equation}
\end{lemma}

We remark that from the definition, we have the following converse to \eqref{eq-280318}:
\begin{equation}
\label{eq-280318b}
D'_{\gamma, \delta'}(u, v)  \geq D'_{\gamma, \delta}(u, v).
\end{equation}
In the next definition we formulate ingredients that will be useful in the proofs of Lemma~\ref{lem-approximate-LGD-two-deltas} and Proposition~\ref{prop-approximate-LGD}.
Recall that $s_B$ denotes the side length of a box $B$, see Section \ref{sec:notation}.
 \begin{defn} \label{def-E-delta-B-prime}  
Let $B$ be a box with side length $s_B$. Let $B_{\mathrm{large}}$ be a  box concentric with  $B$ and with  side length $2s_B$.

For a dyadic $\epsilon > 0$, denote by $\mathcal B(B, \epsilon)$ (respectively,  $\mathcal B_{\partial} (B, \epsilon)$) the collection of dyadic boxes 
in $\mathbb V$ with side lengths $\epsilon s_B$, which lie in $B_{\mathrm{large}}$ (respectively, whose closures intersect $\partial B$).

For $\delta>0$, let $\Psi_{B, \delta}$ be the number of cells  in $\mathcal V_\delta$ that are contained in $B$ and touch the boundary of $B$ (if $B$ is contained in a cell then we set $\Psi_{B, \delta} = 1$).  Let $\Phi_{B, \delta}$ be the minimal number of Euclidean balls with LQG measure at most $\delta^2$ that covers $\partial B$.

For $\lambda>0$, define the event $\mathcal E_{\delta, B, \epsilon, \lambda}$ (respectively, $\mathcal E'_{\delta, B, \epsilon, \lambda}$) to be  the following: there exists a sequence of neighboring boxes $B'_1, \ldots B'_d \subseteq B_{\mathrm{large}} \setminus B$ which encloses $B$ such that
 \begin{itemize}
\item $B'_i \in\mathcal B(B, \epsilon)$ for each $1\leq i\leq d$. 
\item $\Psi_{B'_i, \delta} \leq \lambda$ for each $1\leq i\leq d$ (respectively $\Phi_{B'_i, \delta} \leq \lambda$ for each $1\leq i\leq d$).
\end{itemize}
\end{defn}
\noindent
(In Definition \ref{def-E-delta-B-prime}, by 
 two boxes neighboring each other we mean  that the intersection of their
 closures contains a non-trivial line segment.
By a sequence enclosing  $B$ we mean that it separates $B$ from $\mathbb V \cap \partial B_{\mathrm{large}}$ in $\mathbb V$.)

As we have announced earlier, the proofs for Lemma~\ref{lem-approximate-LGD-two-deltas} and Proposition~\ref{prop-approximate-LGD} employ percolation-type  arguments. More precisely, for a dyadic box $B$, we consider $B' \in \mathcal B(B, \epsilon)$ and $\tilde B\in \mathcal B_{\partial} (B', t/\epsilon)$. If the LQG measures (or respectively approximate LQG) of all $\tilde B$'s are less than some value $\mu$, we call $B'$ an open (in the percolation sense) box. When $B$ is a cell, we will show that  each $B' \in \mathcal B(B, \epsilon)$ is open with large probability  by setting $\epsilon$ and $t$ appropriately, and that the openness of all $B' \in \mathcal B(B, \epsilon)$ are \emph{essentially} independent events. Therefore,
by standard arguments in percolation theory (in our case a straightforward union bound suffices), one can find an open path enclosing  $B$. The union of these enclosures along all cells  in the geodesic of $D'_{\gamma, \delta}$ then gives an approximately minimizing path. To compare $D'_{\gamma, \delta}$ with $D'_{\gamma, \delta'}$ and $D_{\gamma, \delta}$, we respectively set $\mu$ to be $(\delta')^2$ and $\delta^2 / 4$ (see Lemma~\ref{lem-approximate-LGD-two-deltas} 
and Proposition~\ref{prop-approximate-LGD}).  The key technical step 
for these arguments appears in Lemma~\ref{lem-percolation-Phi}. We remark that the proof of Lemma~\ref{lem-regularity} below follows the same type of analysis but is substantially more involved as we will need to keep track of the ratios between side lengths of neighboring boxes along the path we construct.

Recall the constants $\alpha_0$ and $C_{\mathrm{mc}}$, see Lemma~\ref{lem-neighboring-cell} and  Lemma~\ref{lem-partition-minimal-cell}.
\begin{lemma}\label{lem-percolation-Phi}
Let $\alpha > \max\{ \alpha_0, 4 C_{\mathrm{mc}} \}$.
For $0<\delta' = \delta' (\delta) \leq \delta$, 
let 
$$\epsilon = \min \{ 2^{-n} : 2^n \le 4 C_{\mathrm{mc}} \log \delta^{-1} \} \quad \mbox{\rm
 and} \ \lambda = (\delta/\delta')^3 e^{(\log \delta^{-1})^{0.7}}.$$ For each dyadic box $B$ with  side length $s=s_B = 2^{-m}$, $1\leq m \leq C_{\mathrm{mc}}\log_2 \delta^{-1}$, we have
\begin{equation}\label{eq-B-percolation-Phi}
\P(\{M_{\gamma, s} (B) \leq \delta^2\} \cap \mathcal E_{\delta, \alpha} \cap \mathcal E^c_{\delta', B, \epsilon, \lambda}) \leq \delta^{10 C_{\mathrm{mc}} + 10} \,.
\end{equation}
Furthermore,  for any fixed $x\in B_{\mathrm{large}}$ and any fixed $\iota>0$
\begin{equation}\label{eq-B-good-Phi}
\P(\{M_{\gamma, s} (B) \leq \delta^2\} \cap \mathcal E_{\delta, \alpha} \cap \{D'_{\gamma, \delta'}(x, \partial B_{\mathrm{large}}) > \delta^{-\iota}  (\delta/\delta')^3  \} ) \leq  \delta^{\iota/10}\,.
\end{equation}
\end{lemma}
\begin{proof}
Let $t$ be a dyadic such that $\log t^{-1} \ge (\log \delta^{-1})^{0.6}$, to be determined below. Write $K = 1/\epsilon$.

Suppose $\mathcal B (B, \epsilon) = \{ B'_i \}$. Write $\mathcal B'_i = \mathcal B_{\partial} (B'_i, t/\epsilon)$, then each box in $\mathcal B'_i$ has side length $t s$.  By 
\eqref{eq-def-E-delta-alpha}, we see that  on the event $\{M_{\gamma, s} (B) \leq \delta^2\} \cap \mathcal E_{\delta, \alpha}$, for all $\tilde B \in \mathcal B(B, t) \cup \mathcal B_{\partial} (B_{\mathrm{large}}, t)$ we have 
  $$
M_{\gamma,  t s}(\tilde B) \leq \delta^2 e^{2\gamma \alpha \sqrt{\log \delta^{-1}} \log \log \delta^{-1}} t^2 e^{\gamma \eta_{ts}^{\epsilon^2 s}(c_{\tilde B}) - \frac{\gamma^2}{2} \Var( \eta_{t s}^{\epsilon^2 s}(c_{\tilde B}))}\,.
 $$
(Recall that $c_{\tilde B}$ denotes the center of $\tilde B$.) 
By a union bound  and the fact that
 $\mathcal B'_i \subset \mathcal B(B, t) 
\cup \mathcal B_{\partial} (B_{\mathrm{large}}, t)$, we have that
 \begin{equation}
\label{eq-berlin1}
\P(\max_{\tilde B\in \mathcal B'_i}\eta_{t s}^{\epsilon^2 s}(c_{\tilde B}) \leq 1.5 \log t^{-1} ) \geq 1 -  t^{0.1}\,,
 \end{equation}
where we have used that $|\mathcal B'_i| \le 8 \epsilon / t \le 1/ t$. On the event in
\eqref{eq-berlin1},  we have
 \begin{equation}\label{eq-LQG-tilde-B-Phi}
M_{\gamma, t s}(\tilde B) \leq \delta^2 t^{0.8} .
 \end{equation}

To prove \eqref{eq-B-percolation-Phi}, we take $t = \epsilon 2^{-\lceil 0.9 \log_2 \lambda \rceil}$  (this
implies that $t^{-0.4} \ge  \lambda^{0.36} \ge \delta/\delta'$ and therefore
 $\delta^2 t^{0.8} \le \delta'^2$). Fix 
$p = t^{0.1}$ and  $\kappa = 2$. Combined  with \eqref{eq-LQG-tilde-B-Phi},
we see that there exist events $\mathcal E_{B'_i, \mathrm{open}}$ measurable with respect to $\{ \eta_{t s}^{\epsilon^2 s}(c_{\tilde B}): \tilde B\in \mathcal B'_i\}$ such that 
 \begin{equation}\label{eq-B'-i-open}
\left\{ \begin{array}{l}
\P(\mathcal E_{B'_i, \mathrm{open}}) \geq 1 - p,
 \\
\mbox{$\{ \mathcal E_{B'_i, \mathrm{open}}, i \in I \}$ and $\{ \mathcal E_{B'_{i'}, \mathrm{open}}, i' \in I'\}$ are independent if $|B'_i - B'_{i'}| \ge \kappa \epsilon s$ for all $i,i'$,}
 \end{array} \right.
\end{equation}
and 
 $$
( \{ M_{\gamma, s} (B) \leq \delta^2\} \cap \mathcal E_{\delta, \alpha}  \cap \mathcal E_{B'_i, \mathrm{open} })  \subseteq  \{ \Psi_{B_i, \delta'} \le \lambda\}.
 $$
(Note that the parameter $4 C_{\mathrm{mc}}$ in the choice of $\epsilon$
ensures that  $\epsilon^2 s \log \frac 1 {\epsilon^2 s} \le \epsilon s$, 
and that
$4 / t \le \lambda$.)
We are now ready to complete the proof of \eqref{eq-B-percolation-Phi} by finding an open enclosure of $B$, i.e., a sequence of neighboring boxes in $\mathcal B(B, \epsilon)$ enclosing  $B$ such  that $\mathcal E_{B', \mathrm{open}}$ occurs for each $B'$ in the sequence. To this end, we  employ a standard percolation argument. Suppose that such an enclosing path does not exist.
Then, by duality, there exists a sequence of boxes $B_{i_1}', \ldots, B_{i_\ell}'$  joining $\partial B$ and $\partial B_{\mathrm{large}}$ such that the closures of  consecutive boxes $B_{i_r}'$ and $B_{i_{r+1}}'$ intersect (possibly at a single point) for all $r$, and  none of $\mathcal E_{B'_{i_r}, \mathrm{open}}$'s occurs. Since $\ell \ge K$, there are at most $4(2K+1)\times 8^\ell$ such sequences for a fixed $\ell$. For each such sequence, one can find at least $\ell / (2 \kappa + 1)^2$ boxes $B_{i_r}'$'s with pairwise distance
 at least $\kappa \epsilon s$. Consequently, for each fixed such sequence, 
the events $\mathcal E_{B_{i_r}'}$'s are mutually independent. It follows that
 \begin{equation} \label{Eq.percolation-argument}
\P(\mbox{no open enclosure}) \leq \sum_{\ell = K}^\infty 4(2K+1)8^\ell p^{\ell / (2 \kappa + 1)} \le 9K (8 p^{\frac 1 {2 \kappa + 1}})^K ,
 \end{equation}
provided that $\kappa$ is fixed and $p = o(1)$. Substituting $p = t^{0.1}$ and $\kappa = 2$, we see that $9K (8 p^{\frac 1 {2 \kappa + 1}})^K \leq \delta^{10 C_{\mathrm{mc}} + 10}$. This completes the proof of \eqref{eq-B-percolation-Phi}, noting there is no open enclosure on the event $\{ M_{\gamma, s} (B) \leq \delta^2\} \cap \mathcal E_{\delta, \alpha}  \cap \mathcal E^c_{\delta', B, \epsilon, \lambda}$.

In order to prove \eqref{eq-B-good-Phi}, we take 
$t = \max \{ 2^{-r} : 2^{-r} \le \delta^{0.9 \cdot \iota} (\delta' / \delta)^3 \} $.
Denote by $\tilde B_i$, $0<i\leq  4 / t $ all the
(closed) boxes in $\mathcal B (B, t)$ 
which intersect the horizontal line passing through $x$. 
By the definition of $t$, 
one has $\delta^2 t^{0.8} \le (\delta')^2$.
We argue next in a similar way to the derivation of 
\eqref{eq-LQG-tilde-B-Phi}: on the event $\{M_{\gamma, s} (B) \leq \delta^2\} 
\cap \mathcal E_{\delta, \alpha}$, we have
 \begin{align*}
 \eta^{\epsilon^2 s}_{ts} (c_{{\tilde B}_i}) \leq 1.5 \log t^{-1}\mbox{ for all } 1\leq i \leq 4/t &\Rightarrow M_{\gamma, t s}(\tilde B_i) \leq \delta^2 t^{0.8} \mbox{ for all } 1\leq i \leq 4/t  \\
 &\Rightarrow D'_{\gamma, \delta'} (x, \partial B_{\mathrm{large}}) \leq \frac4t\leq
\delta^{-\iota} (\delta / \delta')^3.
 \end{align*} 
Since  $\eta^{\epsilon^2 s}_{ts} (v_{{\tilde B}_i})$ 
is a centered Gaussian variable with variance $\log (\epsilon^2/t)\leq |\log t|$,
we have  by a union bound that
$\P(\eta^{\epsilon^2 s}_{ts} (v_{{\tilde B}_i}) \leq 1.5 \log t^{-1}\mbox{ for all } i) \geq 1 - \frac  4 t t^{1.12}  \geq 1 - \delta^{\iota/10}$.
This completes the proof of  \eqref{eq-B-good-Phi}.
\end{proof}

\begin{proof}[Proof of Lemma~\ref{lem-approximate-LGD-two-deltas}]
Let $u \in A_\delta$, $v \in B_\delta$ be such that $\min_{x \in A_\delta, y\in B_\delta}D'_{\gamma, \delta}(x, y) = D'_{\gamma, \delta}(u, v) = : d$, and suppose $\mathsf C_1, \cdots, \mathsf C_d$ is a sequence of neighboring cells in $\mathcal V_\delta$ joining $u$ to $v$, with $u \in \mathsf C_1$.

In case  $A_\delta = \{u\}$,  let $S_u = \{\mathsf C\in \mathcal V_\delta: u\in \mathsf C_{\mathrm{large}}\}$, where we recall that
 $\mathsf C_{\mathrm{large}}$ is a box concentric with $\mathsf C$ 
of side length $2 s_{\mathsf C}$. We work on the event $\mathcal E_{\delta, \alpha}$,
which 
by Lemma~\ref{lem-neighboring-cell} is possible.
Choose $\iota = C_{\mathrm{Mc}}/3$. Applying \eqref{eq-B-good-Phi} of Lemma~\ref{lem-percolation-Phi} to all dyadic boxes containing $u$ with side length at least $\delta^{C_{\mathrm{mc}}}$ (so in total we apply \eqref{eq-B-good-Phi} $O(\log \delta^{-1})$ times), we see that with high probability we have 
 $$
 D'_{\gamma, \delta'}(u , \partial \mathsf C_{\mathrm{large}}) \leq \delta^{-\iota} (\delta/\delta')^3 \mbox{ for all } \mathsf C\in S_u\,.
 $$
Let $\mathbb C_{\mathrm{start}}$ be the collection of all cells in geodesics of $D'_{\gamma, \delta'}(u , \partial \mathsf C_{\mathrm{large}})$ for all $\mathsf C\in S_u$. In the case $A_\delta$ is a connected set of diameter at least $\delta^\xi$, let $\mathbb C_{\mathrm{start}} = \emptyset$. Similarly, we define $\mathbb C_{\mathrm{end}}$.

By \eqref{eq-B-percolation-Phi} of Lemma~\ref{lem-percolation-Phi} and a union bound, we see that with high probability $\mathcal E_{\delta', \mathsf C, \epsilon, \lambda}$ holds for each $\mathsf C \in \mathcal V_\delta$, where $\epsilon, \lambda$ are specified as in Lemma~\ref{lem-percolation-Phi}. In particular, in what follows we can assume that $\mathcal E_{\delta', \mathsf C_i, \epsilon, \lambda}$ holds for all $i$. 
Then, for each $i$, there exists a sequence, denoted
 $\mathbb C_i$, of neighboring cells in $\mathcal V_{\delta'}$ such that $|\mathbb C_i | \le \lambda/\epsilon^2$, $\mathbb C_i$ encloses $\mathsf C_i$, and each cell in $\mathbb C_i$ intersects with $\mathsf C_{i,\mathrm{large}} \setminus \mathsf C_i$.
 We claim
 that $(\cup_{i=1}^d \mathbb C_i) \cup \mathbb C_{\mathrm{start}} \cup \mathbb C_{\mathrm{end}}$ contains a crossing between $A_\delta$ and  $B_\delta$. This is
 justified as follows: let $i_1 = \max \{ i : \mathbb C_i \mbox{ encloses } u \}$, and define recursively $i_r = \max \{ i > i_{r-1} : \mathbb C_i \mbox{ intersects } \mathbb C_{i_{r-1}} \}$ till $r_0$ such that one can not define $i_{r_0 + 1}$, then $A_\delta$ is connected to $\mathbb C_{\mathrm{start}} \cup \mathbb C_{i_1}$ and respectively $B_\delta$ to $\mathbb C_{\mathrm{end}} \cup \mathbb C_{i_{r_0}}$. It follows that
 \begin{equation} \label{Eq.boundDprime}
\min_{u \in A_\delta, v\in B_\delta}D'_{\gamma, \delta'}(u, v) \le d \lambda / \epsilon^2 +  2 \delta^{-\iota} (\delta / \delta')^3 .
 \end{equation}
This completes the proof, noting that on $\mathcal E_{\delta, \alpha} $
 \begin{equation} \label{Eq.lowerboundforDprime}
\min_{x \in A_\delta, y\in B_\delta}D'_{\gamma, \delta}(x, y) \geq \xi / (2 \delta^{C_{\mathrm{Mc}}}) \geq \delta^{- 2 \iota}\,. \qedhere
 \end{equation}
\end{proof}

\begin{proof}[Proof of Proposition~\ref{prop-approximate-LGD}]
We begin with the upper bound. The proof resembles that of Lemma~\ref{lem-approximate-LGD-two-deltas}, and the key technical ingredient is an analogue of Lemma~\ref{lem-percolation-Phi}. Let $\epsilon = \max\{2^{-n}: 2^n \leq 4 C_{\mathrm{mc}} \log \delta^{-1}\}$ be as  in Lemma~\ref{lem-percolation-Phi}, and 
let $\lambda =e^{(\log \delta^{-1})^{0.7}}$. We will show that there exists an event $\tilde {\mathcal E}_{\delta, \alpha}$ which occurs with high probability  such that for each dyadic box $B$ with side length $s = 2^{-m}$, $1\leq m \leq C_{\mathrm{mc}}\log_2 \delta^{-1}$, 
\begin{equation}\label{eq-B-percolation-Psi}
\P(\{M_{\gamma, s} (B) \leq \delta^2\} \cap \tilde {\mathcal E}_{\delta, \alpha} \cap (\mathcal E'_{\delta, B, \epsilon, \lambda})^c)\leq \delta^{10 C_{\mathrm{mc}} + 10}\,,
\end{equation}
where $\mathcal E'_{\delta, B, \epsilon, \lambda}$ is as in Definition~\ref{def-E-delta-B-prime}. 
Furthermore,  we will show that for any fixed $x\in B_{\mathrm{large}}$ and any fixed $\iota>0$
\begin{equation}\label{eq-B-good-Psi}
\P(\{M_{\gamma, s} (B) \leq \delta^2\} \cap \tilde {\mathcal E}_{\delta, \alpha} \cap \{D_{\gamma, \delta}(x, \partial B_{\mathrm{large}}) > \lambda \delta^{-\iota} \} ) \leq  \delta^{\iota/10}\,.
\end{equation}
Provided with \eqref{eq-B-percolation-Psi} and \eqref{eq-B-good-Psi}, we can complete the proof for the upper bound following the same argument as in the 
proof of Lemma~\ref{lem-approximate-LGD-two-deltas}. Note that in the case here, 
 $$
\min_{u \in A_\delta, v\in B_\delta} D_{\gamma, \delta}(u, v) \le d \lambda / \epsilon^2  + 2 \delta^{-\iota} \lambda \le d e^{(\log \delta^{-1})^{0.8}} + \delta^{- 2 \iota} e^{(\log \delta^{-1})^{0.8}} \le d  e^{(\log \delta^{-1})^{0.9}},
 $$
where $d =  \min_{u \in A_\delta, v\in B_\delta} D'_{\gamma, \delta}(u, v)$ (compare with \eqref{eq-B-percolation-Phi}, \eqref{eq-B-good-Phi}, \eqref{Eq.boundDprime} and \eqref{Eq.lowerboundforDprime}). 
Thus, it remains to prove \eqref{eq-B-percolation-Psi} and \eqref{eq-B-good-Psi} (the proof resembles that of \eqref{eq-B-percolation-Phi} and \eqref{eq-B-good-Phi}). 

Let $t = \epsilon 2^{- \lceil \frac1 {\log 2} (\log \delta^{-1})^{0.6} \rceil} $, and with $\mathcal B (B, \epsilon)=\{ B'_i \}$, set
$\mathcal B'_i = \mathcal B_{\partial} (B'_i, t/\epsilon)$. Write $K = 1/ \epsilon$.
By Lemma~\ref{lem-tilde-h-eta}, we have that 
 \begin{equation} \label{eq-tilde-h-eta-assump}
\max_{j\geq 1}\max_{v\in \mathbb V} |\tilde h_{2^{-j}}(y) - \eta_{2^{-j}}(y)|  = O(\sqrt{\log \delta^{-1}})\,, \quad \mbox{\rm with high probability}\,.
\end{equation}
Let
 \begin{equation} \label{Eq.definition-tilde-E}
\tilde {\mathcal E}_{\delta, \alpha} : = \mbox{the intersection of $\mathcal E_{\delta, \alpha}$ from \eqref{eq-def-E-delta-alpha} and the event described in \eqref{eq-tilde-h-eta-assump}.}
 \end{equation}
Then, on $\{M_{\gamma, s} (B) \leq \delta^2 \} \cap \tilde {\mathcal E}_{\delta, \alpha}$ one has
 \begin{equation} \label{eq-M-tilde-B-bound}
M_\gamma (\tilde B) \le e^{2 \alpha \gamma \sqrt{\log \delta^{-1}} \log \log \delta^{-1}} \times 
\delta^2 s^{-2} \times \tilde M_{\gamma, \epsilon^2 s, \eta} (\tilde B)
  \end{equation}
for any $\tilde B \in \mathcal B (B, t) \cup \mathcal B_{\partial} (B_{\mathrm{large}}, t)$.
 By Fubini's Theorem, we have that $\E \tilde M_{\gamma, \epsilon^2 s, \eta}(\tilde B)  = (ts)^2$. Thus,
 \begin{equation} \label{Eq.LQG-tildeM}
\P (M_{\gamma, \epsilon^2 s, \eta} (\tilde B) > 4^{-1} s^2 e^{- 2 \alpha \gamma \sqrt{\log \delta^{-1}} \log \log \delta^{-1}}  ) \le \frac{ \E (\tilde M_{\gamma, \epsilon^2 s, \eta}(\tilde B)) }{4^{-1} s^2 e^{- 2 \alpha \gamma \sqrt{\log \delta^{-1}} \log \log \delta^{-1}} } \le t^{1.9}.
 \end{equation}
Consequently, with $\mathcal E_{B'_i, \mathrm{open}} $  defined by
 $$
\mathcal E_{B'_i, \mathrm{open}} : =  \{ M_{\gamma, \epsilon^2 s, \eta} (\tilde B) \le 4^{-1} s^2 e^{- 2 \alpha \gamma \sqrt{\log \delta^{-1}} \log \log \delta^{-1}} \mbox{ for all } \tilde B\in \mathcal B'_i \},
 $$
 and using that $|\mathcal B'_i|  \le t^{-1}$, we have that
 $$
\P( \mathcal E_{B'_i, \mathrm{open}}^c) \leq t^{1.9} |\mathcal B'_i| \le t^{0.9}  \,.
 $$
On the one hand,
 $$
\left\{ \begin{array}{l} \P(\mathcal E_{B'_i, \mathrm{open}}) \geq 1 -  t^{0.9}, \\ \mbox{$\{ \mathcal E_{B'_i, 
\mathrm{open}}, i \in I \}$ and $\{ \mathcal E_{B'_{i'}, i' \in I'} \}$ are independent if $|B'_i - B'_{i'}| \ge 2 \epsilon s$ for all $i$, $i'$}. \end{array} \right. 
 $$
On the other hand, consider the balls of radius radius $ts$
centered at the corners of boxes in $\mathcal B'_i$ that are on $\partial B'_i$. The collection of these $4 \epsilon / t$  balls covers $\partial B'_i$. Note that each such ball can be covered by at most $4$ boxes in $\mathcal B'_i$. Thus, each one has LQG measure at most $\delta^2$ if $\{M_{\gamma, s} (B) \leq \delta^2\} \cap \tilde {\mathcal E}_{\delta, \alpha}  \cap \mathcal E_{B'_i, \mathrm{open}}$ occurs, by the definition of $\mathcal E_{B'_i, \mathrm{open}}$ together with \eqref{eq-M-tilde-B-bound}. Therefore, we have that
 $$
(\{M_{\gamma, s} (B) \leq \delta^2\} \cap \tilde {\mathcal E}_{\delta, \alpha}  \cap \mathcal E_{B'_i, \mathrm{open}}) \subseteq \{ \Phi_{B'_i, \delta} \le \lambda  \}, 
 $$
where we use that $4 \epsilon/ t \le e^{(\log \delta^{-1})^{0.7}} = \lambda$. We can now apply the percolation argument as in the proof of Lemma~\ref{lem-percolation-Phi}, with parameters in \eqref{eq-B'-i-open} and \eqref{Eq.percolation-argument} being set as $p = t^{0.9}$ and $\kappa = 2$ here. Then, we obtain that 
 $$
\P (\{M_{\gamma, s} (B) \leq \delta^2\} \cap \tilde {\mathcal E}_{\delta, \alpha}  \cap (\mathcal E'_{\delta, B, \epsilon, \lambda})^c ) \le 9 K (8 p^{\frac 1 {2 \kappa + 1}})^K \leq  \delta^{10 C_{\mathrm{mc}} + 10},
 $$
completing the proof of \eqref{eq-B-percolation-Psi}.

To prove \eqref{eq-B-good-Psi}, we take $t = 2^{- \lceil \log_2 (\delta^{-\iota / 2} \lambda) \rceil} $. Denote by $\tilde B_i$, $0<i\leq 2 / t$ the (closed) boxes in $\mathcal B (B, t)$ that  intersect the horizontal line passing through $x$. 
Denote by $\{ \hat B_j,  j =  1, \ldots, \ell \}$ the collection of $\tilde B_i$'s together with their neighboring boxes in $\mathcal B (B, t) \cup \mathcal B_{\partial} (B_{\mathrm{large}}, t)$, where $\ell \le 6/t + 6$. Consider the balls centered at corners of some $\tilde B_i$ with radius $ts$. The collection of these $4/t+2$ balls covers a line segment from $x$ to $\partial B_{\mathrm{large}}$, and each ball is covered by at most 4 boxes in $\{ \hat B_j \}$. Consequently, on the event $\{M_{\gamma, s} (B) \leq \delta^2\} \cap \mathcal E_{\delta, \alpha}$, we have that $D_{\gamma, \delta} (x, \partial B_{\mathrm{large}}) > \delta^{-\iota} \lambda$ (note 
that $\delta^{-\iota} \lambda \ge (4/t + 2)$) implies that $M_{\gamma, \epsilon^2 s, \eta} (\hat B_j) > 4^{-1} s^2 e^{- 2 \alpha \gamma \sqrt{\log \delta^{-1}} \log \log \delta^{-1}}$ for some $j$, recalling \eqref{eq-M-tilde-B-bound}. This occurs with probability at most $(6/t + 6) t^{1.9}\leq\delta^{\iota / 10}$,
see \eqref{Eq.LQG-tildeM}. This completes the proof of  \eqref{eq-B-good-Psi}.

\medskip

Next, we turn to  the lower bound in Proposition \ref{prop-approximate-LGD}. 
Let $\delta' = \delta  e^{(\log \delta^{-1})^{0.8}}$. With this choice, events of high probability with respect to $\delta$ are also of high probability with respect to $\tilde \delta$, and vice versa. Therefore, we do not distinguish between those notions.
The key  to the proof is the claim that
with high probability, 
\begin{equation}\label{eq-Euclidean-Ball-covering}
\mbox{ every Euclidean ball with LQG-measure } \leq \delta^2 \mbox{ can be covered by 4 cells in } \mathcal V_{\delta'}\,.
\end{equation}
Provided with \eqref{eq-Euclidean-Ball-covering}, it is clear that with high probability we have that 
$$D'_{\gamma, \delta'}(u, v) \leq 4 D_{\gamma, \delta}(u, v) \mbox{ for all } u, v\in \mathbb V\,.$$
Combined with Lemma~\ref{lem-approximate-LGD-two-deltas}, it then yields the desired lower bound in the proposition.

It remains to prove \eqref{eq-Euclidean-Ball-covering}.
 Note that any Euclidean ball $R$ of radius $r$ can be covered by four closed dyadic boxes (which have non-empty pairwise intersection) of side length $s = 2 \min\{2^{-n}: 2^{-n} \geq r\}$. Suppose that $R$ cannot be covered by four 
cells in $\mathcal V_{\delta'}$, which
 means at least one of these four dyadic boxes $B$ satisfies that
 $M_{\gamma, s}(B) > \delta'^2$. Further, partition the box concentric with $B$ of
 side length $4s$ into $(4 \times 2^{10})^2$ squares of side length $s' = 2^{-10} s$.
Denote  the partition by $\mathcal S_B$. Then,
$R$ contains at least one square from $\mathcal S_B$. Therefore,
\eqref{eq-Euclidean-Ball-covering} would follow provided that with high probability, 
 \begin{equation}\label{eq-cell-LQG-compare}
 \mbox{ there exists no dyadic box } B \mbox{ with } M_{\gamma, s}(B) > \delta'^2 \mbox{ and } M_\gamma(S) \leq \delta^2 \mbox{ for some } S\in \mathcal S_B\,.
 \end{equation}
Let $\epsilon$ and $\alpha$ be as in Lemma~\ref{lem-percolation-Phi}, and recall that
$\epsilon = \inf\{ 2^{-n} : 2^n \le 4 C_{\mathrm{mc}} \log \delta^{-1} \}$. We will show that for a fixed box $B$ and a fixed  square $S \in \mathcal S_B$,
 \begin{equation} \label{Eq.LD-lowerbound-approx-LGD}
\P ( \tilde {\mathcal E}_{\delta, \alpha}, \tilde {\mathcal E}_{\delta', \alpha}, M_{\gamma, s}(B) > \delta'^2, M_\gamma(S) \leq \delta^2 ) \le e^{ - (2 C_{\mathrm{mc}} \log \delta^{-1})^2}.
 \end{equation}
Assuming this, one can check \eqref{eq-cell-LQG-compare}, noting that the event there
is not empty only for $s \ge (\delta')^{C_{\mathrm{mc}}}$.

Next, we are going to show \eqref{Eq.LD-lowerbound-approx-LGD}. We work on the high probability events $ \tilde {\mathcal E}_{\delta, \alpha}$ and $ \tilde {\mathcal E}_{\delta', \alpha}$  (see \eqref{Eq.definition-tilde-E} for the definition). We partition  $S \in \mathcal S_B$ into $K^2$ squares $\tilde S_1, \ldots, \tilde S_{K^2}$ of side length $\epsilon s'$ (recall that $K = 1/\epsilon$). Similarly to \eqref{eq-M-tilde-B-bound}, we have that for all $i$,
 $$
M_\gamma (\tilde S_i) \ge e^{- 2 \alpha \gamma \sqrt{\log \delta^{-1}} \log \log \delta^{-1}} \times (\delta')^2 s^{-2} \times \tilde M_{\gamma, \epsilon^2 s', \eta} (\tilde S_i), 
 $$
where $\tilde M_{\gamma, \epsilon^2 s', \eta}$ is defined as in \eqref{eq-def-tilde-M}. 
Since $\epsilon^2 s' \log \frac 1 {\epsilon^2 s'} < \epsilon s'$, one can find 
$K^2 / 4$  squares $\tilde S_{i_1}, \cdots, \tilde S_{i_{K^2/ 4}}$ such that $ \tilde M_{\gamma, \epsilon^2 s', \eta} (\tilde S_{i_j}) $'s are mutually independent. Then, $M_\gamma (S) \le \delta^2$ implies that 
 $$
\sum_{j=1}^{K^2/4} (\epsilon s')^{-2} \tilde M_{\gamma, \epsilon^2 s', \eta} (\tilde S_{i_j})  \le e^{2 \alpha \gamma \sqrt{\log \delta^{-1}} \log \log \delta^{-1}} \times ( \delta / \delta')^2  \times s^2 / (\epsilon s')^2 \le \beta^2 K^2 / 4,
 $$
where $\beta = e^{- (\log \delta^{-1})^{0.7}}$. Let $\mathcal A_j = \{  (\epsilon s')^{-2} \tilde M_{\gamma, \epsilon^2 s', \eta} (\tilde S_{i_j}) >  \beta \}$, which occurs with probability at least $1 - \beta C_{\gamma, -1} \ge 1/2$ (see \eqref{eq-LQG-negative-moment} for the constant $C_{\gamma, -1}$). Then 
 $$
\P ( \sum_{j=1}^{K^2/4} \frac { \tilde M_{\gamma, \epsilon^2 s', \eta} (\tilde S_{i_j}) } {(\epsilon s')^{-2}}  \le \frac{\beta^2 K^2 } 4 ) \le \P (\sum_{j=1}^{K^2/4} {\bf 1}_{\mathcal A_j} \le \frac {\beta K^2 } 4 ) \le ( \frac{1 + e^{-1}} 2 e^\beta)^{K^2/4} \le e^{- K^2} ,
 $$
completing the proof of \eqref{Eq.LD-lowerbound-approx-LGD}.
\end{proof}

It will be useful below to consider the Liouville graph distance when the ``LQG'' measure is computed using a perturbation of the GFF (such as the $\eta$-field). Explicitly, for any Borel set $A$ define
\begin{equation}\label{eq-def-M-eta}
M_{\gamma}^{\zeta} (A) = \lim_{n \to \infty} \int_z e^{\gamma \zeta_{2^{-n}}(z) - \frac{\gamma^2}{2} \E (\zeta_{2^{-n}}(z))^2}  \mathcal L_2(dz)\,,
\end{equation}
where we will only work with fields $\zeta$ such that the above limit exists almost surely, including the white noise process $\tilde h$ as in \eqref{eq:WND_decomposition}, and the $\eta$-process introduced in \eqref{eq:WND_decomposition-approximation}.
 For $u, v\in \mathbb V$, we then define the $\zeta$-Liouville graph distance $D_{\gamma, \delta, \zeta}(u, v)$ to be the size of the smallest collection of Euclidean balls with rational centers, each  of $M_{\gamma}^\zeta$-measure at most $\delta^2$, so that the collection contains a path from $u$ to $v$.

\begin{lemma}\label{lem-LGD-compare}
Suppose that two fields $\zeta^{(1)}_\cdot(\cdot)$ and $\zeta^{(2)}_\cdot(\cdot)$ are such that \eqref{eq-def-M-eta} is well defined for both processes. In addition, assume that
\begin{equation}\label{eq-var-compare}
\max_{v\in \mathbb V} \max_{n\geq 0} |\Var \zeta^{(1)}_{2^{-n}}(v) - \Var\zeta^{(2)}_{a 2^{-n}}(v) | \leq b_1 \mbox{ for some } a, b_1 >0\,.
\end{equation}
Suppose there exists an instance of the two fields satisfying that
\begin{equation}\label{eq-coupling-two-fields}
\max_{v\in \mathbb V}\max_{n\geq 0} |\zeta^{(1)}_{2^{-n}}(v) - \zeta^{(2)}_{a 2^{-n}}(v)|\leq b_2 \mbox{ for some }b_2>0\,.
\end{equation} 
Then, on this instance we have for all $u, v\in \mathbb V$
$$D_{\gamma, \delta e^{\gamma^2 b_1/4 +  \gamma b_2/2},  \zeta^{(2)}} (u, v) \leq D_{\gamma, \delta, \zeta^{(1)}} (u, v) \leq D_{\gamma, \delta e^{- \gamma^2 b_1/4 - \gamma b_2/2},  \zeta^{(2)}} (u, v)\mbox{ for all } \delta > 0\,.$$
\end{lemma}
\begin{proof}
We see from \eqref{eq-var-compare} and \eqref{eq-coupling-two-fields}  that $e^{-\gamma^2 b_1/2  - \gamma b_2} M_{\gamma, \zeta^{(2)}}(A)\leq M_{\gamma, \zeta^{(1)}}(A) \leq e^{\gamma^2 b_1/2 + \gamma b_2}M_{\gamma, \zeta^{(2)}}(A)$  for any Borel set $A\subseteq \mathbb V$.
This implies that $D_{\gamma, \delta, \zeta^{(1)}}(u, v) \leq D_{\gamma, \delta e^{- \gamma^2 b_1/4 - \gamma b_2/2},  \zeta^{(2)}}(u, v)$ for the reason that any Euclidean ball with $M_{\gamma,  \zeta^{(2)}}$-LQG measure at most $[\delta e^{- \gamma^2 b_1/4 - \gamma b_2/2}]^2$ has $M_{\gamma, \zeta^{(1)}}$-LQG measure at most $\delta^{2}$. The other inequality follows from the same reasoning.
\end{proof}
Recall the definition of $C_{\mathrm{Mc}}$  in Lemma~\ref{lem-partition-minimal-cell}.
\begin{cor}\label{cor-LGD-two-deltas}
For any fixed $0<\xi< C_{\mathrm{Mc}}/3$, any function $\delta' = \delta' (\delta) \in (0, \delta)$
and any  sequence of $\xi$-admissible pairs  $(A_\delta, B_\delta )$, we have with high probability that
$$\min_{x\in A_\delta, y\in B_\delta}D_{\gamma, \delta'}(x, y) \leq   \min_{x\in A_\delta, y\in B_\delta} D_{\gamma, \delta}(x, y) \cdot e^{(\log \delta^{-1})^{0.9}}  (\delta/\delta')^3 \,.$$
 Furthermore, the statement holds with $D_{\gamma, \delta}$ replaced by $D_{\gamma, \delta, \eta}$.
\end{cor}
\begin{proof}
The statement on $D_{\gamma, \delta}$ follows immediately from Proposition~\ref{prop-approximate-LGD} and Lemma~\ref{lem-approximate-LGD-two-deltas}. The statement on $D_{\gamma, \delta, \eta}$ then follows additionally from Lemma~\ref{lem-tilde-h-eta} and Lemma~\ref{lem-LGD-compare}.
\end{proof}

\begin{lemma}\label{lem-LGD-two-fields}
For any fixed $0<\xi< C_{\mathrm{Mc}}/3$,
any $\delta>0$ and any  sequence of $\xi$-admissible pairs  $(A_\delta, B_\delta )$, we have with high probability
\begin{equation}
\label{ofer-morn1}
e^{-10(\log \delta^{-1})^{0.9}}\min_{u\in A_\delta, v\in B_\delta}D_{\gamma, \delta, \eta}(u, v)  \leq \min_{u\in A_\delta, v\in B_\delta} D_{\gamma, \delta}(u, v) 
\leq e^{10(\log \delta^{-1})^{0.9}}
\min_{u\in A_\delta, v\in B_\delta} D_{\gamma, \delta, \eta}(u, v)
\end{equation}
Furthermore, 
\begin{equation}
\label{eq-oferberlin1}
|\E \min_{u\in A_\delta, v\in B_\delta}\log D_{\gamma, \delta}(u, v) - \E \min_{u\in A_\delta, v\in B_\delta} \log D_{\gamma, \delta, \eta}(u, v) | =  O((\log \delta^{-1})^{0.9})\,.
\end{equation}
\end{lemma}
\begin{proof}
The estimate \eqref{ofer-morn1} follows from
Lemma~\ref{lem-tilde-h-eta}, Lemma~\ref{lem-LGD-compare} and Corollary~\ref{cor-LGD-two-deltas}.
The estimate \eqref{eq-oferberlin1}
follows from
\eqref{ofer-morn1} combined with
\eqref{eq-very-crude} (and an analogous version for $D_{\gamma, \delta, \eta}$ which can be derived in the same manner), and an application of the
Cauchy-Schwarz inequality.
\end{proof}

\subsection{Regularity of the random partition}
The goal of this section is to prove a version of regularity for the random partition, as incorporated in Lemma~\ref{lem-regularity}. As a consequence, we obtain Lemma~\ref{lem-partition-independence}, which will play a crucial role in proving the lower bound on the Liouville heat kernel and Lemma~\ref{lem-existence-exponent}. 
For $\alpha^*,\delta>0$, set
\begin{equation}\label{eq-def-epsilon*}
\epsilon^* = \epsilon^*_\delta = \max\{2^{-n}: 2^{-n} \leq  \exp\{-\alpha^* \sqrt{\log {\delta^{-1}}} \log \log \delta^{-1}\}\}\,.
\end{equation}

  \begin{defn} \label{def-good-sequence-box}
Let $\mathcal B = (B_1, \ldots, B_d)$ denote
 a sequence of neighboring boxes, and write $B_0:=B_1$ and $B_{d+1}:=B_d$. For
$i=1,\ldots,d$, we say that 
$B_i$ is good (in $\mathcal B$) if 
$s_{B_{i-1}}, s_{B_{i+1}} \in [ s_{B_i} \epsilon^*, s_{B_i} / \epsilon^* ]$. We
say that $\mathcal B$ is a  good sequence if all $B_i$'s are good in $\mathcal B$. We say that a point $x$ is good if for any cell $\mathsf C \in \mathcal V_\delta$ such that $x \in \mathsf C_{\mathrm{large}}$, one has that for any $w\in \mathsf C_{\mathrm{large}}$,
the side length of ${\mathsf C}_{w,\delta}$ satisfies $s_{w, \delta}  \ge \epsilon^* s_{\mathsf C}$.  
  \end{defn}
	Let
 \begin{equation}\label{eq-def-E-delta-alpha*-u-v}
 \begin{split}
\mathcal E_{\delta, \alpha^*, u, v}  : =&   \mathcal E_{\delta, \alpha^*} \cap \{ \mbox{$u$ and $v$ are good, and there exists a good sequence of cells }
 \\ & \quad
  \mathsf C_1, \ldots, \mathsf C_d \mbox{ joining $u$ and $v$ with  $d \leq D'_{\gamma, \delta}(u,  v) e^{(\log \delta^{-1})^{0.6}}$} \}.
\end{split}
 \end{equation} 
 Note that $\mathcal E_{\delta, \alpha^*, u, v}$ 
is measurable with respect to $\mathcal V_\delta$. Recall the constant
$\alpha_0$ from Lemma~\ref{lem-neighboring-cell}.

\begin{lemma}\label{lem-regularity}
There exists $\alpha^* = \alpha^*(\gamma) \ge \alpha_0$ so that,
 for any fixed $u, v\in \mathbb V$, we have $\P(\mathcal E_{\delta, \alpha^*, u, v}) \geq 1- e^{-(\log \delta^{-1})^{1/4}}$.
\end{lemma}

In the rest of the paper, we will stick to the choice of $\alpha^*$ so that the conclusion of  Lemma~\ref{lem-regularity} holds.
The following lemma clarifies the notion of goodness encoded in the definition 
of $\mathcal E_{\delta, \alpha^*, u, v}$.
In the statement, we do not distinguish between $\mathcal V_\delta$ and the filtration generated by it.

\begin{lemma}\label{lem-partition-independence}
On the event $\mathcal E_{\delta, \alpha^*, u, v}$, there exists a sequence,
measurable with respect to $\mathcal V_\delta$,
 of neighboring dyadic boxes $\mathsf B_1, \ldots, \mathsf B_d$ joining $u, v$ with $d\leq D'_{\gamma, \delta}(u, v) e^{2(\log \delta^{-1})^{0.6}}$,
such that each $\mathsf B_i$ is contained in some cell $\mathsf C$ with $s_{\mathsf B_i} = s_{\mathsf C} (\epsilon^*)^2 $. Furthermore, the 
 law of  
$\{\eta_{\delta'}^{s_{B_i}}(x): \delta' < s_{B_i}, x \in (B_i)_{\mathrm{large}}, i=1, \ldots, d \}$
conditioned on $\mathcal V_{\delta}$ coincides with
its unconditional version. Explicitly, for any measurable function  $F$,
 $$
{\bf 1}_{\mathcal E_{u, v, \delta, \alpha^*}} \E( F( \{\eta_{\delta'}^{s_{B_i}}(x): \delta' < s_{x,\delta} \epsilon^*, x \in \cup_{i=1}^d \mathsf B_i \} )
\mid \mathcal V_\delta )= {\bf 1}_{\mathcal E_{u, v, \delta, \alpha^*}}  \varphi_F (f)\,,
 $$ 
where $\varphi_F (g) := \E(F(\{\eta^{g(i)}_{\delta'} (x) : \delta' < s_i, x \in \mathsf (B_i)_{\mathrm{large}}, i = 1, \ldots, d) )$, and $f(i) = s_{B_i}$.
\end{lemma}
\noindent
(Recall that the collection of random variables 
$\{s_{v,\delta}\}$ is measurable with respect to $\mathcal V_\delta$.)
\begin{proof}
On  $\mathcal E_{\delta, \alpha^*,u,v}$, one can find a good sequence $\mathcal C = (\mathsf C_1, \ldots, \mathsf C_{d_0})$ joining $u$ and $v$ with $d_0 \leq D'_{\gamma, \delta}(u, v) e^{(\log \delta^{-1})^{0.6}}$. Denote $\Lambda_j = \partial \mathsf C_j \cap \partial \mathsf C_{j+1}$,  let $x_j$ denote the middle of $\Lambda_j$, and let $\mathcal C_j$ denote the partition of $\mathsf C_j$ into boxes of side length $(\epsilon^*)^2 s_{\mathsf C_j}$. 
Since $\mathcal C$ is good, one can find for each $2 \le j \le d_0-1$ a sequence of boxes $\{ B_{j,i}  \}$'s in $\mathcal C_j$ joining $x_{j-1}$ and $x_j$ such that each $B_{j,i}$ has distance at least $\frac 1 3 \epsilon^* s_{\mathsf C_j}$ from  $\partial \mathsf C_j \setminus (\Lambda_{j-1} \cup \Lambda_j) $. Let $\{ B_{1,i} \}$ be an arbitrary sequence of boxes in $\mathcal C_1$ joining $u$ and $x_1$, and define
similarly  $\{ B_{d_0, i} \}$.  To ensure connectivity, the boxes in $\mathcal C_j$ whose closures contain $x_{j-1}$ or $x_j$ are all collected in $B_{j,i}$'s. Now it suffices to check the requirement of conditional law for the sequence $\cup_j \{ B_{j,i} \}$. Note on $\mathcal E_{u, v, \delta, \alpha^*}$, one has $s_{\mathsf C_j} \ge \delta^{C_{\mathrm{mc}}}$ thus $ s_{B_{j,i}} \log \frac 1 {s_{B_{j,i}}} = (\epsilon^*)^2 s_{\mathsf C_j} \log \frac 1 {(\epsilon^*)^2 s_{\mathsf C_j}} < \frac 1 {10} \epsilon^* s_{\mathsf C_j} $. Combined with the fact that $\mathcal C$ is good, this implies that
 \begin{align} \label{Eq.fine-field-independent}
&\mbox{the construction of $\mathcal V_\delta$ does not explore the white noise}
 \\ & 
\mbox{appearing in $\{ \eta_{\delta'}^{s_{B_{j, i}}} (x): \delta' < s_{B_{j,i}}, x \in (B_{j,i})_{\mathrm{large}} \}$, }\nonumber
  \end{align}
completing the proof.
\end{proof}

The main task for the rest of the section is to prove Lemma~\ref{lem-regularity}. We will employ  a percolation-type  analysis of the same flavor as 
 in the proof of Lemma~\ref{lem-approximate-LGD-two-deltas} and Proposition~\ref{prop-approximate-LGD}. However, the percolation argument employed here is substantially more involved as we are required to control the ratios for the sizes of neighboring cells in the short path we find (by deforming the geodesic). 

\begin{defn}[$\mathcal E_{\delta, B}$]  \label{def-E-delta-B}
Let $B$ denote a dyadic box and fix $\delta>0$.
We define the event $\mathcal E_{\delta, B}$ to be  the following: there exists a sequence of neighboring boxes   $B_1, \ldots B_d \subseteq B_{\mathrm{large}}  \setminus B$ enclosing $B$
such that
 \begin{itemize}
\item $B_i \in\mathcal B(B, \epsilon^*)$ for each $1\leq i\leq d$, where 
$\mathcal B (B, \epsilon^*)$ is as in Definition~\ref{def-E-delta-B-prime}.  
\item $M_{\gamma, \epsilon^* s_{B}}(B_i) \leq \delta^2$ for each $1\leq i\leq d$.
\end{itemize}
\end{defn}
\begin{remark}\label{remark-box-cell}
We note that the sequence $B_1, \ldots, B_d$ does not
necessarily consists of cells in $\mathcal V_\delta$. However, each of the $B_i$s
 must be contained in a  (possibly larger) cell, which intersects $B_{\mathrm{large}} \setminus B$.
\end{remark}
\begin{lemma}\label{lem-percolation-regularity}
The following holds for large enough $0<\alpha<\alpha^* = \alpha^*(\gamma)$:
 for each dyadic box $B$ with 
side length $s =s_B= 2^{-m}$, $1\leq m \leq C_{\mathrm{mc}}\log_2 \delta^{-1}$, we have
\begin{equation}\label{eq-B-percolation}
\P(\{M_{\gamma, s} (B) \leq \delta^2\} \cap \mathcal E_{\delta, \alpha} \cap \mathcal E^c_{\delta, B}) \leq \delta^{10 C_{\mathrm{mc}} + 10}\,.
\end{equation}
Furthermore, 
\begin{equation}\label{eq-B-good}
\P(\{M_{\gamma, s} (B) \leq \delta^2\} \cap \mathcal E_{\delta, \alpha} \cap \{M_{\gamma, \epsilon^* s}(B') > \delta^2 \mbox{ for some } B'\in \mathcal B (B, \epsilon^*) \} ) \leq e^{-\sqrt{\log \delta^{-1}}}\,.
\end{equation}
\end{lemma}
(Note that $\alpha^*$ enters in the statement of Lemma \ref{lem-percolation-regularity} 
through the definition of $\epsilon^*$, see \eqref{eq-def-epsilon*}.)
\begin{proof}
The proof resembles that for Lemma~\ref{lem-percolation-Phi}. Let 
 $\epsilon = (\alpha_1\log \delta^{-1})^{-1}$ be dyadic with $\alpha_1\in [1,2]$, and
assume in what follows that $\alpha>\max(2,\alpha_0,4C_{\mbox{\rm mc}})$, see
Lemma~\ref{lem-percolation-Phi}. 
Let $t = \epsilon^*$. 
Recall that $\mathcal B(B, \epsilon)$ denotes the partition of $B_{\mathrm{large}}$ into small boxes $B'_i$'s of side length $\epsilon s$, and that $\mathcal B'_i = \mathcal B_{\partial} (B_i, t/\epsilon)$ denotes the
 boxes of side length $\epsilon^* s$ whose closures intersect $\partial B_i$.

Replacing $\eta_{ts}^{\epsilon^2 s}$ in the proof of Lemma~\ref{lem-percolation-Phi} with $\eta_{ts}^{\epsilon s / \log s^{-1}}$,  we obtain
 \begin{equation}\label{eq-LQG-tilde-B}
M_{\gamma, \epsilon^* s}(\tilde B) \leq \delta^2 e^{2\gamma \alpha \sqrt{\log \delta^{-1}} \log \log \delta^{-1}} (\epsilon^*)^2 e^{\gamma \eta_{\epsilon^* s}^{\epsilon s / (\log s^{-1})}(c_{\tilde B}) - \frac{\gamma^2}{2} \Var( \eta_{\epsilon^* s}^{ \epsilon s / (\log s^{-1})}(c_{\tilde B}))}\,,
 \end{equation}
where  $c_{\tilde B}$ denotes the center of $\tilde B \in \mathcal B (B, \epsilon^*) \
\cup \mathcal B_{\partial} (B_{\mathrm{large}}, \epsilon^*)$. 
Analogously, for appropriate choices of $\alpha, \alpha^*$,  we have that
$\mathcal E_{B'_i, \mathrm{open}} : = \{ \max_{\tilde B \in \mathcal B'_i } \eta_{ts}^{\epsilon s / (\log s^{-1}) } (c_{\tilde B}) \le 1.5 \log t^{-1} \}$ satisfies \eqref{eq-B'-i-open} with $p = t^{0.1}$ and $\kappa = 2$, and 
 $$
( \{M_{\gamma, s} (B) \leq \delta^2\} \cap \mathcal E_{\delta, \alpha}  \cap \mathcal E_{B'_i, \mathrm{open}} ) \subseteq \{M_{\gamma, s \epsilon^*} (\tilde B) \leq \delta ^2: \tilde B\in \mathcal B'_i\}\,,
 $$
(recall \eqref{eq-LQG-tilde-B-Phi}). Then, the percolation argument in  Lemma~\ref{lem-percolation-Phi} yields \eqref{eq-B-percolation}. 

It remains to prove \eqref{eq-B-good}. By a union bound, we see that
\begin{equation}\label{eq-for-B'-good}
\P(\max_{\tilde B\in \mathcal B (B, \epsilon^*)}\eta_{\epsilon^* s}^{ \epsilon s / ( \log s^{-1})}(c_{\tilde B}) \leq (1 + \frac{1}{\gamma} + \frac{\gamma}{4}) \log (1/\epsilon^*)) \geq 1 -  (\epsilon^*)^{\Omega(\gamma)}\,,
\end{equation}
where the choice of $1 + \frac{1}{\gamma} + \frac{\gamma}{4}$ is so that $1 + \frac{1}{\gamma} + \frac{\gamma}{4} > 2$ and $\gamma (1 + \frac{1}{\gamma} + \frac{\gamma}{4}) < 2 + \frac{\gamma^2}{2}$. Combining this with \eqref{eq-LQG-tilde-B} completes the proof of \eqref{eq-B-good}.
\end{proof}

\begin{proof}[Proof of Lemma~\ref{lem-regularity}]
We work on the event $\mathcal E_{\delta, \alpha}$, since by Lemma~\ref{lem-neighboring-cell} it occurs with high probability.
Applying \eqref{eq-B-good} to all dyadic boxes $B$ with $B_{\mathrm{large}}$ containing $u$ or $v$, and with side length $s_B \ge \delta^{C_{\mathrm{mc}}}$ (so in total we apply \eqref{eq-B-good} $O(\log \delta^{-1})$ times), we see that $u$ and $v$ are good with probability at least $1 - O(\log \delta^{-1}) e^{-\sqrt{\log \delta^{-1}}}$. Also, by \eqref{eq-B-percolation} and a union bound, we see that with high probability $\mathcal E_{\delta, \mathsf C}$ holds for each $\mathsf C \in \mathcal V_\delta$. Recalling Remark~\ref{remark-box-cell}, we then get that with high probability 
\begin{equation}\label{eq-percolation-good-surrounding}
\begin{split}
&\mbox{ there exists a sequence of neighboring cells with side length at least } \epsilon^* s_{\mathsf C} \\
&\mbox{which encloses $\mathsf C$ and which 
has all cells  intersecting with } \mathsf C_{\mathrm{large}}\setminus \mathsf C \mbox{  for each }\mathsf C \in \mathcal V_\delta\,. 
\end{split}
\end{equation} 
Let $\mathcal C_0 = (\mathsf C_1, \ldots, \mathsf C_{d_0})$ be the geodesics in $D'_{\gamma, \delta}$ joining $u$ and $v$. We will show that 
\eqref{eq-percolation-good-surrounding} and the assumption that $u$ and $v$ are good imply that  $\mathcal E_{\delta, \alpha^*, u,v}$ holds, that is that
 \begin{equation} \label{Eq.sequence-good-cells}
\mbox{one can find a good sequence of cells joining $u$ and $v$ with length at most $d_0 e^{(\log \delta^{-1})^{0.6}}$.}
 \end{equation}
Since $\mathcal E_{\delta, \alpha^*, u, v}$ is increasing in $\alpha^*$, 
one can adjust $\alpha^*$ such that it is larger than $\alpha_0$,  completing the proof
of the lemma. 

It remains to prove \eqref{Eq.sequence-good-cells}, assuming that
$\mathcal E_{\delta, \alpha}$ holds, $u$ and $v$ are
 good, and \eqref{eq-percolation-good-surrounding}. 
For a sequence of neighboring cells $\mathcal C$, we let $\psi(\mathcal C)$ be the collection of cells $\mathsf C\in \mathcal C$ which have
 a neighboring cell in $\mathcal C$ with side length less than $\epsilon^* s_{\mathsf C}$ (that is to say, $\mathsf C$ is a not a good cell in $\mathcal C$ as in Definition~\ref{def-good-sequence-box}, which we refer to as a bad cell). Let $q(\mathcal C)$ be the side length of the largest cell in $\psi(\mathcal C)$. For ${\mathsf C},  {\mathsf C}' \in \mathcal C$, we denote by $[ {\mathsf C},  {\mathsf C}' ]_{\mathcal C}$ the path
 in $\mathcal C$ connecting ${\mathsf C}$ and ${\mathsf C}'$, and 
by $( {\mathsf C},  {\mathsf C}' )_{\mathcal C}$ the interior of $[{\mathsf C}, {\mathsf C}' ]_{\mathcal C}$ (i.e., excluding $ {\mathsf C}$ and $ {\mathsf C}'$). Similarly, we have $[\mathsf C, \mathsf C')_{\mathcal C}$ and $(\mathsf C, \mathsf C']_{\mathcal C}$.  
 
For $i\geq 0$, we will employ the following iterative construction,
constructing $\mathcal C_{i+1}$ from $\mathcal C_{i}$. If $\mathcal C_i$ is not good, we pick the largest $\mathsf C\in \psi(\mathcal C_i)$.  Since $u$ and $v$ are good, we see that $u, v\notin \mathsf C_{\mathrm{large}}$ and thus $\mathcal C_i$ will have to enter from outside and also exit from $\partial \mathsf C_{\mathrm{large}}$ --- here naturally $\mathsf C$ should implicitly depend on $i$, but we have suppressed it in the notation for simplicity. Let $\mathsf C_{\mathrm{enter}}$ be the last cell in $\mathcal C_i$ before $\mathsf C$ which intersects $\partial \mathsf C_{\mathrm{\mathrm{large}}}$ and let $\mathsf C_{\mathrm{exit}}$ be the next cell in $\mathcal C_i$ after $\mathsf C$ that intersects $\partial \mathsf C_{\mathrm{large}}$. 

We claim that there always exists a sequence of neighboring cells $\mathcal C_{i, \mathrm{replace}}$ (which is a segment of \eqref{eq-percolation-good-surrounding}) joining ${\mathsf C}_{i, 1} \in [\mathsf C_{\mathrm{enter}}, \mathsf C)$ and ${\mathsf C}_{i, 2} \in (\mathsf C, \mathsf C_{\mathrm{exit}}]$ such that if we construct $\mathcal C_{i+1}$ by replacing $[{\mathsf C}_{i, 1}, {\mathsf C}_{i, 2}]_{\mathcal C_i}$ in $\mathcal C_i$ with $\mathcal C_{i, \mathrm{replace}}$, then either of the following occurs:
\begin{enumerate}[(i)]
\item $|\psi(\mathcal C_{i+1})| \leq |\psi(\mathcal C_i) |-1$;
\item $|\psi(\mathcal C_{i+1})| \leq |\psi(\mathcal C_i) |$ and $q(\mathcal C_{i+1}) \geq 2 q(\mathcal C_{i})$.
\end{enumerate}
Provided with this claim, we can then construct iteratively
$\mathcal C_{i+1}$, and we see that in every $C_{\mathrm{mc}} \log_2 \delta^{-1}$ 
steps the number of bad cells has to decrease by at least 1 (this is because the second scenario cannot occur continuously for more than $C_{\mathrm{mc}} \log_2 \delta^{-1}$ steps due to the fact that all cells have size 
between $\delta^{C_{\mathrm{mc}}}$ and 1). Thus, the iterative procedure will stop after at most $d_0 \times C_{\mathrm{mc}} \log_2 \delta^{-1}$ steps and  end up with a good sequence. Also, in every step, the number of cells increases by at most $4(\epsilon^*)^{-2}$. Therefore, in the end, we obtain a good sequence of neighboring cells with length at most $d_0 \times (\epsilon^{*})^{-2} C_{\mathrm{mc}} \log_2 \delta^{-1} \le d_0 e^{(\log \delta^{-1})^{0.6}}$, as required. That is, \eqref{Eq.sequence-good-cells} holds.

 \begin{figure}[h]
  \includegraphics[width=17cm]{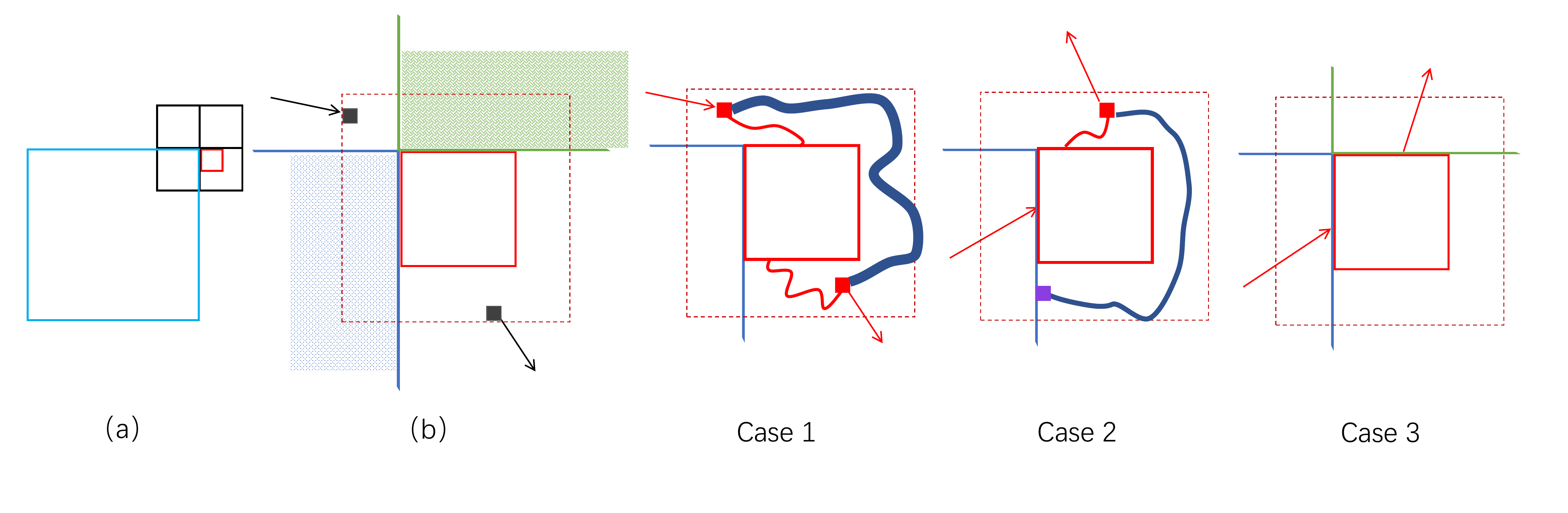}
\\  \vspace{-1.5cm} \caption{(a) The smallest (red) box is $\mathsf C$, the intermediate
(black) boxes are $\mathsf B_{\mathrm{lt}}$, etc, and the largest (blue)
 box is $\mathsf C_{\mathrm{lb}}$. (b) lack of connectivity by $\mathsf C_{\mathrm{lb}}$ and $\mathsf C_{\mathrm{rt}}$, where the small (black) solid boxes are $\mathsf C_{\mathrm{enter}}$ and $\mathsf C_{\mathrm{exit}}$. In Case 1, the small
(red) solid boxes are $\mathsf C_{i,1}$ and $\mathsf C_{i,2}$, the thin (red) curve stands for $[\mathsf C_{i,1}, \mathsf C_{i,2} ]$ and the thick (blue) curve stands for $\mathcal C_{i, \mathrm{replace}}$. In Case 2,  the small bottom (purple) solid box
 stands for the neighbor of $\mathsf C_{i,1} = \mathsf C_{\mathrm{lb}}$ in 
$\mathcal C_{i, \mathrm{replace}}$, which may have side length less than $\epsilon^* s_{\mathsf C_{i,1}}$.}
 \label{fig-Lemma32} 
 \end{figure}

It remains to justify the above claim.  We first prove it in the harder case when $\mathsf C_{\mathrm{large}} \subseteq \mathbb V^o$.  As shown in (a) of Figure~\ref{fig-Lemma32}, let $\mathsf B_{\mathrm{lt}},  \mathsf B_{\mathrm{rt}}, \mathsf B_{\mathrm{lb}}, \mathsf B_{\mathrm{rb}}$ be the four dyadic boxes with side length $2s_{\mathsf C}$ whose closures  have non-empty intersection with the closure of $\mathsf C$; here,
the subscript $\mathrm{lt}$ means ``left-top' and $\mathrm{rb}$ means ``right-bottom'', etc. (Note that $\mathsf B_{\mathrm{lt}}, \mathsf B_{\mathrm{rt}}, \mathsf B_{\mathrm{lb}}, \mathsf B_{\mathrm{rb}}$ are not necessarily cells in $\mathcal V_\delta$.)  We suppose without loss of generality that $\mathsf C\subset \mathsf B_{\mathrm rb}$ (so that 
all cells in $\mathsf B_{\mathrm{rb}}$ have side length at most $s_{\mathsf C}$).  If $\mathsf B_{\mathrm{lt}}$ is not partitioned, then denote by  $\mathsf C_{\mathrm{lt}}$ the cell containing $\mathsf B_{\mathrm{lt}}$ (otherwise we define $\mathsf C_{\mathrm{lt}} = \emptyset$) --- similarly for $\mathrm{lb}, \mathrm{rt}, \mathrm{rb}$. Let  $\mathfrak C_{\mathrm{parents}} = \{ \mathsf C_{\mathrm{lt}}, \mathsf C_{\mathrm{rt}}, \mathsf C_{\mathrm{lb}} \}$, noting $\mathsf C_{\mathrm{rb}} = \emptyset$. Note that it is possible that 
$\mathfrak C_{\mathrm{parents}} = \{\emptyset\}$.  By \eqref{eq-percolation-good-surrounding} there exists a sequence $\mathcal C_{i, \mathrm{cross}}$
of neighboring cells  with side length at least $\epsilon^* s_{\mathsf C}$, which encloses $\mathsf C$ and has all cells intersecting with $\mathsf C_{\mathrm{large}}\setminus \mathsf C$. Suppose that $\mathcal C_{i, \mathrm{cross}}$ intersects $[\mathsf C_{\mathrm{enter}}, \mathsf C_{\mathrm{exit}}]_{\mathcal C_i}$ at $ {\mathsf C}_{i, 1}$ and $ {\mathsf C}_{i, 2}$. Then, $\mathcal C_{i, \mathrm{cross}}$ can be
split two segments,
with respective ending cells
 $ {\mathsf C}_{i, 1}$ and $ {\mathsf C}_{i, 2}$.

We first show that the interior of one of the segments does not intersect  $ \mathfrak C_{\mathrm{parents}}$. Suppose this does not hold. If $\mathsf C_{\mathrm{lt}}$ lies in 
the interior of a segment, neither  $\mathsf C_{\mathrm{lb}}$ nor
 $\mathsf C_{\mathrm{rt}}$ lie in the  interior of the other segment, because
they are neighbors of $\mathsf C_{\mathrm{lt}}$. Then, 
 $\mathsf C_{\mathrm{lb}}$ and $\mathsf C_{\mathrm{rt}}$ respectively lie in the interior of different segments, as shown in (b) of Figure~\ref{fig-Lemma32}. 
By connectivity, this implies that one of them is contained in $(\mathsf C_{i,1}, \mathsf C_{i,2})_{\mathcal C_i} \subseteq (\mathsf C_{\mathrm{enter}}, \mathsf C_{\mathrm{exit}})_{\mathcal C_i}$, arriving at a contradiction to the definitions of $\mathsf C_{\mathrm{enter}}$ and $\mathsf C_{\mathrm{exit}}$.

Next, we prove our claim in the following separate cases, as shown in Figure~\ref{fig-Lemma32}.

\noindent {\bf Case 1: $ {\mathsf C}_{i, 1}, {\mathsf C}_{i, 2} \notin \mathfrak C_{\mathrm{parents}}$.}  In this case we can just let $\mathcal C_{i, \mathrm{replace}}$ be the segment which does not contain any cell in $ \mathfrak C_{\mathrm{parents}}$. By our assumption, we see that all cells in $\mathcal C_{i, \mathrm{replace}}$ have side lengths in $[\epsilon^* s, s]$. Therefore, $\psi(\mathcal C_{i, \mathrm{replace}}) = \emptyset$. In addition,  $\mathsf C\notin \mathcal C_{i, \mathrm{replace}}$. Thus, we have justified (i) of the claim.

\noindent {\bf Case 2: $|\{ {\mathsf C}_{i, 1}, {\mathsf C}_{i, 2} \} \cap \mathfrak C_{\mathrm{parents}}| = 1$.} In this case, we  repeat the procedure as in Case 1. 
However, (supposing ${\mathsf C}_{i, 1} \in \mathfrak C_{\mathrm{parents}}$) it is now possible that $\psi(\mathcal C_{i, \mathrm{replace}}) = \emptyset$ or $\psi(\mathcal C_{i, \mathrm{replace}}) = \{ {\mathsf C}_{i, 1}\}$. The former case shows (i);
in the latter case, we have (ii), where $q(\mathcal C_{i+1}) = s_{ {\mathsf C}_{i, 1} } \geq 2s_{\mathsf C} = 2 q(\mathcal C_i)$.

\noindent {\bf Case 3: $\{ {\mathsf C}_{i, 1}, {\mathsf C}_{i, 2} \} \subset \mathfrak C_{\mathrm{parents}}$}. In this case, we also have ${\mathsf C}_{i, 1} = \mathsf C_{\mathrm{enter}}$ and ${\mathsf C}_{i, 2} =  \mathsf C_{\mathrm{exit}}$ (or with the ordering switched), and thus both ${\mathsf C}_{i, 1}$ and ${\mathsf C}_{i, 2}$ are neighboring to (in the sequence $\mathcal C_i$) cells of side length at most $s_{\mathsf C}$. By maximality of $\mathsf C$ in $\psi(\mathcal C_i)$, we see that $ {\mathsf C}_{i, 1}$ and $ {\mathsf C}_{i, 2}$ have side lengths at most $s_{\mathsf C}/\epsilon^*$ (and at least $2s_{\mathsf C}$ since they are in $\mathfrak C_{\mathrm{parents}}$). If $ {\mathsf C}_{i, 1}$  and ${\mathsf C}_{i, 2}$ are diagonal to each other (then they must be both neighboring $\mathsf C$), we let $\mathcal C_{i, \mathrm{replace}}$ be the sequence $ {\mathsf C}_{i, 1}, \mathsf C,  {\mathsf C}_{i, 2}$; if $ {\mathsf C}_{i, 1}$  and $ {\mathsf C}_{i, 2}$ are neighboring to each other, then we let $\mathcal C_{i, \mathrm{replace}}$ be the sequence $ {\mathsf C}_{i, 1}, {\mathsf C}_{i, 2}$. In both cases, we have $\psi(\mathcal C_{i, \mathrm{replace}}) = \emptyset$, justifying (i).

We next consider the easier case that $\mathsf C$ intersects $\partial \mathbb V$. In this case, $\mathfrak C_{\mathrm{parents}}$ contains at most one cell in $\mathbb V$ and we are either in Case 1  or Case 2. Following similar (and slightly simpler) analysis to the one above then yields the proof of the claim in this case.
Altogether, this completes the verification of the claim, and thus completes the proof of the lemma.
\end{proof}

\subsection{Concentration of the distances}
In this section we show the following concentration result on the Liouville graph distance. Recall the constant $C_{\mathrm{Mc}}$ specified in
 Lemma~\ref{lem-partition-minimal-cell}. 
\begin{prop}\label{prop-concentration}
For any fixed $0<\xi< C_{\mathrm{Mc}}/3$ there exists a constant $c=c(\gamma, \xi)$ so that for any  sequence of $\xi$-admissible pairs  $(A_\delta, B_\delta )$
we have that for any $\iota\in (0,1)$, 
\begin{equation}\label{eq-concentration-1}
|\log \min_{x\in A_\delta, y\in B_\delta} D_{\gamma, \delta}(x, y) - \E \log \min_{x\in A_\delta, y\in B_\delta}D_{\gamma, \delta}(x, y) | \leq \iota \log \delta^{-1}
\ \mbox{\rm with $c\cdot \iota^2$-high probability}\,.
\end{equation}
In addition, with probability at least $1- e^{-(\log \delta^{-1})^{0.7}}$, we have that
\begin{equation}\label{eq-concentration-2}
|\log \min_{x\in A_\delta, y\in B_\delta} D_{\gamma, \delta}(x, y) - \E \log \min_{x\in A_\delta, y\in B_\delta}D_{\gamma, \delta}(x, y) | \leq (\log \delta^{-1})^{0.95}\,.
\end{equation}
Furthermore, \eqref{eq-concentration-1} and \eqref{eq-concentration-2} hold with $D_{\gamma, \delta}$ replaced with $D_{\gamma, \delta, \eta}$.
\end{prop}
\begin{proof}
We first give a detailed proof of \eqref{eq-concentration-1} and then sketch the necessary minor adaptations needed in order to obtain \eqref{eq-concentration-2}. For  both \eqref{eq-concentration-1} and \eqref{eq-concentration-2}, we will only provide a proof in the case of 
$A = \{u\}$ and $B=\{v\}$, as the  general case follows by the same proof with minimal change --- the assumption of admissible pairs is required only in order to be able to apply Proposition~\ref{prop-approximate-LGD} and Lemma~\ref{lem-approximate-LGD-two-deltas}. Also, provided with \eqref{eq-concentration-1} and \eqref{eq-concentration-2}, the fact that \eqref{eq-concentration-1} and \eqref{eq-concentration-2} hold with $D_{\gamma, \delta}$ replaced with $D_{\gamma, \delta, \eta}$ follows from Lemma~\ref{lem-tilde-h-eta}, Corollary~\ref{cor-LGD-two-deltas} and Lemma~\ref{lem-LGD-compare}.

\medskip

\noindent {\bf Proof of \eqref{eq-concentration-1}}.
It is obvious from Proposition~\ref{prop-approximate-LGD} and Corollary \ref{cor-approximate-LGD-expectation} that  \eqref{eq-concentration-1} is equivalent to the statement that  with 
$c \cdot \iota^2$-high probability
\begin{equation}\label{eq-concentration-approximate}
|\frac{\log D'_{\gamma, \delta}(u, v)}{\log \delta^{-1}} - \E \frac{\log D'_{\gamma, \delta}(u, v)}{\log \delta^{-1}} | \leq \iota\,.
\end{equation}
Thus, it suffices to prove the concentration for either of the two distances. The natural attempt to prove Proposition~\ref{prop-concentration} is to verify the Lipschitz condition for the Liouville graph distance (viewed as a function on a Gaussian process) and then apply a Gaussian concentration inequality. However, while the Lipschitz condition for the Liouville graph distance can be verified, the maximal individual variance for the Gaussian variables involved in the definition of the Liouville graph distance is infinite. On the other hand, while the maximal individual variance for the Gaussian variables involved in the definition of the approximate Liouville graph distance can be controlled, the Lipschitz condition does not hold in an obvious way. In order to see the failing of the Lipschitz condition, note that one  can perturb the Gaussian process such that in constructing $\mathcal V_\delta$,  a cell that was not further 
partitioned in the original environment would now be further partitioned. 
Once this extra partitioning occurs, it is possible (but unlikely) that these sub-cells would be further partitioned into arbitrarily small Euclidean squares. 
(Indeed,  the decision concerning further partitioning   depends on random variables which are independent  from those determining the original partition.)
 In order to address this issue, we will employ the Lipschitz condition for the Liouville graph distance and the control on the maximal individual variance for the Gaussian variables involved in the approximate Liouville graph distance, and use Proposition~\ref{prop-approximate-LGD} to make a connection between these two distances.

We consider the Gaussian space generated by the collection 
$\{(\eta_\delta(v),\tilde h_\delta(v))\}_{v\in \mathbb V, \delta>0}$,
see \eqref{eq:WND_decomposition} and \eqref{eq:WND_decomposition-approximation}. For $\delta>0$, let
$\mathbf X_\delta$ denote the subspace spanned by
$\{(\eta_{\epsilon}(v), \tilde h_\epsilon (v)): v\in \mathbb V, \epsilon \geq \delta^{C_{\mathrm{mc}}}\}$. Let $\mathbf Y_\delta$ denote the subspace orthogonal to $\mathbf X_\delta$, and note that it is generated by the white noise $W(dw,ds)$ for $s<\delta^{2 C_{\mathrm{mc}}}$).
For $\delta'<\delta^{C_{\mathrm{mc}}}$ we  write the orthogonal decomposition 
\[(\eta_{\delta'}(\cdot),\tilde h_{\delta'}(\cdot))=
(\eta_{\delta^{C_{\mathrm{mc}}}}(\cdot ),  \tilde h_{\delta^{C_{\mathrm{mc}}}}(\cdot))
+
(\eta^\perp_{\delta,\delta'}(\cdot),\tilde h^\perp_{\delta,\delta'}(\cdot))
=: {\cal X}_\delta+{\cal Y}_{\delta,\delta'} , \]
where ${\cal Y}_{\delta,\delta'}(\cdot)$ is measurable
on $\mathbf Y_\delta$. (Possible configurations of $\cal X$ and $\cal Y$ will be  denoted by $\bf x$ and $\bf y$. We use   ${\bf x}_\delta$ and ${\bf y}_{\delta,\delta'}$  as
convenient  shorthand notation, and we further use ${\bf y}_\delta$ to denote the collection 
${\bf y}_{\delta,\delta'}$ for $\delta'<\delta^{C_{\mathrm{mc}}}$.)
 Denote by $M_{\gamma, \mathbf x_\delta}$ the LQG measure of the GFF on the realization $(\mathbf x_\delta, {\cal Y}_\delta)$.
We apply a similar convention for $M_{\gamma, \mathbf x'_\delta}$, $D_{\gamma, \delta, \mathbf x_\delta}(u, v)$, etc. We note that, by definition,  $D'_{\gamma, \delta} (u,v) = D'_{\gamma, \delta, \mathcal X_\delta} (u,v)$. Furthermore, $D'_{\gamma, \delta, \mathbf x_\delta} (u,v)$ is a real number if each cell has side length larger than $\delta^{C_{\mathrm{mc}}}$, since then $D'_{\gamma, \delta}$ does not depend on $\mathcal Y_\delta$. Next, we are going to show that $\log D'_{\gamma, \delta, \mathbf x_\delta} (u,v) - \log D'_{\gamma, \delta, \mathbf x'_\delta} (u,v)$ is bounded by $O(1) \| \mathbf x_\delta - \mathbf x'_\delta \|_\infty$, see \eqref{eq-distance-Lip} below.

Let $\mathcal A_\delta$ be such that $\{ \mathcal X_\delta \in \mathcal A_\delta \} = \{ \mbox {each cell in }\mathcal V_\delta \mbox{ has side length at least }\delta^{C_{\mathrm{mc}}}  \}$. Let $\iota$ be an arbitrarily small positive number and $\alpha>0$, and let $\mathcal E^*_{\delta, \iota, \alpha}  =  \{  ({\cal X}_\delta, {\cal Y}_\delta) \in \tilde{\mathcal E}^*_{\delta, \iota, \alpha} \} $ be the event such that
$$
 |\frac{\log D_{\gamma, \delta'}(u, v)}{\log \delta^{-1}} - \frac{\log D'_{\gamma, \delta''}(u, v)}{\log \delta^{-1}} | \leq \iota/4 \mbox{ for all } \delta^{1 + \iota / \alpha} \leq \delta', \delta'' \leq \delta^{1-\iota/\alpha} .
 $$
It will be convenient in what follows to write 
\[\mathcal E^*_{\delta,\iota,\alpha,\mathbf x_\delta}=
\{\bf y_\delta: (\bf x_\delta,\bf y_\delta)\in \tilde{\mathcal E}^*_{\delta,\iota,\alpha}\}.  \]

By Proposition~\ref{prop-approximate-LGD} 
and Lemmas~\ref{lem-partition-minimal-cell} 
and \ref{lem-approximate-LGD-two-deltas}, we can choose an $\alpha>0$ depending only on $\gamma$, such that for any arbitrarily small $\iota>0$,
$\mathcal E^*_{\delta, \iota, \alpha}$ occurs with $c \cdot \iota$-high probability.
As a result, we see that there exists a set $\mathcal A \subseteq \mathcal A_\delta$  such that 
 $$\P(  \mathcal E^*_{\delta, \iota, \alpha} \mid \mathbf X_\delta) \geq 0.9   \mbox{ on the event } {\cal X}_\delta \in \mathcal A,  \mbox{which occurs with  $c \cdot \iota$-high probability} \,.$$
In particular, for ${\bf x}_\delta, {\bf x}_\delta'\in \mathcal A$,
$\mathcal E^*_{\delta,\iota,\alpha,\mathbf x_\delta}\cap
\mathcal E^*_{\delta,\iota,\alpha,\mathbf x_\delta'}$ is non-empty.

Let $\ell = \|\mathbf x_\delta - \mathbf x'_\delta\|_\infty$. We see from Lemma~\ref{lem-LGD-compare} that as long as $\ell \leq \ell_\delta =  \frac{\iota \log \delta^{-1}}{2\gamma \alpha}$ we have 
$$
D_{\gamma, \delta^{1-\iota/\alpha} ,\mathbf x_\delta}(u, v)\leq D_{\gamma, \delta, \mathbf x'_\delta}(u, v) \leq D_{\gamma, \delta^{1+\iota/\alpha} ,\mathbf x_\delta}(u, v)\,.
 $$
(Note that the above is an inequality between random variables that depend on ${\cal Y}_\delta$, 
which holds for almost all configurations $\mathbf y_\delta$.) 
Consequently, on the event ${\cal Y}_\delta \in \mathcal E^*_{\delta,\iota,\alpha,\mathbf x_\delta}\cap
\mathcal E^*_{\delta,\iota,\alpha,\mathbf x_\delta'}$ we have 
 $|\log D_{\gamma, \delta, \mathbf x'_\delta}(u, v) - \log D_{\gamma, \delta ,\mathbf x_\delta}(u, v)| \leq 
\frac 1 2 \iota \log \delta^{-1}$ 
and thus, 
\begin{equation}\label{eq-distance-Lip}
|\log D'_{\gamma, \delta, \mathbf x'_\delta}(u, v) - \log D'_{\gamma, \delta ,\mathbf x_\delta}(u, v)| \leq \iota \log \delta^{-1}\,.
\end{equation}
Recall that,  for all $\mathbf x_\delta$, $D'_{\gamma,\delta,\bf x_\delta}$ does not depend on ${\cal Y}_\delta$.  Then, we have deduced that \eqref{eq-distance-Lip}
holds for all $\bf x_\delta, \bf x'_\delta\in \mathcal A$ satisfying  $\ell \leq \ell_\delta$.
 
At this point, we are ready to deduce our concentration result. 
Let $d'_{u,v}$ be the minimal number such that 
$$\P({\cal X}_\delta \in \mathcal A') \geq 1/2\,, \mbox{ where } \mathcal A' = \{ \mathbf x_\delta \in \mathcal A: D'_{\gamma, \delta, \mathbf x_\delta}(u, v) \leq d'_{u, v}\}\,.$$
Note that the  above is well defined since when $\mathbf x_\delta \in \mathcal A$, we have that $ D'_{\gamma, \delta, \mathbf x_\delta}(u, v)$ is a measurable function 
of $\mathbf x_\delta$. Recalling 
\eqref{eq-distance-Lip}, we see that for $c=c(\gamma)>0$
\begin{equation}\label{eq-upper-tail-deviation}
\P(\log D'_{\gamma, \delta}(u, v) \geq \log d'_{u, v} +  \iota \log \delta^{-1}) \leq \P(\mathcal X_\delta \not \in \mathcal A) +
 \P(\min_{\mathbf x_\delta' \in \mathcal A'} \|{\mathcal X_\delta} - \mathbf x'_\delta\|_\infty \geq \ell_\delta) \leq \delta^{c \iota^2 } \,,
 \end{equation}
where in the last step we have used Lemma~\ref{lem-Gaussian-concentration}, as well as the fact that maximal individual variance of the random variables in
 ${\mathcal X}_\delta$ is $O_{C_{\mathrm{mc}}}(\log \delta^{-1})$. 
 
 By a similar reasoning, we can also get that
\begin{equation}\label{eq-lower-tail-deviation}
\P(\log D'_{\gamma, \delta}(u, v) \leq \log d'_{u, v} -  \iota \log \delta^{-1}) \leq \delta^{c \iota^2 } \,.
\end{equation}

   Due to the uniform square integrability of
$\log D'_{\gamma, \delta}(u, v)/\log(1/\delta)$, which follows from $|D'_{\gamma, \delta}| \le |\mathcal V_\delta|$ and the reasoning in  Lemma~\ref{lem-partition-minimal-cell},
we conclude from \eqref{eq-upper-tail-deviation} and \eqref{eq-lower-tail-deviation} that $| \E \log D'_{\gamma,\delta}(u,v)  - \log d'_{u, v}| \leq 2\iota \log \delta^{-1}$.
Combined with \eqref{eq-upper-tail-deviation} and \eqref{eq-lower-tail-deviation}, this completes the proof of
\eqref{eq-concentration-approximate} (we adjust the value of $\iota$ appropriately).

\medskip

\noindent {\bf Proof of \eqref{eq-concentration-2}}. We now sketch the necessary modifications in order to prove \eqref{eq-concentration-2}. For simplicity of exposition, in what follows we will repeatedly use higher powers of $\log \delta^{-1}$ to absorb error terms with lower powers of $\log \delta^{-1}$.
It is obvious from Proposition~\ref{prop-approximate-LGD} and Corollary \ref{cor-approximate-LGD-expectation} that \eqref{eq-concentration-2} can be deduced from the statement that  with probability at least $1- e^{(\log \delta^{-1})^{0.8}}$,
\begin{equation}\label{eq-concentration-approximate-2}
|\log D'_{\gamma, \delta}(u, v) - \E \log D'_{\gamma, \delta}(u, v) | \leq (\log \delta^{-1})^{0.94}\,.
\end{equation}
 To prove \eqref{eq-concentration-approximate-2}, we follow the proof of \eqref{eq-concentration-1}, but in place of  $\mathcal E^*_{\delta, \iota, \alpha}$  we define $\mathcal E^*_{\delta, \alpha}$ to be the event that
 $$
|\frac{\log D_{\gamma, \delta'}(u, v)}{\log \delta^{-1}} - \frac{\log D'_{\gamma, \delta''}(u, v)}{\log \delta^{-1}} | \leq (\log \delta^{-1})^{-0.09} \mbox{ for all } \delta e^{-\alpha^{-1}(\log \delta^{-1})^{0.9}} \leq \delta', \delta'' \leq \delta e^{\alpha^{-1}(\log \delta^{-1})^{0.9}} .
    $$
By Proposition~\ref{prop-approximate-LGD} 
and Lemmas~\ref{lem-partition-minimal-cell} 
and \ref{lem-approximate-LGD-two-deltas}, we can choose an $\alpha>0$ depending only on $\gamma$, such that 
$\P( \mathcal E^*_{\delta, \alpha}) \geq 1-  \delta^{1/\alpha}$. 
As a result, we see that there exists a set $\mathcal A \subseteq \mathcal A_\delta$  such that 
$$\P(  \mathcal E^*_{\delta, \alpha} \mid \mathbf X_\delta) \geq 0.9 \mbox{ on the event } {\cal \mathcal X_\delta} \in \mathcal A, \mbox{ and  }\P( \mathcal X_\delta \in \mathcal A\mbox)\geq 1- \delta^{\frac{1}{2\alpha}}\,.$$
At this point, we can repeat the analysis as for \eqref{eq-concentration-1} and deduce that for $\ell_\delta = (\log \delta^{-1})^{0.9}$
$$\P(\log D'_{\gamma, \delta}(u, v) \geq \log d'_{u, v} +  (\log \delta^{-1})^{0.95}) \leq \P( \mathcal X_\delta \not \in \mathcal A) + \P(\min_{\mathbf x_\delta' \in \mathcal A'} \|
\mathcal X_\delta - \mathbf x'_\delta\|_\infty \geq \ell_\delta) \leq e^{-\Omega((\log \delta^{-1})^{0.8})}\,,$$
where in the last step we again have used Lemma~\ref{lem-Gaussian-concentration}, as well as the fact that the maximal individual variance of the random variables in $\mathcal X_\delta$  is $O_{C_{\mathrm{mc}}}(\log \delta^{-1})$. The proof of the lower
deviation in \eqref{eq-concentration-approximate-2} 
is similar, leading to 
\eqref{eq-concentration-approximate-2} and thus completing the proof of \eqref{eq-concentration-2}. 
\end{proof}

\section{Liouville heat kernel}
\label{sec-LHK}
In this section, we relate the Liouville heat kernel to the Liouville graph distance.
\subsection{Lower bound}
\label{subsec-LHKLB}
In this section, we  provide a lower bound on the Liouville heat kernel in terms of the Liouville graph distance. For $u, v\in \mathbb V$, we denote 
$$\chi^+_{u, v} = \limsup_{\delta \to 0} \frac{\E \log D_{\gamma, \delta}(u, v)}{\log \delta^{-1}}\,.$$
Recalling Lemma~\ref{lem-obvious-bounds}, we see that  $0<\chi^+_{u, v} <1$.
We will show that there exists a finite random variable $C>0$ (measurable with respect to the GFF, and depending on $u,v$) such that for all  $t\in (0,1]$,
\begin{equation}\label{eq-heat-kernel-lower-bound-fancy}
\pe_t^{\gamma}(u, v) \geq C \exp\Big\{-t^{-\frac{\chi^+_{u, v}}{2 - \chi^+_{u, v}} +o(1)}\Big\}\,.
\end{equation}
In order to prove \eqref{eq-heat-kernel-lower-bound-fancy}, it suffices to show that there exists a $t_0>0$ (deterministic) so that
for any arbitrarily small and fixed  $\iota>0$,
there exists a  small positive random variable $c = c_{\gamma, u, v,\iota} >0$, measurable on the GFF,
such that
for all $t\in (0,t_0]$, the following 
holds: with probability at least $1-  e^{-(\log t^{-1})^{0.2}} $,
\begin{equation}\label{eq-heat-kernel-lower-bound}
\pe_s^{\gamma}(u, v) \geq  c\exp\Big\{-t^{-\frac{\chi^+_{u, v}}{2 - \chi^+_{u, v}} -\iota}\Big\} \mbox{ for all } t\leq s\leq 2t\,.
\end{equation}
Indeed, \eqref{eq-heat-kernel-lower-bound} yields \eqref{eq-heat-kernel-lower-bound-fancy} for $t\leq t_0$ by an  application of the Borel-Cantelli lemma for times $t_i=2^{-i}$. On the other hand, \eqref{eq-heat-kernel-lower-bound-fancy} holds for $t>t_0$  by the Markov property and multiple applications of \cite[Corollary 5.20]{MRVZ14}.

To show \eqref{eq-heat-kernel-lower-bound},  fix an arbitrarily small $\iota>0$ and let $\delta = t^{{1}/(2 - \chi^+_{u, v}) + \iota}$.  Also,
throughout the section, 
we use $\hat {\mathsf C}$ to denote a cell in $\mathcal V_\delta$, 
while $\mathsf C$ will stand for the boxes $\{ B_i \}$ in Lemma~\ref{lem-partition-independence}.

A natural approach to proving  \eqref{eq-heat-kernel-lower-bound}
is to show that with not too small probability,
 the Liouville Brownian motion can cross each cell in $\mathcal V_\delta$ without accumulating too much ``Liouville time'' (i.e., the PCAF as defined in \eqref{eq-def-PCAF}), provided with which one can then force the SBM to travel along the geodesic between $u$ and $v$ in $\mathcal V_\delta$.  However, there is a substantial obstacle due to the the possibility that two neighboring cells along the geodesic may have side lengths differing by a factor as large as a power in $\delta$. This is further complicated by
a technical challenge: for a cell $\hat {\mathsf C} \in \mathcal V_\delta$, the Liouville time accumulated during traveling through $\hat {\mathsf C}$ depends on the starting and ending points, and we do not expect uniform bounds on that.

We now discuss how to address these challenges; a crucial role is played by
Lemma~\ref{lem-partition-independence}.
We work on the event $\mathcal E_1$ defined as
\begin{equation}
\label{eq-E1} \mathcal E_1 = 
\mathcal E_{\delta, \alpha^*} \cap \mathcal E_{\delta, \alpha^*, u, v} 
\cap \{  \mbox{\rm \eqref{eq-tilde-h-eta-assump} holds}   \} \,,
\end{equation}
where $\mathcal E_{\delta, \alpha^*}$ and $\mathcal E_{\delta, \alpha^*, u, v}$
are  defined 
in \eqref{eq-def-E-delta-alpha} and 
\eqref{eq-def-E-delta-alpha*-u-v}, respectively.  Note that $\P(\mathcal E_1)\geq
1-e^{-(\log \delta^{-1})^{0.24}}$, by Lemmas~\ref{lem-neighboring-cell}, \ref{lem-regularity} and the discussion above \eqref{eq-tilde-h-eta-assump}. We next will extract a sequence of neighboring boxes using Lemma~\ref{lem-partition-independence}. To ensure more desirable properties of this sequence of boxes, we will work on a more restricted event than $\mathcal E_1$.
By Propositions~\ref{prop-approximate-LGD} and \ref{prop-concentration}, we see that with $c \cdot \iota^2$-high probability, $D'_{\delta, \gamma}(u, v) \leq \delta^{-\chi^+_{u, v} - \frac 1 4 (2 - \chi^+_{u,v})^2 \cdot \iota}$. 
Setting \begin{equation} \label{eq-E2}
\mathcal E_2=\mathcal E_1\cap \{D'_{\delta, \gamma}(u, v) \leq \delta^{-\chi^+_{u, v} - \frac 1 4 (2 - \chi^+_{u,v})^2 \cdot \iota}\} \cap \{ \mbox{the event in \eqref{Eq.error-coarse-fine}} \} ,
\end{equation}
we deduce, using arguments similar to those employed 
 in the  proof of Lemma~\ref{lem-neighboring-cell}, that with high probability we have 
\begin{equation} \label{Eq.error-coarse-fine}
\max_{\hat{\mathsf C} \in \mathcal V_\delta} \max_{x\in \mathbb V \cap \hat{\mathsf C}_{\mathrm{large}}} |\eta_{(\epsilon^*)^2 s_{\hat{\mathsf C}}}^{s_{\hat{\mathsf C}}}(x)| \leq (\log \delta^{-1})^{0.8}.
 \end{equation}
and therefore,
$\P(\mathcal E_2) \geq 1-e^{-(\log \delta^{-1})^{0.23}}$. 

We work on $\mathcal E_2$ in what follows. Denote the sequence $\{ B_i \}$ provided
in Lemma~\ref{lem-partition-independence} by 
$\mathcal C =  ( \mathsf C_1, \ldots, \mathsf C_d )$; recall that this sequence is measurable with respect to  $\mathcal V_\delta$, and joins $u$ to  $v$. 
Then, $d\leq \delta^{-\chi^+_{u, v} - \frac 1 2(2 - \chi^+_{u,v})^2 \cdot \iota}$, and each $\mathsf C_i$ satisfies $M_{\gamma, s_{\mathsf C_i}} (\mathsf C_i) \leq \delta^2 e^{O( (\log \delta^{-1})^{0.8})}$ (recall \eqref{Eq.error-coarse-fine} and that
$s_{\mathsf C_i} = (\epsilon^*)^2 s_{\hat {\mathsf C_i}}$, where $\hat {\mathsf C_i}$ is the cell containing $\mathsf C_i$, see Lemma~\ref{lem-partition-independence}).  Furthermore, the 
 law of  $\{\eta_{\delta'}^{s_{\mathsf C_i}}(x): \delta' < s_{\mathsf C_i}, x \in (\mathsf C_i)_{\mathrm{large}} \mbox{ for some } \mathsf C_i \in \mathcal C \}$
conditioned on $\mathcal V_{\delta}$ coincides with
its unconditional version. (Here, we  abuse notation by using
 $\mathsf C_i$ to denote a dyadic box which is not necessarily a cell. The 
abuse of notation is justified by the fact that 
$M_{\gamma, s_{\mathsf C_i}} (\mathsf C_i) \leq \delta^2 e^{O( (\log \delta^{-1})^{0.8})}$ and thus the $\mathsf C_i$'s will essentially play the role of cells.)
For $i=1, \ldots, d-1$, denote for brevity $s_i:=s_{\mathsf C_i}$ and write
 $\Lambda_i = \partial \mathsf C_i \cap \partial \mathsf C_{i+1}$.  We emphasize that the $\Lambda_i$'s are measurable with respect to $\mathcal V_\delta$. 
  As discussed above, we will  force the SBM to travel 
	through $\mathsf C_1, \ldots, \mathsf C_d$ sequentially, and will show that this occurs with high enough probability.
	To this end, we will crucially use the fact  $\mathcal C$ is a good sequence, and thus
 \begin{equation}\label{eq-Lambda-i-not-small}
 \mathcal L_1 (\Lambda_i), \mathcal L_1(\Lambda_{(i-1)\vee 1})\geq 
\epsilon^* s_i \mbox{ for } 1\leq i < d\,.
 \end{equation}
Here $\mathcal L_1$ is the 1-dimensional Lebesgue measure, and $\epsilon^*$ is defined in \eqref{eq-def-epsilon*}.

Consider $2\leq i\leq d-1$. For  $\mathsf C_i \in \mathcal C$ and $z\in \mathsf C_i$, we say that $z$ is a fast point (with respect to
$\mathsf C_i$)
if for any $\Lambda \subseteq \Lambda_i$ such that $\mathcal L_1(\Lambda) \geq  0.1 \mathcal L_1(\Lambda_i)$ one has
\begin{equation}\label{eq-def-fast}
P_z(F(\sigma_{\Lambda}) \leq \delta^2 C_\delta) \geq \exp\{-\exp\{(\log \delta^{-1})^{2/3}\}\}=:p_{\mathrm{fast}}(=\exp\{-(1/\delta)^{o(1)}\})\,,
\end{equation}
where $\sigma_\Lambda$ is the first time when the SBM hits $\Lambda$ and
\begin{equation}\label{eq-def-err-3}
C_\delta = \exp\{(\log \delta^{-1})^{0.95}\} (= (1/\delta)^{o(1)})\,.
\end{equation} 
Note that we allow $z\in \partial \mathsf C_i$, however the fact that being fast involves considering all possible $\Lambda$ makes the notion non-trivial even for a point $z\in \partial \mathsf C_i$, since  we need to consider sets $\Lambda$ with $z\not\in \Lambda$.
We say that  $\mathsf C_i$ is fast if
\begin{equation}\label{eq-def-cell-fast}
\mathcal L_1(\Lambda_{i-1, \mathrm{fast}}) \geq 0.1 \mathcal L_1(\Lambda_{i-1}) \mbox{ where } \Lambda_{i-1, \mathrm{fast}} = \{z\in \mathsf \Lambda_{i-1}: z \mbox{ is fast\ with respect to $\mathsf C_i$}\}\,.
\end{equation}
A crucial ingredient for the proof of \eqref{eq-heat-kernel-lower-bound} is the proof that with high probability  all the  $\mathsf C_i$'s are fast simultaneously. To this end, we now estimate the probability that a particular  $\mathsf C_i$ is fast. (We will later apply a union bound.)  

\begin{lemma}\label{lem-fast-cell}
There exists a $\delta_0>0$ such that 
for all $\delta<\delta_0$ there exists an event  $\mathcal E_3$ of
high probability such that
the following holds. For each $2\leq i\leq d-1$ there exists an event 
$\mathcal E_{\mathsf C_i}$ such that 
 $$
\P(\mathcal E_{\mathsf C_i} \mid \mathcal V_\delta) \geq 1 - \exp \{ - 2^{\sqrt{\log \delta^{-1}}} \} \mbox{ and } ( \mathcal E_{\mathsf C_i} \cap \mathcal E_3) \subset \{ \mathsf C_i \mbox{ is fast} \} . 
 $$
 \end{lemma}
\noindent
(The event $\mathcal E_3$ is defined in \eqref{eq-E3} below.)

We begin our preparation for the proof of Lemma \ref{lem-fast-cell}.
Since our goal is to show that with very high probability $\mathsf C_i\in \mathcal C$ is fast, 
a first (or second) moment computation will not be enough. Instead, we will use a simple multi-scale analysis and employ a percolation  argument. We first introduce some definitions.
Set $k = \lfloor (\log \delta^{-1})^{0.51}\rfloor$ and $K = 2^k$.
Take $\mathsf C\in \mathcal C$ and partition it into $K^2$ many dyadic squares with side length $s_{\mathsf C} /K$. Denote the collection of these boxes by $\mathcal B_{\mathsf C}$, 
and denote by $\mathsf B_{\mathrm{large}}$ (respectively, 
 $\mathsf B_{\mathrm{Large}}$)  the boxes concentric with  $\mathsf B$ but with double 
(respectively, triple) side length. For a fixed 
$\mathsf B\in \mathcal B_{\mathsf C}$ and $z\in \mathsf B$, we say that
$z$ is a pre-fast point with respect to the box $\mathsf B$ if for any $\Lambda \subseteq \partial  \mathsf B$ with $\mathcal L_1(\Lambda) \geq 10^{-5}s_{\mathsf C}/K$ one has
\begin{equation}\label{eq-def-pre-fast}
P_z(F(\sigma_{\Lambda}) \leq \delta^2 C_\delta K^{-4}; \sigma_{\Lambda} \leq \sigma_{\partial \mathsf B_{\mathrm{Large}}}) \geq \exp\{- (\log \delta^{-1})^{1/10}\}\,.
 \end{equation}
Note that the notions of
pre-fast and fast are related,
 but one does not necessarily imply the other. We say that 
$\mathsf B$ is pre-fast if the subset of pre-fast points  with respect to $\mathsf B$ on $\partial \mathsf B$ has 1-dimensional Lebesgue measure at least $(1 - 10^{-5})\mathcal L_1(\partial \mathsf  B)$. By definition, the property of boxes being pre-fast  
has long range correlation, though we expect that the correlation decays quickly. 

 \begin{figure}[h]
\vspace{-0.5cm}\hspace{2cm}  \includegraphics[width=12cm]{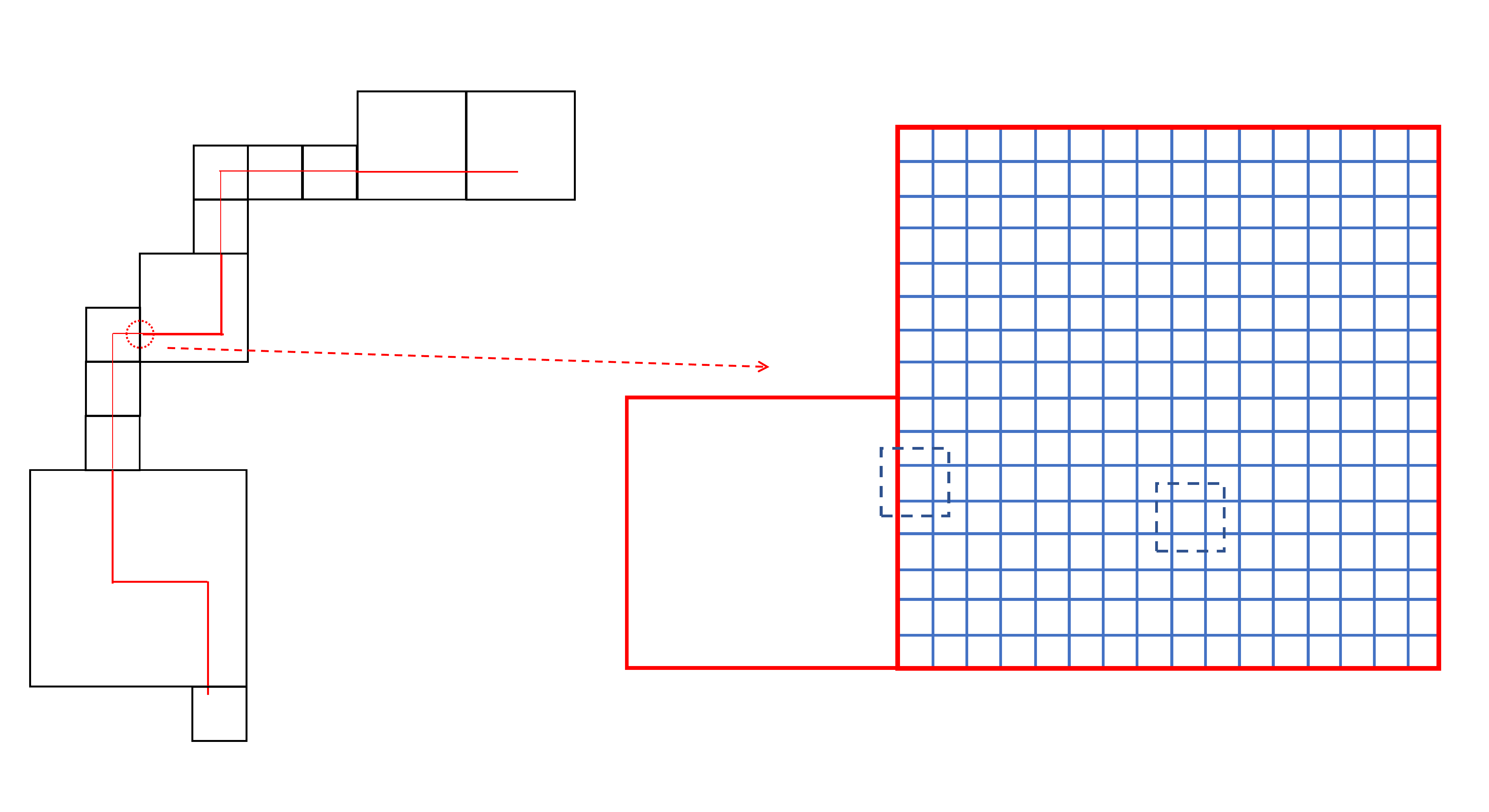}
\\  \vspace{-1cm}  \caption{In the left  picture, the (black) boxes stand for (a piece of) the sequence of good cells joining $u$ and $v$. The (red) line stands for the good sequence of boxes $\{ \mathsf B_i \}$ in Lemma~\ref{lem-partition-independence}, which are denoted by $\{ \mathsf C_i \}$ now. The right picture is a zoom in, where the big 
(red) box is $\mathsf C = \mathsf C_i$, and the small (blue) boxes form $\mathcal B_{\mathsf C}$. }
 \label{fig-Lemma41} 
 \end{figure}

 In order to control the correlation, we define a field $\tilde \eta^{\mathsf B} : = \{\tilde \eta_{\epsilon'}^{\mathsf B, s_{\mathsf C}} (z): \epsilon',  z \}$ by
 \begin{equation}\label{eq-def-tilde-eta}
 \tilde \eta_{\epsilon'}^{\mathsf B, s_{\mathsf C} }(z) : = \left\{ \begin{array}{ll} \sqrt{\pi}\int_{\mathbb V \times ((\epsilon')^2, s_{\mathsf C}^2 )}p_{_{\mathsf B_{ \mathrm{Large}}}}(s/2; z, w)W(dw, ds), & \mbox{if } z \in \mathsf B_{\mathrm{large}} \mbox{ and } \epsilon' < s_{\mathsf C}, \\ 0, & \mbox{otherwise,} \end{array} \right.  
 \end{equation}  
where $p_{_{\mathsf B_{\mathrm{Large} } }} (s/2; z, w)$ is the transition density for SBM truncated upon exiting the box $\mathsf B_{\mathrm{Large} }$. A derivation
similar  to \eqref{eq-tilde-h-eta-assump} yields that with high probability we have
 \begin{equation}\label{eq-assumption-in-lower-heat-kernel}
 \max_{\mathsf C \in \mathcal C} \max_{\mathsf B \in \mathcal B_{\mathsf C}, z \in \mathsf B_{\mathrm{large}} } \ \max_{\epsilon' <  s_{\mathsf C},  \log_2 \epsilon' \in \mathbb Z}  |\tilde \eta_{\epsilon'}^{\mathsf B,  s_{\mathsf C}}   (z)   -\eta_{ \epsilon'}^{s_{\mathsf C}} (z)  | = O(\sqrt{\log \delta^{-1}})\,.
 \end{equation}


With $\mathcal E_2$ as in \eqref{eq-E2}, let $\mathcal E_3$ be defined by 
\begin{equation}
\label{eq-E3}
\mathcal E_3=\mathcal E_2 \cap 
\{ \mbox{the event in \eqref{eq-assumption-in-lower-heat-kernel} holds 
} \} \,.
\end{equation}
Since $\P(\mathcal E_2) \geq 1-e^{-(\log \delta^{-1})^{0.23}}$, we have that
$\P(\mathcal E_3) \geq 1-e^{-(\log \delta^{-1})^{0.22}}$; in the sequel, we
work  on ${\mathcal E}_3$.
For a SBM $X_\cdot$ started at a point $z$ in $\mathsf B$, define 
\begin{equation}\label{eq-def-PCAF-approx}
\tilde F_{\mathsf B} (r) :=  \lim_{n\to \infty}\int_0^r \exp\{\gamma \tilde \eta_{2^{-n}}^{\mathsf B, s_{\mathsf C}}(X_{r'}) - \frac {\gamma^2}2 \Var( \tilde \eta_{2^{-n}}^{ \mathsf B,  s_{\mathsf C}} (X_{r'}))\} d r' ,
\end{equation}
where the existence of the limit follows from the same martingale argument 
yielding the existence of the original PCAF  (see \cite{GRV13}). 
On the event ${\mathcal E}_3$ we work on, for any stopping time 
$\tau$ so that $X_{r} \in \mathsf B_{\mathrm{large}}$ 
for all $0\leq r\leq \tau$, we have
\begin{equation}\label{eq-PCAF-approx}
F(\tau) \leq \tilde F_{\mathsf B}(\tau) \delta^2 s_{\mathsf C}^{-2}  \exp\{(\log \delta^{-1})^{0.91}\}\,.
\end{equation}
We note that $\tilde F_{\mathsf B} (r)$ is measurable with respect to
the SBM $X.$ and the field $\tilde \eta^{\mathsf B}$, for which Lemma~\ref{lem-partition-independence} is also valid (see \eqref{Eq.fine-field-independent} and note that
$\mathsf B_{\mathrm{large}} \subset \mathsf C_{\mathrm{large}}$).
The following lemma is the key to the proof of Lemma~\ref{lem-fast-cell}.
It in particular implies that the events that geometrically 
separated boxes  are pre-fast stochastically dominate a sequence of 
i.i.d.\ Bernoulli indicators.  In what follows, we denote for brevity 
$\mathcal B_i : = \mathcal B_{\mathsf C_i}$.
\begin{lemma}\label{lem-ij-pre-fast}
For $\mathsf B\in \mathcal B_i$, there exists an 
event $\mathcal E_{\mathsf B, \mathrm{prefast}}$ which is measurable with respect to the field $\tilde \eta^{\mathsf B}$, such that
 \begin{equation*}
 \P(\mathcal E_{\mathsf B, \mathrm{prefast}}  \mid \mathcal V_\delta  ) \geq 1- O( K^{-2}) \mbox{ and }  ( \mathcal E_{\mathsf B, \mathrm{prefast}} \cap \mathcal E_3 )   \subseteq \{\mathsf B \mbox{ is pre-fast}\}\,.
 \end{equation*}
\end{lemma}
\begin{proof}
 Let $\mathsf B_{\mathrm{small}}$ 
denote the  box concentric with
 $\mathsf B$,  of half the side length.
Let $\sigma_{\partial \mathsf B_{\mathrm{small}}}$ (respectively 
$\sigma_{\partial \mathsf B_{\mathrm{large}}}
 $) be the hitting time of $\partial \mathsf B_{\mathrm{small}}$ (respectively, $\partial \mathsf B_{\mathrm{large}}$) by the SBM.
Let $\tau$ be the first  hitting time of $\partial {\mathsf B}$ after 
$\sigma_{\partial \mathsf B_{\mathrm{small}}}$. Define $\mathcal E = \{\sigma_{\partial \mathsf B_{\mathrm{small}}} \leq \sigma_{\partial \mathsf B_{\mathrm{large}}}, \tau \leq s_i^2 K^{-2}\}$. From standard properties of the
SBM we have that that $P(\mathcal E) \geq 10^{-4}$ and that 
 \begin{equation}\label{eq-exiting-measure}
 P_z(X_\tau \in \Lambda \mid \mathcal E) \geq 10^{-10}\,,
 \end{equation} 
 for any $z \in \partial \mathsf B$ and  $\Lambda \subseteq \partial \mathsf B$ with 1-dimensional Lebesgue measure $\mathcal L_1(\Lambda) \geq 10^{-5} s_i/K$. Write
$\tilde F$ for $\tilde F_{\mathsf B}$. A straightforward computation yields that
 $$
\mathbb E (E_z ( \tilde F (\tau) \mid \mathcal E)  \mid \mathcal V_\delta  ) \leq 10^4 E_z \mathbb E ( \tilde F  (s_i^2 K^{-2})  \mid \mathcal V_\delta ) \leq 10^4 s_i^2 K^{-2} \,,
 $$
where we used Lemma~\ref{lem-partition-independence} for $\tilde \eta^{\mathsf B}$ in the second inequality.
Therefore, by Markov's inequality, we see that 
 $$ 
\P(P_z(\tilde F(\tau)  \geq  s_i^2 \mid \mathcal E) \geq 10^{-11}  \mid \mathcal V_\delta) \leq O(K^{-2})\,. 
 $$ 
Combining the preceding inequality with \eqref{eq-exiting-measure} and using the fact that 
 $$P_z( \tilde F(\tau)\leq {s_i^2}, X_\tau \in \Lambda, \mathcal E) \geq P_z(\mathcal E) (P_z(X_\tau \in \Lambda \mid \mathcal E)  -P_z(\tilde F(\tau)  \geq s_i^2 \mid \mathcal E) )\,,$$
 we get that for any $\Lambda \subseteq \partial \mathsf B$ with $\mathcal L_1(\Lambda) \geq 10^{-5} s_i/K$
 \begin{align*}
 \P ( P_z(\tilde F(\tau)\leq s_i^2 , X_\tau \in \Lambda, \mathcal E) \geq 10^{-15}  \mid \mathcal V_\delta) \geq 1- O(K^{-2})\,.
 \end{align*}
Combined with \eqref{eq-PCAF-approx}, this yields that 
 \begin{equation}\label{eq-z-pre-fast}
 \P(\mathcal E_{z, \mathrm{fast}}  \mid \mathcal V_\delta  ) \geq 1- O(K^{-2})  \mbox{ and }  ( \mathcal E_{z, \mathrm{fast}} \cap \mathcal E_3)   \subseteq \{z \mbox{ is pre-fast}\} \,,
 \end{equation} 
where $\mathcal E_{z, \mathrm{fast}} : = \{ P_z(\tilde F(\tau)\leq s_i^2 , X_\tau \in \Lambda, \mathcal E) \geq 10^{-15} \} $ is measurable with respect to the field  $\tilde \eta^{\mathsf B}$.
Another application of Markov's inequality concludes the proof of the lemma.
  \end{proof}

 \begin{proof}[Proof of Lemma~\ref{lem-fast-cell}]
In what follows, we work conditionally on $\mathcal V_\delta$.
Fix $i$. Recall that $\mathcal B_i$ denotes the partition of $\mathsf C_i$ into $K^2$ boxes of side length $s_{\mathsf C_i} / K$, where $K=2^{\lfloor (\log \delta^{-1})^{0.51} \rfloor}$. Correspondingly, $\partial \mathsf C_i$ is partitioned into $4 K$ segments, whose collection is denoted by $\mathbb {BS}$. For $L \in \mathbb {BS}$, let $\mathsf B^L$
denote the unique box in $\mathcal B_i$ containing $L$. Set $\mathbb {BS}_{i, A} = \{ L \in \mathbb {BS} : L \subset A\}$ for all $A \subset \partial \mathsf C_i$.

For any $\Lambda\subseteq  \Lambda_i$ with $\mathcal L_1(\Lambda) \geq 0.1 \mathcal L_1(\Lambda_i)$, we define
 \begin{equation}\label{eq-def-Lambda-star}
\mathbb L = \mathbb L_\Lambda : = \{ L \in \mathbb {BS}_{i, \Lambda_i} : \mathcal L_1(L \cap \Lambda) \geq 10^{-5} s_i/K \},
 \end{equation}
and set
\begin{align*}
&\mathbb L' = \mathbb L'_\Lambda =\\
& \{ L \in \mathbb {BS}_{i, \Lambda_{i-1}} : \mbox{$L$ is connected to (some segment in) $\mathbb L$ by a path of neighboring pre-fast boxes} \}.
\end{align*}
Let $\Lambda' = \cup_{L \in \mathbb L'} L$, 
and introduce the event
$$
\mathcal A  = \{\mbox{$\mathcal L_1 (\Lambda') \ge 0.2 \mathcal L_1 (\Lambda_{i-1})$ for any $\Lambda\subseteq  \Lambda_i$ with $\mathcal L_1(\Lambda) \geq 0.1 \mathcal L_1(\Lambda_i)$} \}.
 $$
 The event $\mathcal A$ ensures 
 that any not-so-small subset $\Lambda$ of $\Lambda_i$ is connected with a not-so-small subset (i.e. $\Lambda'$) of $\Lambda_{i-1}$ by pre-fast boxes.
The heart of the proof of the lemma consists of showing the following statement:
\begin{equation} \label{eq-par}
\P(\mathcal A |\mathcal V_\delta)\geq 1 - e^{-2^{\sqrt{\log \delta^{-1}}}}.
\end{equation}
We postpone the proof of \eqref{eq-par} and complete the proof of the lemma, assuming its validity.
Take 
 $$
\Lambda'_{\mathrm{prefast}}= \cup_{L \in \mathbb L'} \{z \in L : z \  \mbox{ is pre-fast with respect to $\mathsf B^L$} \}.
 $$ 
Note that on $\mathcal A$,
\begin{eqnarray}
\mathcal L_1(\Lambda'_{\mathrm{prefast}}) 
 & \geq &
\mathcal L_1(\Lambda') - \mathcal L_1 (\cup_{L \in \mathbb L'} \{z \in L : z \mbox{ is not pre-fast with respect to } \mathsf B^L \} ) \nonumber
 \\ & \geq & 
0.2 \mathcal L_1(\Lambda_{i-1}) -  \sum_{L \in \mathbb L'} 10^{-5} \mathcal L_1 (\partial \mathsf B^L) \geq 0.1\mathcal L_1(\Lambda_{i-1})\,, \label{Eq.measure-of-Lambda-prime}
 \end{eqnarray}
where we have used the fact that $\mathsf B^L$ is pre-fast for all $L \in \mathbb L'$.
In addition, for each $L \in \mathbb L'$ we denote by $\mathsf B_1, \ldots, \mathsf B_\ell$ with $\ell\leq K^2$
 the sequence of pre-fast boxes in $\mathcal B_i$ with from $L$ to $\mathbb L$. For all
$1\leq j\leq \ell-1$, we let  $\Lambda_{i,j}$ denote
 the collection of all pre-fast points  with respect to $\mathsf B_{ j+1}$
	lying on the common boundary of $\mathsf B_{j}$ and $\mathsf B_{ j+1}$.
	We also set $\Lambda_{i, \ell} = \mathsf B_{\ell} \cap \Lambda$, which has 1-dimensional Lebesgue measure larger than  $10^{-5} s_i / K$ by \eqref{eq-def-Lambda-star}.
	Note that
$\mathcal L_1(\Lambda_{i,j}) \geq s_i/(2K)$ for each $1\leq j\leq \ell-1$.
 Consequently, $\mathcal L_1 (\Lambda_{i,j}) \ge 10^{-5} s_i / K$ for all $j$, by the definition of pre-fast boxes and the construction of $\Lambda_{i,j}$'s.
 Define  $\sigma_0 = 0$ and recursively for $1\leq j\leq \ell$,
 $$\sigma_j = \min\{r\geq \sigma_{j-1}: X_r \in \Lambda_{i,j}\}\,.$$
 Applying \eqref{eq-def-pre-fast} repeatedly and using the strong Markov property of SBM together with the definition of $K$, 
we obtain that \eqref{eq-def-fast} holds for $z\in \Lambda'_{\mathrm{prefast}}$, that is, $\Lambda'_{\mathrm{prefast}} \subseteq \Lambda'_{i-1, \mathrm{fast}}$. Since $\mathcal L_1(\Lambda'_{\mathrm{prefast}}) \geq 0.1 \mathcal L_1(\Lambda_{i-1})$, this completes the proof of the lemma, except for the proof 
of \eqref{eq-par}, to which we turn next.
Indeed, we will check that $\P (\mathcal A^c | \mathcal V_\delta) \le e^{-2^{\sqrt{\log \delta^{-1}}}}$. 

Suppose that $\mathcal A$ does not occur. Then there exists a $\Lambda$ such that $\mathcal L_1 (\Lambda) \ge 0.1 \mathcal L_1 (\Lambda_i)$ and moreover $\mathcal L_1 ( \cup_{L \in \tilde {\mathbb L}}  L) \ge 0.8 \mathcal L_1 (\Lambda_{i-1})$, where $\tilde {\mathbb L} = \mathbb {BS}_{i,\Lambda_{i-1}}\setminus \mathbb L'$. By the definition of $\mathbb L$, $\mathcal L_1(\Lambda \setminus \cup_{L \in \mathbb L} L) \leq 10^{-5} \mathcal L_1(\Lambda_i)$, thus $\mathcal L_1(\cup_{L \in \mathbb L} L) \ge \mathcal L_1 (\Lambda) - 10^{-5} \mathcal L_1 (\Lambda_i) \geq 0.05 \mathcal L_1 (\Lambda_i)$. Recalling  \eqref{eq-Lambda-i-not-small}, it follows that $| \mathbb L |, | \tilde{ \mathbb L} | \ge 0.05 \epsilon^* s_i \times K / s_i \ge \lfloor \frac 1 {20} K \epsilon^* \rfloor = : \ell$. Note that $\mathbb L$ is not connected with $\tilde {\mathbb L}$ by pre-fast boxes, by the defintion of $\mathbb L'$. It follows that  on
${\mathcal A}^c$,
 \begin{equation} \label{eq-connectivity-percolation}
    \begin{split}
\mbox{there exist  $\mathcal B_{i,1}\subseteq \mathbb {BS}_{i, \Lambda_i}$ and  $\mathcal B_{i,2} \subseteq \mathbb {BS}_{i, \Lambda_{i-1}}$ with $|\mathcal B_{i,1}|, |\mathcal B_{i,2}| = \ell$,  that }
 \\
    \mbox{are not connected by a sequence of neighboring pre-fast boxes in $\mathsf C_i$.}
    \end{split}
    \end{equation}
Provided with Lemma~\ref{lem-ij-pre-fast}, the desired upper bound on $\P (\mathcal A^c | \mathcal V_\delta)$ follows from a 
		Peierls 
		argument concerning very subcritical percolation with local dependencies. For completeness,
		we provide a proof. 		
By planar duality  there exist $(\mathsf B_{i,1}^j, \mathsf B_{i,2}^j) \subseteq \mathcal B_{\partial \mathsf C_i}$ for all $1\leq j  \leq r$ and some $r \leq \ell$ such that (here $\mathcal B_{\partial \mathsf C_i}$ is the collection of boxes in $\mathcal B_i$ which intersects with $\partial \mathsf C_i$)
 \begin{itemize}
 \item For each $j$, there exists a sequence of $*$-connected boxes $\mathcal B_{i, j, \mathrm{separate}} \subseteq \mathcal B_i$  which starts at $\mathsf B_{i,1}^j$ and ends at $\mathsf B_{i,2}^j$ (two boxes are $*$-connected as long as their intersection is non-empty);
 \item The union of $\mathcal B_{i, j, \mathrm{separate}}$'s separates $\mathcal B_{i,1}$ from $\mathcal B_{i,2}$.
 \item Each box in $\mathcal B_{i, j, \mathrm{separate}}$ for $1\leq j\leq r$ is not pre-fast.
 \item Each box in $\mathcal B_{i, j, \mathrm{separate}}$ for $1\leq j\leq r$ is of $\ell_\infty$-distance at most $4|\mathcal B_{i, j, \mathrm{separate}}| s_i /K$ away from some $\mathsf B_{i, j, \mathrm{separate}} \in \mathcal B_{i, \Lambda_i}\cup \mathcal B_{i, \Lambda_{i-1}}$, where $\mathsf B_{i, j, \mathrm{separate}}$'s are distinct from each other --- this is because each $*$-connected path  (together with $\mathcal B_{\mathsf C_i}$) is supposed to separate at least one box in $\mathcal B_{i, \Lambda_i}\cup \mathcal B_{i, \Lambda_{i-1}}$ which are not separated otherwise.
 \item  $L:= \sum_{j=1}^r L_j\geq \ell$, \mbox{\rm where $L_j=|\mathcal B_{i, j, \mathrm{separate}}|$}.
 \end{itemize}
Therefore, when the total number of boxes is $L$, the number of valid choices for $\mathcal B_{i, j, \mathrm{separate}}$'s is at most 
\begin{equation}\label{eq-enumeration}
N_L  = \sum_{r=1}^{ \ell} \sum_{\sum_{j=1}^rL_j= L }  \binom{2 \ell}{r} \prod_{j=1}^r (4L_j)^2 8^{L_j}
\end{equation}
where 
$\binom{2\ell}{r}$ bounds the number of choices for  $\mathsf B_{i, j, \mathrm{separate}}$'s, $(4L_j)^2$ bounds the number of choices for $\mathsf B_{i,1}^j$ and $\mathsf B_{i,2}^j$, and $8^{L_j}$ bounds the number of choices for the rest of $\mathcal B_{i, j, \mathrm{separate}}$. A straightforward computation then gives that $N_L \leq C^L$ for some constant $C>0$. In addition, the number of choices for  $\mathcal B_{i,1}$ and $\mathcal B_{i,2}$ is at most $\binom{K}{ \ell}^2$.
 Furthermore, since we can choose at least $L/25$ many boxes from $\cup_j\mathcal B_{i, j, \mathrm{separate}}$ whose $\ell_\infty$-distance are at least $2 s_i/K$. Note that the construction of $\tilde \eta^{\mathsf B}$ does not explore the white noise outside the spatial box $\mathsf B_{\mathrm{Large}}$. By Lemmas~\ref{lem-partition-independence} and \ref{lem-ij-pre-fast}, we see that for each such choice the probability for all these boxes in $\cup_j\mathcal B_{i, j, \mathrm{separate}}$ to be not prefast is at most $(C'K^{-2})^{L/25}$ for some absolute constant $C'>0$. Summing over $L \geq \ell$, we see that the probability for the existence of such $\mathcal B_{i,1}$ and $\mathcal B_{i,2}$ is bounded by 
 $$\binom{K}{ \ell}^2 \sum_{L \geq \ell}N_L (C'K^{-2})^{L/25} \leq (10^3/\epsilon^*)^{2\ell} \sum_{L \geq \ell} C^L (C'K^{-2})^{L/25}\leq 2^{-\ell}$$
for $\delta<\delta_0$ where $\delta_0>0$ is a small absolute constant (we used the fact that $K^{\frac 1 {25}} \epsilon^* \ge e^{\sqrt{\log \delta^{-1}}}$ in the last inequality).
Thus,  $\P (\mathcal A^c | \mathcal V_\delta)\leq 2^{-\ell}$.
Since $\ell= \lfloor \frac 1 {20} K \epsilon^* \rfloor\gg \sqrt{\log \delta^{-1}}$, this yields \eqref{eq-par} and completes the proof of the lemma. 
 \end{proof}
 
The next lemma  controls the behavior of the
Liouville Brownian motion near $v$ and $u$. Recall the event $\mathcal E_3$ from \eqref{eq-E3}. 
Recall the notation in the paragraph below \eqref{Eq.error-coarse-fine}, and the
definitions of $p_{\mathrm{fast}}$ and $C_\delta$, see \eqref{eq-def-fast}
and \eqref{eq-def-err-3}, and recall that $\delta = t^{1 / (2 - \chi^+_{u,v} ) + \iota}$.
 \begin{lemma}\label{lem-LBM-near-end-points}
Assume $\delta<\delta_0$. For any  $\iota>0$ small enough and $\beta>0$ fixed large enough, there exist events $\mathcal E_{u, \mathrm{fast}}$ (measurable with respect to $\{ \eta_{\delta'}^{s_d} (x): \delta' < s_d, x \in (\mathsf C_d)_{\mathrm{large}} \}$) and $\mathcal E_{d, \mathrm{fast}}$ having
 $\iota$-high probability with respect to $\P (\cdot | \mathcal V_\delta)$, such that the following holds.  
\begin{enumerate}[(i)]
\item On $\mathcal E_{d, \mathrm{fast}} \cap \mathcal E_3$, there exists $\Lambda_{d, \mathrm{fast}} \subseteq \mathsf \Lambda_{d-1}$  with $\mathcal L_1(\Lambda_{d, \mathrm{fast}} ) \geq 0.1 \mathcal L_1 (\Lambda_{d-1})$ such that 
\begin{align}\label{eq-item-1}
P_z (F(\sigma_{v, \beta}) \leq \delta^{2-\iota} C_\delta) \geq  t^{2 \beta + 1
 }
\mbox{ for all z }\in \Lambda_{d, \mathrm{fast}}\,,
\end{align}
where $\sigma_{v, \beta}$ is the hitting time of $B(v, (t/4)^\beta)$.
\item On $\mathcal E_{u, \mathrm{fast}} \cap \mathcal E_3$, for any (possibly random, but measurable with respect to $\{\eta_{\delta'}^{s_i} (x): \delta' < s_i, x \in (\mathsf C_i)_{\mathrm{large}}, i = 1, \ldots, d \}$)
$\Lambda_{1, u}\subseteq \partial \mathsf C_1$ with $\mathcal L_1(\Lambda_{1, u}) \geq 0.1 \mathcal L_1(\Lambda_1)$, we have
\begin{align}\label{eq-item-2}
P_u (F(\sigma_{\Lambda_{1, u}}) \leq \delta^{2-\iota} C_\delta) \geq 
p_{\mathrm{fast}}\,,
\end{align}
where $\sigma_{\Lambda_{1, u}}$ is the hitting time of $\Lambda_{1, u}$.
\end{enumerate}
  \end{lemma}

 \begin{proof}
 Let $\sigma_{\partial { \mathsf C}_{d, \mathrm{large}}}$ be the hitting time of $\partial \mathsf C_{d, \mathrm{large}}$ by SBM, where ${\mathsf C}_{d, {\mathrm{large}}}$ is a box concentric with $\mathsf C_d$ 
but of doubled side length.
Consider the field $\tilde \eta^{\mathsf C_d}$, which equals $\tilde \eta^{s_d}_{\epsilon'} (z)$ for $\epsilon' < s_d$ and $z \in {\mathsf C}_{d, {\mathrm{large}}}$,
 and vanishes for $z \not\in {\mathsf C}_{d, {\mathrm{large}}}$. Note that
$s_{z, \delta} \ge s_{\hat {\mathsf C}_d} / \epsilon^* \ge s_d$ for all $z \in 
{\mathsf C}_{d, {\mathrm{large}}}$ (recall Definition~\ref{def-good-sequence-box} and \eqref{eq-def-E-delta-alpha*-u-v}), where $\hat {\mathsf C}_d$ is the cell containing $\mathsf C_d$.
 Let $\tilde F$ be the PCAF with respect to $\tilde \eta^{\mathsf C_d}$. We call
 $z$  a \textit{good point} if $P_z \left( \tilde F(s_d^2) \le s_d^2 \delta^{-\iota} \sqrt{ C_\delta} \mid \mathcal E \right) \ge \frac 1 2$, where $ \mathcal E : =  \{ \sigma_{v, \beta} \leq s_d^2 \le \sigma_{\partial \mathsf C_{d, \mathrm{large}}} \}$. Let $\mathcal E_{d, \mathrm{fast}}$ be the event that $\mathcal L_1 (\mbox{good points in } \Lambda_{d-1}) \ge 0.1 \mathcal L_1 (\Lambda_{d-1})$, which has $\P ( \cdot | \mathcal V_\delta)$-probability larger than $1 - \delta^{\iota}$ by Markov's inequality. On $\mathcal E_{d, \mathrm{fast}} \cap \mathcal E_3$,  any good point $z \in \Lambda_{d-1}$ satisfies that $P_z (F(\sigma_{v, \beta}) \le \delta^{2 - \iota} C_\delta) \ge \frac 1 2 P_z (\mathcal E) \ge t^{2 \beta+1}$ (compare with \eqref{eq-PCAF-approx}), thereby establishing \eqref{eq-item-1}.

The proof of \eqref{eq-item-2} follows a similar argument, noting 
that $P_u(\sigma_{\Lambda_{1, u}} \leq s_1^2 \le \sigma_{\partial \mathsf C_{1, \mathrm{large}}} ) = \Omega (\epsilon^*) \ge 2 p_{\mathrm{fast}}$.
We omit further details.
 \end{proof}

 \begin{proof}[Proof of \eqref{eq-heat-kernel-lower-bound}]
It is enough to prove the claim for  $\delta<\delta_0$, as this will determined $t_0$ through the relation $\delta=t^{1/(2-\chi_{u,v}^+)+\iota}$. Using Lemma~\ref{lem-fast-cell},
$$\P (\mathsf C_i \mbox{ is not fast for some } i) 
\le \P (\mathcal E_3^c) + \E \P (\cup_i \mathcal E_{\mathsf C_i}^c \mid \mathcal V_\delta).$$ From the lemma, we conclude
that the event that all $\mathsf C_1, \ldots, \mathsf C_d$ are fast occurs with high probability
on the event $\mathcal E_3$. 
Similarly, by Lemma~\ref{lem-LBM-near-end-points}, \eqref{eq-item-1} and \eqref{eq-item-2} hold with $\iota$-high probability on the event $\mathcal E_3$. 
We work on the intersections of these events 
with $\mathcal E_3$, which occurs with probability 
at least $1-e^{-(\log \delta^{-1})^{0.21}}$, see \eqref{eq-E3}. Note that with our choice of parameters we have for sufficiently small $\delta>0$ that
\begin{equation}\label{eq-delta-t-lower-bound}
\delta^2 \cdot C_\delta \cdot  \delta^{-\chi^+_{u, v} - \frac 1 2 (2 - \chi^+_{u,v})^2 \cdot \iota} \leq t/4\,.
\end{equation} 
In addition, note that 
\begin{equation}\label{eq-bound-t-2}
\delta^{2-\iota} C_\delta \leq t/4 \mbox{ for fixed $\iota$ small enough.}
\end{equation}
Define $\sigma_{0} = 0$ and recursively 
$\sigma_i = \min\{r \geq \sigma_{i-1}: X_r\in \Lambda_{r, \mathrm{fast}}\}$ for $i=1, \ldots, d$. By our assumption that the $\mathsf C_i$'s are fast 
(see \eqref{eq-def-fast} and \eqref{eq-def-cell-fast}), \eqref{eq-item-2} and 
the strong Markov property of the 
SBM, we see, recalling \eqref{eq-delta-t-lower-bound} and \eqref{eq-bound-t-2}, that
$$P_u\Big(F(\mbox{$\sum_{i=1}^d$} \sigma_i) \leq t/2 \Big) \geq (p_{\mathrm{fast}})^d\,.$$
Combined with \eqref{eq-item-1}, we obtain that 
$$P_u(|Y_s - v| \leq (t/4)^\beta \mbox{ for some } s\leq 3t/4) \geq 
(p_{\mathrm{fast}})^d \cdot t^{2 \beta + 1} \ge \exp \{ - t^{- \frac {\chi^+_{u,v}} {2 - \chi^+_{u,v}} - 2 \iota } \} \,,$$
where in the last estimate we used that $\chi^+_{u,v}< 1$.
Combined with \cite[Corollary 5.20]{MRVZ14} (with an appropriately chosen large $\beta$ in part (i) of 
Lemma~\ref{lem-LBM-near-end-points}), this completes the proof of \eqref{eq-heat-kernel-lower-bound}.
 \end{proof}
 
\subsection{Upper bound}
\label{subsec-LHKUB}

In this section, we will provide an upper bound on the Liouville heat kernel based on the Liouville graph distance. For $u, v\in \mathbb V$, we denote 
$$\chi^-_{u, v} = \liminf_{\delta \to 0} \frac{\E \log D_{\gamma, \delta}(u, v)}{\log \delta^{-1}}\,.$$
Recalling Lemma~\ref{lem-obvious-bounds}, we see that $0<\chi^-_{u, v} <1$.
We will show that there exists a finite random variable
 $C>0$ (measurable with respect to the GFF) such that for all $t\in (0,1]$,
\begin{equation}\label{eq-heat-kernel-upper-bound-fancy}
\pe_t^{\gamma}(u, v) \leq C \exp\Big\{-t^{-\frac{\chi^-_{u, v}}{2 - \chi^-_{u, v}} +o(1)}\Big\}\,.
\end{equation}
(As we discuss below, the restriction to $t\in (0,1]$ is possible because of   Lemma~\ref{lem-Liouville-hitting-probability}.)
In order to prove \eqref{eq-heat-kernel-upper-bound-fancy}, the key is to show that there exists a small positive constant $c = c_{\gamma, u, v} >0$ such that, for all 
small $\iota>0$, it holds with probability at least $1- t^{c \cdot \iota^2}$
that
\begin{equation}\label{eq-heat-kernel-upper-bound}
P_u(Y_r \in B(v, \delta^{C_{\mathrm{mc}}}) \mbox{ for some } r\leq t )\leq  \exp\Big\{-t^{-\frac{\chi^-_{u, v}}{2 - \chi^-_{u, v}} + \iota}\Big\}\,.
\end{equation}

In analogy with the proof of  \eqref{eq-heat-kernel-lower-bound}, in order to 
show \eqref{eq-heat-kernel-upper-bound} we will  show that for any
cell in $\mathcal V_\delta$, with not too small probability the Liouville Brownian motion will accumulate not too small Liouville time when crossing it
(here, we will choose 
$\delta \approx t^{\frac{1}{2 - \chi^-_{u, v}}}$). Throughout, we
continue to work on $\mathcal E_1$, see
\eqref{eq-E1}, and recall the notation $\epsilon^*$ from
 \eqref{eq-def-epsilon*} and $C_\delta$ from \eqref{eq-def-err-3}.
For  a cell $\mathsf C\in \mathcal V_\delta$ and $z\in \mathsf C$, we say that 
$z$ is a \textit{slow point} if 
\begin{equation}\label{eq-def-slow}
P_z(F(\sigma_{\partial \mathsf C}) \geq \delta^2 / C_\delta) \geq \alpha_{\mathrm{slow}}\,,
\end{equation}
where $\alpha_{\mathrm{slow}}>0$ is a constant depending only on $\gamma$, 
which is determined in Lemma \ref{lem-slow-cell} below. We note that a point can be both fast and slow according to our definition. We say that
a cell $\mathsf C$ is slow if the (two-dimensional) Lebesgue measure of  slow points in $\mathsf C$ is at least $\alpha_{\mathrm{slow}} s_{\mathsf C}^2$.
\begin{lemma}\label{lem-slow-cell}
 There exists  a constant $\alpha_{\mathrm{slow}}>0$ depending only on $\gamma$ such that the following holds. For each $\mathsf C\in \mathcal V_\delta$, we have that $$\P(\mathsf C \ \mbox{\rm is slow } \mid \mathcal V_\delta) \geq 1 - e^{- \alpha_{\mathrm{slow}} 2^{\sqrt{\log \delta^{-1}}}}\,.$$
\end{lemma}
\begin{proof}
We set $k =  \lfloor\sqrt{\log \delta^{-1}}\rfloor$ and $K = 2^k$. We remark that 
$k,K$ are different from those used in the course of the proof of the lower 
bound.


Partition $\mathsf C$ into $K^2$ many dyadic squares with side length $s_\mathsf C/K$, and  denote  by $\mathcal B_{\mathsf C}$ the collection of these boxes. For $\mathsf B \in 
\mathcal B_{\mathsf C}$ and  $z\in \mathsf B$, we say $z$ is a \textit{very-slow point} (with respect to the box $\mathsf B$) if
\begin{equation}\label{eq-def-pre-slow}
P_z(F(\sigma_{\partial \mathsf B_{\mathrm{large}}}) \geq \delta^2 / C_\delta ) \geq \alpha_{\mathrm{slow}}
\,,
\end{equation}
where we recall that $\mathsf B_{\mathrm{large}}$ is defined to be the box
concentric with $\mathsf B$  with doubled side length.
Note that a point $x$ away from $\partial \mathsf C$ (more precisely, if $\|x - \partial \mathsf C\|_\infty \geq 2 s_{\mathsf C}/K$) is slow if it is very slow.

We will work with the field $\hat \eta^{\mathsf B} : = \{ \tilde \eta^{\mathsf B, \epsilon^* s_{\mathsf C}}_{\epsilon'} (z) : \epsilon', z \}$, defined by replacing $s_{\mathsf C}$ with $\epsilon^* s_{\mathsf C}$ in \eqref{eq-def-tilde-eta}, i.e.
 $$
 \tilde \eta^{\mathsf B, \epsilon^* s_{\mathsf C}}_{\epsilon'} (z) : = \left\{ \begin{array}{ll} \sqrt{\pi} \int_{\mathbb V \times ((\epsilon')^2, (\epsilon^* s_{\mathsf C})^2)} p_{_{\mathsf B_{\mathrm{large}}}} (s/2; z, w) W(dw, ds), & \mbox{if } z \in \mathsf B_{\mathrm{large}} \mbox{ and } \epsilon' < \epsilon^* s_{\mathsf C}, \\ 0, & \mbox{otherwise.} \end{array} \right.
 $$
Analogously to $\mathcal E_1$, we have that with high probability,
\begin{equation}
\label{eq-disco}
\max_{\mathsf C \in \mathcal V_\delta, \mathsf B \in \mathcal B_{\mathsf C}} \max_{z \in \mathsf B_{\mathrm{large}}} \max_{\epsilon' < \epsilon^* s_{\mathsf C}, \log_2 \epsilon' \in \mathbb Z} |\hat \eta^{\mathsf B, \epsilon^* s_{\mathsf C}}_{\epsilon'} (z) - \eta^{\epsilon^* s_{\mathsf C}}_{\epsilon'} (z)| = O(\sqrt{\log \delta^{-1}}).
 \end{equation}
Set now
\begin{equation}
\label{eq-E1p}
\mathcal E_1'=\mathcal E_{\delta, \alpha^*}
\cap \{  \mbox{\rm \eqref{eq-tilde-h-eta-assump} holds}   \} \cap \{ {\mbox{the event in \eqref{eq-disco} holds}}
\},
\end{equation}
and note that $\mathcal E_1'$ occurs with high probability.
Let $\tilde F = \tilde F_{\mathsf B}$ be defined as in \eqref{eq-def-PCAF-approx} with $\tilde \eta^{\mathsf B}$ replaced by $\hat \eta^{\mathsf B}$. We have 
on $\mathcal E_1'$ the following estimate for 
the SBM $X_\cdot$ started at $z\in \mathsf B$ and any stopping time $\tau$ so that $X_{r} \in \mathsf B_{\mathrm{large}}$ for all $0\leq r\leq \tau$ :
\begin{equation}\label{eq-PCAF-approx-lower}
F(\tau) \geq \tilde F (\tau) \delta^2 s_{\mathsf C}^{-2} \exp\{(-\log \delta^{-1})^{0.91}\}\,.
\end{equation}
We will restrict our discussion to $\mathsf B$ at least at distance $4s_{\mathsf C}/K$  away from $\partial \mathsf C$, for the reason that for such $\mathsf B$,
 \begin{equation}\label{eq-field-independence-V-delta-LHK}
\mbox{the white noise that determines $\hat \eta^{\mathsf B}$ has not explored in constructing $\mathcal V_\delta$.}
\end{equation}
For $z \in \mathsf B$, we claim that there is an event $\mathcal E_{z, \mathrm{slow}}$, measurable with respect to the field $\hat \eta^{\mathsf B}$,
such that  \begin{equation}\label{eq-z-pre-slow}
 \P(\mathcal E_{z, \mathrm{slow}} \mid \mathcal V_\delta) \geq \alpha^\star  \mbox{ and }  (\mathcal E_{z, \mathrm{slow}} \cap \mathcal E_1')  \subseteq \{z \mbox{ is very-slow}\}\,,
 \end{equation}
where $\alpha^\star>0$ is a constant depending only on $\gamma$. We will first complete the proof of the lemma assuming  \eqref{eq-z-pre-slow}. We take a sub-collection of boxes $\mathcal B^* \subseteq \mathcal B_{\mathsf C}$ such that 
 \begin{itemize}
 \item All boxes in $\mathcal B^*$ are at least $4s_{\mathsf C}/K$ distance away from $\partial \mathsf C$;
 \item The pairwise distance of two boxes in $\mathcal B^*$ is at least $4s_{\mathsf C}/K$;
 \item $|\mathcal B^*| \geq 10^{-4} K^2$.
 \end{itemize}
 For each $\mathsf B\in \mathcal B^*$, let $\mathcal L_{\mathrm{slow}} (\mathsf B)$ be the Lebesgue measure of very-slow points in $\mathsf B$. Then by \eqref{eq-z-pre-slow} and our assumption on $\mathcal B^*$, we see that, on $\mathcal E_1'$, we have
 $\{\mathcal L_{\mathrm{slow}}(\mathsf B): \mathsf B\in \mathcal B^*\}$ dominates a sequence of i.i.d.\ random variables $\{\mathcal L'_{\mathrm{slow}}(\mathsf B): \mathsf B\in \mathcal B^*\}$ such that $\E  \mathcal L'_{\mathrm{slow}}(\mathsf B)  \geq \alpha^\star s_{\mathsf C}^2 K^{-2}$ and $\mathcal L'_{\mathrm{slow}}(\mathsf B) \leq s_{\mathsf C}^2 K^{-2}$. Therefore, $\P (\mathcal L'_{\mathrm{slow}}(\mathsf B) \geq \alpha^\star s_{\mathsf C}^2 K^{-2}/2 ) \geq \alpha^\star/2$. We
deduce that
 $$
\P(\sum_{\mathsf B\in \mathcal B^*} {\bf 1}_{ \{ \mathcal L'_{\mathrm{slow}}(\mathsf B) \geq s_{\mathsf C}^2 K^{-2} \alpha^\star/2 \} } \geq 10^{-6} \alpha^\star K^2 ) \geq
1 - e^{-10^{-10} \alpha^\star K^2}\,,
 $$
 completing the proof of the lemma, except for the proof of \eqref{eq-z-pre-slow},
to which we turn next.
 
Let $t_{\mathsf C} = s_{\mathsf C}^2 K^{-3}$. We will show below that for all $z \in \mathsf B \in \mathcal B_{\mathsf C}$,
 \begin{equation}\label{eq-tilde-F-lower}
\P(P_z( \tilde F (t_{\mathsf C}) \geq s_{\mathsf C}^2 K^{-4}) \geq \alpha_1 \mid \mathcal V_\delta) \geq \alpha_1, \mbox{ where } \alpha_1 >0 \mbox{ is an absolute constant}.
 \end{equation}
 Since $\sigma_{\partial  \mathsf B_{\mathrm{large}}} \geq t_{\mathsf C}$ occurs with probability tending to 1 as $\delta \to 0$, we can deduce from \eqref{eq-tilde-F-lower} that for sufficiently small $\delta$,
$$\P(P_z(\sigma_{\partial \mathsf B_{\mathrm{large}}} \geq t_{\mathsf C}, \tilde F (t_{\mathsf C}) \geq s_{\mathsf C}^2 K^{-4}) \geq \alpha_1/2  \mid \mathcal V_\delta) \geq \alpha_1\,.$$
Combined with \eqref{eq-PCAF-approx-lower}, this implies \eqref{eq-z-pre-slow} with an appropriate choice of the absolute constant $\alpha_{\mathrm{slow}}>0$.

We finally turn to the proof of \eqref{eq-tilde-F-lower}. Fix $1<p<4/\gamma^2$. We
follow the arguments in \cite[Appendix B]{GRV13} to show that $\E E_z (\tilde F(t_{\mathsf C}))^p \leq O(t_{\mathsf C}^p)$. (The proof in  \cite{GRV13} applies to any
log-correlated Gaussian field, and thus carries over to
the field $\hat \eta^{\mathsf B}$
with no essential change.) With the moment estimate at hand, we can apply H\"older's inequality and get that for any $\kappa>0$
$$\E E_z (\tilde F(t_{\mathsf C})) \leq \kappa t_{\mathsf C} + \E E_z ( \tilde F(t_{\mathsf C}) 1_{\{\tilde F(t_{\mathsf C}) \geq \kappa t_{\mathsf C}\}} ) \leq \kappa t_{\mathsf C} + O\Big( t_{\mathsf C} \left( \E \big( P_z(\tilde F(t_{\mathsf C}) \geq \kappa t_{\mathsf C} ) \big) \right)^{1-1/p} \Big)\,.$$
Combined with the fact that $\E E_z (\tilde F(t_{\mathsf C})) = t_{\mathsf C}$ and an appropriate choice of $\kappa>0$ (a small constant depending only on $\gamma$), we deduce that $\E \big( P_z(\tilde F(t_{\mathsf C}) \geq \kappa t_{\mathsf C} ) \big) $ is lower bounded by a positive constant depending only on $\gamma$. 
Combined with \eqref{eq-field-independence-V-delta-LHK}, this then implies \eqref{eq-tilde-F-lower}, as desired.
 \end{proof}

\begin{proof}[Proof of \eqref{eq-heat-kernel-upper-bound}]
Fix an arbitrarily small $\iota>0$. 
 Let $\delta = t^{\frac{1}{2 - \chi^-_{u, v}} - \iota}$. By Propositions~\ref{prop-approximate-LGD} and \ref{prop-concentration}, we see that with $(c \cdot \iota^2)$-high probability for some $d \geq \delta^{-\chi^-_{u, v} + \beta \cdot \iota / 2}$ every sequence of neighboring cells in $\mathcal V_\delta$ connecting  $u$ to $v$ contains at least $d$ cells, where $\beta = \frac 1 2 (2 - \chi^-_{u,v})^2$. On $\mathcal E'_1$  from \eqref{eq-E1p},  all the cells have side length at least $\delta^{C_{\mathrm{mc}}}$,
and therefore   the number of neighboring cells connecting  $u$ to $B(v, \delta^{C_{\mathrm{mc}}})$ is at least $d-2$.
 Define $\sigma_0 = 0$ and for $i\geq 1$ define
$$\sigma_i = \{r\geq \sigma_{i-1}: X_r \in \partial \mathsf C_{X_{\sigma_{i-1}}, \mathrm{large}} \}\,,$$
where  we recall that $\mathsf C_{z, \mathrm{large}}$ denotes a box concentric 
with $\mathsf C_{z, \delta}$, the cell containing $z$, with  doubled side length. On $\mathcal E'_1$,  the event $\mathcal E_{\delta, \alpha^*}$ from
  \eqref{eq-def-E-delta-alpha} holds, and therefore
 in order to hit $B(v, \delta^{C_{\mathrm{mc}}})$, the Liouville Brownian motion has to go through $d-2$ cells and every time it exits $\mathsf C_{\mathrm{large}}$ from $\mathsf C$ (for some $\mathsf C\in \mathcal V_\delta$) it crosses at most $\delta^{- \beta \cdot \iota / 2}$
many cells. Thus, 
\begin{equation}\label{eq-Y-sigmas}
\{Y_r \in B(v, \delta^{C_{\mathrm{mc}}}) \mbox{ for some } r\leq t\} \subseteq \{\sum_{i=1}^{d \delta^{\beta \cdot \iota}} (F(\sigma_ i) - F(\sigma_{i-1})) \leq t\}\,.
\end{equation}
By Lemma~\ref{lem-slow-cell}, the event that all 
cells are slow has high probability. On this event,
$$P_{X_{\sigma_{i-1}}}(F(\sigma_ i) - F(\sigma_{i-1}) \geq \delta^2 /C_\delta ) \geq P_{X_{\sigma_{i-1}}}(\mbox{$X_\cdot$ hits a slow point in } \mathsf C_{X_{\sigma_{i-1}}} \mbox{ before } \sigma_i) \alpha_{\mathrm{slow}}\,,$$
which is bounded below by 
a constant $\alpha'_{\mathrm{slow}}>0$ depending only on $\gamma$. By the 
strong Markov property of the SBM, we conclude  that 
$(F(\sigma_ i) - F(\sigma_{i-1}))'s$ dominates a sequence of i.i.d.\ non-negative random variables which take value $\delta^2 / C_\delta$ with probability $\alpha'_{\mathrm{slow}}>0$.  At this point, a simple large deviation estimates yields that for sufficiently small $t$,
 $$ 
P_u(\sum_{i=1}^{d \delta^{\beta \cdot \iota} } \big( F(\sigma_ i) - F(\sigma_{i-1}) \big) \leq t) \le e^{- \Omega(1) d \delta^{\beta \cdot \iota}} \leq e^{-d \delta^{2 \cdot \iota}}  \le \exp \{- t^{- \frac {\chi^-_{u,v} } {2 - \chi^-_{u,v}} + 4 \cdot \iota }  \}  \,,
 $$
where the three inequallities hold respectively because  the exponent of $\frac t {\delta^2 / C_\delta} $ (with respect to $1/t$) is strictly less than that of $d \delta^{\beta \cdot \iota}$,
because 
$\beta \le 2$, and because $\chi^-_{u,v} < 1$. Combined with \eqref{eq-Y-sigmas} and the fact that we considered a high probability event, this completes the proof of \eqref{eq-heat-kernel-upper-bound}.\end{proof}

\begin{proof}[Proof of \eqref{eq-heat-kernel-upper-bound-fancy}]
Since the event $\{Y_r \in B(v, \delta^{C_{\mathrm{mc}}}) \mbox{ for some } r\leq t \}$ is increasing in $t$, we can apply a union bound over  all $t$ of the form $t=2^{-j}$, use
 \eqref{eq-heat-kernel-upper-bound} and the Borel-Cantelli lemma to conclude that for any $\iota>0$,
 there exists a random variable $C>0$ 
such that  for all  $t>0$,
$$P_u(Y_r \in B(v, \delta^{C_{\mathrm{mc}}}) \mbox{ for some } r\leq t )\leq C  \exp\Big\{-C^{-1}t^{-\frac{\chi^-_{u, v}}{2 - \chi^-_{u, v}} +  c \cdot \iota}\Big\}\,.$$
Applying Lemma~\ref{lem-Liouville-hitting-probability} with $\alpha_1 = \frac{C_{\mathrm{mc}}}{2 - \chi^-_{u,v}}, \alpha_2 = - \frac{\chi^-_{u, v}}{2 - \chi^-_{u, v}} + \frac 1 2 c \cdot \iota$ and a corresponding choice of $\alpha_3$, this completes the proof of \eqref{eq-heat-kernel-upper-bound-fancy}.
\end{proof}

\section{Existence of the Liouville graph distance exponent}\label{sec:exponent}

In this section, we will show that the exponent for the Liouville graph distance exists, and that the exponent does not depend on the choice of starting or ending points. Recall that $\mathbb V^\xi = \{v\in \mathbb V: |v-\partial \mathbb V| \geq \xi\}$.
\begin{prop}\label{prop-existence-exponent}
For any $\gamma\in (0, 2)$, there exists $\chi = \chi(\gamma)$ such that  for any $u,v \in \mathbb V \setminus \partial \mathbb V$,
$$ \lim_{\delta \to 0} \frac{\E \log D_{\gamma, \delta}(u, v)}{\log \delta^{-1}} = \chi\,.$$
Furthermore, the $\chi(\gamma)$ here is the same as that in Lemma~\ref{lem-existence-exponent}.
\end{prop}

Our proof of Proposition~\ref{prop-existence-exponent} is based on
subadditivity; however, some preparations are needed
before subadditivity can be invoked. We begin by setting a few notations.
Let $\bar {\mathbb V}$ (respectively, $\tilde {\mathbb V}$)
be a box concentric with $\mathbb V$ and 
of side length $1/20$ (respectively, $1/5$). 
For $u, v \in \mathbb V$ and  $\lambda > 0$, let ${\mathbb V}_{u, \lambda}$ 
denote the box centered at $u$ and of side length $\lambda$,
let $\tilde {\mathbb V}_{u, v}$ denote the
translated and rotated box centered at $\frac{u+v}{2}$, of side length $2|u-v|$, and with two sides parallel to the line segment joining $u$ and $v$.
In particular,  for all $u, v\in \bar{\mathbb V}$ we have $\tilde {\mathbb V}_{u, v} \subseteq \tilde {\mathbb V}$. Furthermore,   for all $v\in \tilde {\mathbb V}$ in the definition for $\eta_{\delta}^{\tilde \delta}(v)$ as in \eqref{eq:WND_decomposition-approximation}, the truncation for the transition kernel upon exiting $\mathbb V$ becomes redundant since $B(v, 4^{-1} s^2 |\log s^{-1}| \wedge 10^{-1}) \subseteq \mathbb V$ for all $s>0$. Therefore, for $u, v, u', v'\in \bar{\mathbb V}$ with $|u-v| = |u'-v'|$, denoting by $\theta$ an isometry which  maps  from $\tilde {\mathbb V}_{u, v}$ to $\tilde {\mathbb V}_{u',v'}$, we have that
\begin{equation}\label{eq-translation-invariant}
\{\eta_{\delta}^{\tilde \delta}(x): x\in \tilde {\mathbb V}_{u, v}\} \stackrel{law}{=} \{\eta_{\delta}^{\tilde \delta}(\theta x): x\in \tilde {\mathbb V}_{u, v}\} {\mbox{ for all } 
0<\delta <\tilde \delta \leq \infty.}
\end{equation}
For $u, v \in {\mathbb V}$ and $\delta$, we define $D_{\gamma, \delta}^A (u, v)$ to be the minimal number of Euclidean balls with rational center and radius, \emph{contained in $A$} with LQG measure at most $\delta^2$, whose union contains a path from $u$ to $v$. Denote $\tilde D_{\gamma, \delta}(u, v) = D_{\gamma, \delta}^{\tilde {\mathbb V}_{u, v}} (u, v)$ and $\bar D_{\gamma, \delta}^{x,\lambda} (u, v) = D_{\gamma, \delta}^{ {\mathbb V}_{x, \lambda}} (u, v)$ for brevity. We also define the tilde-approximate Liouville graph distance, similar to the approximate Liouville graph distance. That is, we repeatedly and dyadically partition $\tilde {\mathbb V}_{u, v}$ until all cells have approximate Liouville quantum gravity measure (as defined in \eqref{eq-def-approximate-LQG}) at most $\delta^2$, and we denote by $\mathcal V_{\delta, u, v}$ the resulting partition.
Let $\tilde D'_{\gamma, \delta}(u, v)$ be the graph distance between the two cells containing $u$ and $v$ in $\mathcal V_{\delta, u, v}$ (note that, of course, all cells are contained in $\tilde {\mathbb V}_{u, v}$).  

By \eqref{eq-translation-invariant}, we see that for $u,v \in \bar {\mathbb V}$,
\begin{equation}\label{eq-depend-on-Euclidean}
\mbox{ the law of } \tilde D'_{\gamma, \delta}(u, v) \mbox{ or } \tilde D_{\gamma, \delta, \eta}(u, v) \mbox{ depends on } u, v \mbox{ only through } |u-v|\,. 
\end{equation}
 The translation invariance property in \eqref{eq-depend-on-Euclidean}
will be useful below when setting up the sub-additive argument. 

\begin{remark}
\label{rem-5.2}
One can verify that our proofs for Propositions~\ref{prop-approximate-LGD}, \ref{prop-concentration}, Lemmas~\ref{lem-approximate-LGD-two-deltas}, \ref{lem-LGD-compare}, \ref{lem-LGD-two-fields} and Corollary~\ref{cor-LGD-two-deltas} extend automatically to the tilde-Liouville graph distance and the approximate tilde-Liouville graph distance. As a result, in this section  we often apply these results to the tilde-version of these statements (formally, replacing $D$ by $\tilde D$ and replacing $D'$ by $\tilde D'$).
\end{remark}

The next two lemmas are the key ingredients for the proof of Proposition~\ref{prop-existence-exponent} .
\begin{lemma}\label{lem-existence-exponent}
For any $\gamma\in (0, 2)$, there exists $\chi = \chi(\gamma)$ such that 
for any $u, v\in \bar {\mathbb V}$,
$$\lim_{\delta \to 0} \frac{\E \log \tilde D_{\gamma, \delta, \eta}(u, v)}{\log \delta^{-1}} = \lim_{\delta \to 0} \frac{\E \log \tilde D'_{\gamma, \delta}(u, v)}{\log \delta^{-1}} = \chi\,.$$
\end{lemma}

\begin{lemma}\label{lem-exponent-point-to-boundary}
Let $\chi$ be as in Lemma~\ref{lem-existence-exponent}. For any $u\in \bar {\mathbb V}$, $\lambda = \frac 1 {20}$, 
 $$ 
\lim_{\delta\to 0} \frac{\E \log (\min_{x\in \partial {\mathbb V}_{u, \lambda}} \bar D_{\gamma, \delta, \eta}^{u, 2 \lambda} (u, x))}{\log \delta^{-1}} = \chi\,.
 $$
\end{lemma}

\begin{proof}[Proof of Proposition~\ref{prop-existence-exponent} (assuming
Lemmas \ref{lem-existence-exponent} and \ref{lem-exponent-point-to-boundary})] 
 We first prove that for an arbitrarily small $\iota>0$
\begin{equation}\label{eq-exponent-upper-bound}
\E \log D_{\gamma, \delta, \eta}(u, v) \leq (\chi + \iota) \log \delta^{-1}\mbox{ as } \delta \to 0.
\end{equation}

To this end, let $y_i = u + \frac{i}{l} (v-u)$, $i=0, \ldots, l$ with $l = \min \{ \ell \in \mathbb Z :  \frac {|u-v|}{\ell} \le \min \{ \frac 2 {\sqrt 5} \xi, \frac 1 {20} \} \}$, where $\xi = \frac 1 2 \min \{ |u - \partial \mathbb V |_\infty, |v - \partial \mathbb V |_\infty \} $. Pick $\bar u, \bar v \in \bar{\mathbb V}$ with $|\bar u-\bar v|_\infty = 1/20$.
 Applying Lemma~\ref{lem-scaling-coupling} to each pair $(\tilde {\mathbb V}_{\bar u, \bar v}, \tilde {\mathbb V}_{y_i, y_{i+1}} )$ so that $\zeta^{(1)}$ has the same law as the $\eta$-process on $\tilde {\mathbb V}_{\bar u, \bar v}$ and 
$\zeta^{(2)}$ has the same law as the
$\eta$-process on $\tilde {\mathbb V}_{y_i, y_{i+1}}$,
 as well as using Lemma~\ref{lem-LGD-compare} (note that we can choose some constant $b_1 = b_1(u,v)$ as in the assumption of Lemma~\ref{lem-LGD-compare}), we see that with high probability 
\begin{equation}\label{eq-coupling-comparison}
\tilde D_{\gamma, \delta, \zeta^{(2)}}(y_i, y_{i+1}) \leq \tilde D_{\gamma, \delta e^{-\sqrt{\log \delta^{-1}}},  \zeta^{(1)}}  (\bar u, \bar v) \,.
\end{equation}
Combined with Lemmas~\ref{lem-existence-exponent}, \ref{lem-LGD-two-fields}, Corollary~\ref{cor-LGD-two-deltas} and Proposition~\ref{prop-concentration},
we see that with high probability,
 \begin{equation} \label{Eq.boundfortildeD}
 \tilde D_{\gamma, \delta, \eta}(y_i, y_{i+1}) \leq \delta^{-\chi-\iota} \mbox{ for } i=1, \ldots, l\,,
  \end{equation}
implying that  $D_{\gamma, \delta, \eta} (u,v) \le l \times \delta^{- \chi - \iota}$ by triangle inequality.  This yields \eqref{eq-exponent-upper-bound} (recall Proposition~\ref{prop-concentration}).

Next, we prove the lower bound, i.e., we prove that for arbitrarily small $\iota>0$,
\begin{equation}\label{eq-exponent-lower-bound}
\E \log D_{\gamma, \delta, \eta}(u, v) \geq (\chi - \iota) \log \delta^{-1}\mbox{ as } \delta \to 0.
\end{equation}
To this end, let $\lambda = \min\{ \frac 1 {\sqrt 2} \xi, \frac 1 {\sqrt 2} |u-v|, \frac 1 {20} \}$,
 and we see that $v \notin {\mathbb V}_{u, \lambda} \subseteq \mathbb V_{\xi}$. Similarly to the derivation of \eqref{eq-coupling-comparison}, we apply  Lemma~\ref{lem-scaling-coupling}
to the pair $( {\mathbb V}_{\bar u, \frac 1 {20}}, {\mathbb V}_{u, \lambda})$, combine with Lemma~\ref{lem-exponent-point-to-boundary}, and get that with high probability,
 $$
\min_{x \in \partial {\mathbb V}_{u, \lambda}} \bar D_{\gamma, \delta, \zeta^{(2)}}^{u, 2 \lambda}(u, x) \geq \min_{x\in \partial  {\mathbb V}_{\bar u, \frac 1 {20}}} \bar D_{\gamma, \delta e^{\sqrt{\log \delta^{-1}}}, \zeta^{(1)}}^{\bar u, \frac 1 {10}} (\bar u, x) \geq   (\chi - \iota) \log \delta^{-1} \,,
 $$
where $\zeta^{(1)}$ has the same law as the $\eta$-process on 
${\mathbb V}_{\bar u, \frac 1 {20}}$, and $\zeta^{(2)}$ has the same law as the
$\eta$-process on ${\mathbb V}_{u, \lambda}$. With high probability, balls intersecting both $\partial  {\mathbb V}_{u, \lambda}$ and $\partial {\mathbb V}_{u, 2 \lambda}$ have LQG measure larger than $2 \delta^2$, implying 
\begin{equation}\label{eq-geodesic-range}
\min_{x\in \partial {\mathbb V}_{u, \lambda}}  D_{\gamma, \delta, \eta}(u, x) = \min_{x\in \partial {\mathbb V}_{u, \lambda}} \bar D_{\gamma, \delta, \eta}^{u, 2 \lambda} (u, x) \,.
\end{equation} It follows that
 \begin{equation}\label{eq-distance-point-to-boundary-bound}
\E  \min_{x\in \partial \bar {\mathbb V}_{u, \lambda}} D_{\gamma, \delta, \eta}(u, x) \geq  (\chi - \iota) \log \delta^{-1} .
 \end{equation}
Since $D_{\gamma, \delta, \eta}(u, v) \geq \min_{x\in \partial {\mathbb V}_{u, \lambda}} D_{\gamma, \delta, \eta}(u, x)$ for $v \notin {\mathbb V}_{u, \lambda}$, we get \eqref{eq-exponent-lower-bound} as required.

Combining \eqref{eq-exponent-upper-bound}, \eqref{eq-exponent-lower-bound} and Lemma~\ref{lem-LGD-two-fields} we complete the proof of the proposition.
\end{proof}

Next, we prove Lemma~\ref{lem-existence-exponent}, employing 
a sub-additive argument.  As in the proof of  \eqref{eq-heat-kernel-lower-bound}, Lemma~\ref{lem-partition-independence} plays a crucial role.

\begin{proof}[Proof of Lemma~\ref{lem-existence-exponent}]
For $u, v\in \bar {\mathbb V}$, let $w_i = u + \frac i 9 |u-v|$ so that $\tilde {\mathbb V}_{x,y} \subseteq \tilde {\mathbb V}_{u,v}$ for all $x,y \in \tilde {\mathbb V}_{w_{i-1},w_i}$, $i = 1, \ldots, 9$ (we made such choices so that later the paths we build to join $w_i$ and $w_{i+1}$ will be all contained in $\tilde {\mathbb V}_{u,v}$). Fix $\delta>0$.
\begin{defn}[$\mathcal E^\star_{\delta, \alpha^*, u, v}$]
\label{def-Estar}  Let  $\mathcal E^\star_{\delta, \alpha^*, u, v}$ denote  the following event: there exists a good sequence as in  Definition~\ref{def-good-sequence-box}
of  neighboring dyadic boxes $\mathcal C = \mathsf C_1, \ldots, \mathsf C_d$, contained in  $\cup_{i=1}^9  \tilde {\mathbb V}_{w_{i-1}, w_i}$ and  measurable with respect to $\mathcal F^* = \sigma(\cup_{i=1}^9  \mathcal V_{\delta, w_{i-1}, w_i})$,  joining $u$ to $ v$, such that 
\begin{itemize}
\item $d\leq  e^{(\log \delta^{-1})^{0.7}} \sum_{i=1}^9 d_i$ with $d_i =\tilde D'_{\gamma, \delta} (w_{i-1}, w_i)$; 
\item Each $\mathsf C_i$ satisfies $M_{\gamma, s_{\mathsf C_i}} (\mathsf C_i) \leq \delta^2 e^{O( (\log \delta^{-1})^{0.8})}$;
\item  The 
 law of  $\{\eta_{\delta'}^{s_{\mathsf C_i}}(x): \delta' < s_{\mathsf C_i}, x \in (\mathsf C_i)_{\mathrm{large}}, 
 \mathsf C_i \in \mathcal C  \}$
conditioned on $\mathcal F^*$ coincides with
its unconditional version. 
\end{itemize}
\end{defn}
Note that here as in Section~\ref{subsec-LHKLB} we have abused the notation by denoting by $\mathsf C_i$ a dyadic box which is not necessarily a cell. The abuse of notation is justified by the fact that $M_{\gamma, s_{\mathsf C_i}} (\mathsf C_i) \leq \delta^2 e^{O( (\log \delta^{-1})^{0.8})}$ and thus the  $\mathsf C_i$'s will essentially play the role of cells.

 Following the discussions after \eqref{eq-E2}
(with a crucial application of Lemma~\ref{lem-partition-independence}), we see that $\P(\mathcal E^\star_{\delta, \alpha^*, u, v}) \geq  1-e^{-(\log \delta^{-1})^{0.23}}$.
By Proposition~\ref{prop-approximate-LGD}, Lemmas~\ref{lem-scaling-coupling}, \ref{lem-LGD-compare}, \ref{lem-LGD-two-fields}, Corollary~\ref{cor-LGD-two-deltas} and \eqref{eq-coupling-comparison}, with high probability,
 $$
d_i \le e^{(\log \delta^{-1})^{0.92}} \tilde D^{(i)}_{\gamma, \delta, \eta} (u,v) \le e^{(\log \delta^{-1})^{0.94}}   \exp \{\E \log \tilde D_{\gamma, \delta, \eta} (u,v)  \}, 
  $$
where $\tilde D^{(i)}_{\gamma, \delta, \eta} (u,v)$ is a copy of $\tilde D_{\gamma, \delta, \eta} (u,v)$ and is coupled with $d_i$. Thus, 
 \begin{equation} \label{Eq.calD1}
\P (\mathcal D_1) \geq  1-  e^{- (\log \delta^{-1})^{0.22}}, \ \ \ \mbox{where } \mathcal D_1 : = \{  \log d \le (\log \delta^{-1})^{0.95} + \E \log \tilde D_{\gamma, \delta, \eta} (u,v)  \} \,.  
 \end{equation}

In order to set a sub-additivity argument, we need
to further relate $d$ to $\tilde D_{\gamma, \delta\tilde \delta, \eta}(u, v)$ for $\tilde \delta>\delta$. To this end, we let $x_i\in \Lambda_i = \partial \mathsf C_i \cap \partial \mathsf C_{i+1}$ for each $i=1, \ldots, d-1$, to be chosen later depending on the GFF (for convenience we write $x_0 = u$ and $x_d= v$). By the 
triangle inequality, we see that
\begin{equation}\label{eq-triangle-inequality}
\tilde D_{\gamma, \delta\tilde \delta, \eta}(u, v) \leq \sum_{i=0}^{d-1} D_{\gamma, \delta \tilde \delta, \eta}^{ \tilde {\mathbb V}_{u,v}} (x_i, x_{i+1})\,.
\end{equation}
We claim that with probability at least $1-e^{-c(\log 1/\delta)^{0.51}}$,
there exists a choice of $x_1, \ldots, x_{d-1}$ such that for all $0\leq i\leq d-1$
\begin{equation}\label{eq-union-bound-distance}
\log D_{\gamma, \delta \tilde \delta, \eta}^{\tilde {\mathbb V}_{u,v}} (x_i, x_{i+1}) \leq \E \log \tilde D_{\gamma, \tilde \delta, \eta}(u, v) + 4(\log \delta^{-1})^{0.98}\,.
\end{equation}
Assuming \eqref{eq-union-bound-distance}, we can complete the proof of the lemma, as follows. 
Denote the event in \eqref{eq-union-bound-distance} by $\mathcal D_2$ and let $\mathcal D = \mathcal D_1 \cap \mathcal D_2$. We obtain from \eqref{Eq.calD1}, \eqref{eq-triangle-inequality} and \eqref{eq-union-bound-distance} that
\begin{equation}
\label{eq-040418}
\E({\bf 1}_{\mathcal D} \log \tilde D_{\gamma, \delta\tilde \delta, \eta}(u, v))
\leq \E \log \tilde D_{\gamma, \delta, \eta} (u,v) +\E \log \tilde D_{\gamma, \tilde \delta, \eta}(u, v) + 5 \times (\log \delta^{-1})^{0.98}.
\end{equation}
On the other hand, using an analogue of \eqref{eq-very-crude}, we have by an application of Jensen's inequality that 
\[\E({\bf 1}_{{\mathcal D}^c} \log \tilde D_{\gamma, \delta\tilde \delta, \eta}(u, v))
\leq { (\log \delta^{-1})} e^{-(\log \delta^{-1})^{{ 0.1}}}.\]
Setting $\chi_\delta = \frac{\E \log 
\tilde D_{\gamma, \delta, \eta}(u, v)}{\log \delta^{-1}}$ and
combining the last display with
\eqref{eq-040418}, we obtain
$$\chi_{\delta\tilde \delta} \leq \frac{\log \delta^{-1}}{\log \delta^{-1} + \log \tilde \delta^{-1}} \chi_{\delta} +  \frac{\log \tilde \delta^{-1}}{\log \delta^{-1} + \log \tilde \delta^{-1}} \chi_{\tilde \delta} + 
(\log \delta^{-1})^{-0.01}\,.$$
Applying  \cite{Hammersley62} (see also \cite[Lemma 6.4.10]{DZ10}), this yields that $\chi_\delta$ converges to some constant $\chi$ as $\delta\to 0$ over a sequence $\delta_k=2^{-k}$, and then by continuity the convergence extends to arbitrary $\delta\to 0$.
By Proposition~\ref{prop-approximate-LGD}, Lemmas~\ref{lem-scaling-coupling}, \ref{lem-LGD-compare}, \ref{lem-LGD-two-fields} and Corollary~\ref{cor-LGD-two-deltas}, $\chi$ does not depend on $u,v$.
Combined with 
Corollary~\ref{cor-approximate-LGD-expectation} and  Lemma~\ref{lem-LGD-two-fields}, this yields Lemma~\ref{lem-existence-exponent}.

It remains to prove \eqref{eq-union-bound-distance}. The proof follows the proof strategy for \eqref{eq-heat-kernel-lower-bound}.
Set $k = \lfloor (\log \delta^{-1})^{0.51}\rfloor$ and $K = 2^k$. Partition $\mathsf C_i$ into $K^2$ many dyadic squares with side length $s_i/K$, and we denote the collection of such squares as $\mathcal B_i$, where $s_i = s_{\mathsf C_i}$. For each $\mathsf B\in \mathcal B_i$, we say  $\mathsf B$ is open if for any $\Lambda \subseteq \partial  \mathsf B$ with $\mathcal L_1(\Lambda) \geq 10^{-5}s_i/K$ there exists $\Lambda'\subseteq \partial \mathsf B$ with $\mathcal L_1(\Lambda') \geq (1 - 10^{-5})s_i/K$ such that
$$\min_{z \in \Lambda} \log \tilde D_{\gamma, \delta \tilde \delta, \eta}(z,z') \leq \E \log \tilde D_{\gamma, \tilde \delta, \eta}(u, v)   + (\log \delta^{-1})^{0.98}\,, \mbox{ for each }z'\in \Lambda'\,.$$
Recall the definitions of $\epsilon^*$  from \eqref{eq-def-epsilon*}.
Let $\check \eta^{\mathsf B}$  be defined as in \eqref{eq-def-tilde-eta} with $\mathsf B_\mathrm{large}$ and $\mathsf B_{\mathrm{Large}}$ respectively replaced by
 $\mathsf B^* = \{ x : \| x - \partial B \|_\infty \le 2 s_i / K \}$ and $\mathsf B^{**} = \{ x : \| x - \partial B \|_\infty \le 3 s_i / K \}$, i.e.
 $$
\check \eta^{\mathsf B, s_i}_{\epsilon'} (z) : = \left\{ \begin{array}{ll} \sqrt {\pi} \int_{\mathbb V \times ( (\epsilon')^2, s_i^2 )} p_{_{\mathsf B^{**}}} (s/2, z,w) W(dw, ds), & \mbox{if } z \in \mathsf B^* \mbox{ and } \epsilon' < s_i, \\ 0, & \mbox{otherwise.} \end{array} \right. 
 $$
Similarly to \eqref{eq-assumption-in-lower-heat-kernel}, we have
 $$
\max_{\mathsf C_i \in \mathcal C} \max_{\mathsf B \in \mathsf B_i, z \in \mathsf B^*} \max_{\epsilon' < s_{\mathsf C}, \log_2 \epsilon' \in \mathbb Z} |\check \eta^{\mathsf B, s_i}_{\epsilon'} (z) - \eta^{s_i}_{\epsilon'} (z)| = O(\sqrt{\log \delta^{-1}}), \quad \mbox{\rm  with high probability}. 
 $$  
Let
 $$
\mathcal E_4 : =  \mathcal E_{\delta, \alpha^*} \cap \mathcal E^\star_{\delta, \alpha^*, u, v} \cap \{ {\mbox{the above event holds}} \} \cap \{  \mbox{\rm \eqref{eq-tilde-h-eta-assump} holds}   \} \,,
 $$ 
and we work on $\mathcal E_4$.
For each $\mathsf B\in \mathcal B_i$, we claim that there is an event $\mathcal E_{\mathsf B, \mathrm{open}}$ which is measurable with respect to the field $\check \eta^{\mathsf B}$
so that 
 \begin{equation}\label{eq-z-open}
 \P(\mathcal E_{\mathsf B, \mathrm{open}} \mid {\mathcal F^*}) \geq 1- O(K^{-2})  \mbox{ and } (\mathcal E_{\mathsf B, \mathrm{open} } \cap \mathcal E_4) \subseteq \{\mathsf B \mbox{ is open}\}\,.
 \end{equation}
 (We remark that this is very  similar to \eqref{eq-z-pre-fast}.) We now verify \eqref{eq-z-open}.
 On the event we work on, we have that for any Borel set $A \subseteq \mathsf B \in \mathcal B_i$, 
 \begin{equation}\label{eq-M-A-upper-bound-bis}
M_{\gamma}(A) \leq \delta^2 s_{i}^{-2} M_\gamma^{\check \eta^{\mathsf B}} (A) 
\exp\{(\log \delta^{-1})^{0.91}\}
 \end{equation}
(this is similar to \eqref{eq-PCAF-approx}) where $M_\gamma^{\check \eta^{\mathsf B}} (A)$
is defined as in \eqref{eq-def-M-eta} with $\zeta_\cdot$ replaced by $\check \eta^{\hat B}$.
Consider $z, z' \in \partial \mathsf B$.  It would be simple to proceed if
the process $\{ \check \eta_{\epsilon'}^{\mathsf B, s_i } (x): \epsilon' <  s_i,  x\in \tilde {\mathbb V}_{z, z'}\}$ had the same law as 
$\{\eta_{a\epsilon'}(\theta x): \epsilon' < s_i, x\in \tilde {\mathbb V}_{z, z'}\}$, where $\theta (x) = a\theta'(x)$ for an appropriate $a>0$ and an isometry $\theta'$ such that $\theta$ maps $\tilde {\mathbb V}_{z, z'}$  to $\tilde {\mathbb V}_{u, v}$ (we see that $a$ is of the same order as $s_{i}^{-1} K$ and so $a^{-1} \leq s_i$).  While such desired identity in law does not hold precisely, we 
claim that
there exists a coupling of $\{ \check \eta_{\epsilon'}^{\mathsf B, s_i } (x): 
\epsilon' < s_i, x\in \tilde {\mathbb V}_{z, z'}\}$ and 
$\{\eta_{a\epsilon'}(\theta x): \epsilon' < s_i, 
x\in \tilde {\mathbb V}_{z, z'}\}$ such that with high probability with respect to $\P (\cdot \mid {\mathcal F^*})$
 \begin{equation}\label{eq-scaling-invariance-approximate}
\max_{n\geq 1, 2^{-n}\leq  a^{-1}}\max_{x\in \tilde {\mathbb V}_{z, z'}} |\check \eta_{2^{-n}}^{\mathsf B, s_i }  
(x) - \eta_{a2^{-n}}(\theta x)| \leq (\log \delta^{-1})^{0.92}\,.
 \end{equation}
We postpone the proof of \eqref{eq-scaling-invariance-approximate} and proceed 
with the proof of \eqref{eq-z-open}.
Since (by a straightforward computation) $|\Var ( \check \eta_{2^{-n}}^{\mathsf B, s_i }
(x)) - \Var(\eta_{a2^{-n}}(\theta x))| = O(1)(\log \delta^{-1})^{0.6}$ for all $x\in \tilde {\mathbb V}_{z, z'}$ and $2^{-n}\leq  a^{-1}$, we see that on the event that
 \eqref{eq-M-A-upper-bound-bis} and \eqref{eq-scaling-invariance-approximate} hold we 
have that 
 $$
M_\gamma^{\check \eta^{\mathsf B}} (A) \le \exp\{ ( \log \delta^{-1} )^{-0.93} \} a^{-2} M_\gamma^{\eta} (\theta A) \le \exp\{ ( \log \delta^{-1} )^{-0.94} \} s_i^2 M_\gamma^{\eta} (\theta A) ,
 $$
recalling $a^{-1} \le s_i$. Combined with \eqref{eq-M-A-upper-bound-bis}, it follows that
 $$
\tilde D_{\gamma, \delta \tilde \delta, \eta}(z, z') \leq \tilde D_{\gamma, \tilde \delta \exp\{ -(\log \delta^{-1})^{0.95}\}, \eta}(u, v)\,.
 $$
We now combine the preceding inequality with Corollary~\ref{cor-LGD-two-deltas} and Proposition \ref{prop-concentration}, and deduce that 
$$\P( {\log \tilde D_{\gamma, \delta \tilde \delta, \eta}(z,z') \geq \E \log \tilde D_{\gamma, \tilde \delta, \eta}(u, v)  + (\log \delta^{-1}})^{0.97} \mid { \mathcal F^*}) \leq O(K^{-4})\,.$$
Write $\Lambda_{z, \mathrm{far}} = \{z'\in \partial \mathsf B:
\log \tilde D_{\gamma, \delta \tilde \delta, \eta}(z,z') \geq \E \log \tilde D_{\gamma, \tilde \delta, \eta}(u, v)  + (\log \delta^{-1})^{0.97} \}$. The preceding inequality implies that 
$$\P \left( \mathcal L_1(\Lambda_{z, \mathrm{far}} ) \geq K^{-1} {\mathcal L_1 (\partial \mathsf B) } \mid {\mathcal F^*} \right) = O(K^{-3}) \mbox{ for each } z \in \partial \mathsf B\,.$$
Therefore, we get that
$$\P \left(  \left.\mathcal L_1(\{z\in \partial \mathsf B : \mathcal L_1(\Lambda_{z, \mathrm{far}} ) \geq K^{-1}{\mathcal L_1 (\partial \mathsf B)} \}) \geq K^{-1}{\mathcal L_1 (\partial \mathsf B)}
\, \right\rvert {\mathcal F^*} \right) = O(K^{-2})\,.$$
This implies that \eqref{eq-z-open} holds (up to the proof of \eqref{eq-scaling-invariance-approximate}, which is still postponed).

Having established \eqref{eq-z-open}, we proceed with the percolation argument. We say  $\mathsf C_i$ is \textit{desirable} if for any $\Lambda_{i, \mathrm{end}} \subseteq  \Lambda_i$ with $\mathcal L_1(\Lambda_{i, \mathrm{end}}) \geq 0.1 \mathcal L_1(\Lambda_i) $ (here it is useful to recall \eqref{eq-Lambda-i-not-small}), there exists 
\begin{equation}\label{eq-def-B-i-2}
\Lambda_{i, \mathrm{start}} = \Lambda_{i, \mathrm{start}}(\Lambda_{i, \mathrm{end}})\subseteq \Lambda_{i-1}\mbox{  with } \mathcal L_1(\Lambda_{i, \mathrm{start}}) \geq 0.1\mathcal L_1(\Lambda_{i-1}) 
\end{equation}
such that the following holds for each $x\in  \Lambda_{i, \mathrm{start}}$:
\begin{equation}\label{eq-start-end}
\min_{x'\in  \Lambda_{i, \mathrm{end}}} \log \tilde D^{\tilde {\mathbb V}_{u,v}}_{\gamma, \delta\tilde \delta, \eta} (x, x') \leq \E \log \tilde D_{\gamma, \tilde \delta, \eta}(u, v)  + 2 (\log \delta^{-1})^{0.98}\,.
\end{equation}
In words, $\mathsf C_i$ is desirable if any not-so-small subset of $\Lambda_i$ is connected with a not-so-small subset of $\Lambda_{i-1}$ by open boxes. Similar to \eqref{eq-par},
we obtain  that each cell $\mathsf C_i$ is desirable with probability  $1-e^{-\Omega(2^{\sqrt{\log \delta^{-1}}})}$ and thus a union bound verifies that all cells $\mathsf C_2, \ldots, \mathsf C_{d-1}$ are desirable with high probability. 

We also need to consider the cells containing $u$ and $v$. Consider $\mathsf C_1 = \mathsf C_{\delta, u}$. Using a similar but simpler argument, we can show that with probability tending to 1 there exists $\Lambda_{u} \subseteq \Lambda_1$ with  $\mathcal L_1(\Lambda_u) \geq 0.99\mathcal L_1(\Lambda_1)$  such that for $x\in  \Lambda_u$ we have $\log \tilde D^{\tilde {\mathbb V}_{u,v}}_{\gamma, \delta\tilde \delta, \eta} (u, x) \leq \log \E \tilde D_{\gamma, \tilde \delta, \eta}(u, v)  + (\log \delta^{-1})^{0.98}$. 
When this occurs, we say that \textit{$u$ is desirable}. 
As before, 
with high probability,
we have  that
$v$ is desirable, i.e.,  there exists $\Lambda_{v} \subseteq \Lambda_{d-1}$ with  $\mathcal L_1(\Lambda_v) \geq  0.99\mathcal L_1(\Lambda_{d-1})$  such that for each $x\in \Lambda_v$ we have  $
\log \tilde D^{\tilde {\mathbb V}_{u,v}}_{\gamma, \delta\tilde \delta, \eta} (v, x) \leq \log \E \tilde D_{\gamma, \tilde \delta, \eta}(u, v) +  (\log \delta^{-1})^{0.98}$.

 We now work on the event that
$u, v$ are desirable and that $\mathsf C_2, \ldots, \mathsf C_{d-1}$ are desirable, and we describe in what follows how to choose $x_i\in \Lambda_i$ so that \eqref{eq-union-bound-distance} holds. We let $\Lambda_{d-1}^* = \Lambda_v$ and for $i=d-2, \ldots 1$ we recursively let $\Lambda_{i}^* = \Lambda_{i+1, \mathrm{start}} (\Lambda_{i+1}^*)$  (where the set $\Lambda_{i+1, \mathrm{start}} (\cdot)$ is defined as in \eqref{eq-def-B-i-2}). Therefore, we see that $\Lambda_i^*\subseteq \Lambda_{i}$ and $\mathcal L_1(\Lambda_i^*) \geq 0.1\mathcal L_1(\Lambda_i)$.
  Next, we set $x_0 = u$ and sequentially set for $i=1, \ldots, d-1$
 $$x_{i} = \arg\min_{x'\in \Lambda_{i}^*} \tilde D^{\tilde {\mathbb V}_{u,v}}_{\gamma, \delta\tilde \delta, \eta} (x_{i-1}, x')\,.$$
  It remains to verify \eqref{eq-union-bound-distance} for our choices of $x_i$'s. Since $\Lambda_u \cap \Lambda_1 \neq \emptyset$ (this comes from the lower bounds on their Lebesgue measures), we see \eqref{eq-union-bound-distance} holds for $i=0$. By \eqref{eq-start-end}, \eqref{eq-union-bound-distance} holds for $1\leq i\leq d-2$. Finally, \eqref{eq-union-bound-distance} holds for $i=d-1$ by our choice of $\Lambda_{d-1}^* = \Lambda_v$.
			
We finally return to the proof of \eqref{eq-scaling-invariance-approximate}, which is similar to that of Lemma~\ref{lem-scaling-coupling}. Recall that $\theta (x) = a\theta'(x)$ for appropriate $a>0$ and an isometry $\theta'$  is a bijective mapping from $\tilde {\mathbb V}_{z, z'}$  to $\tilde {\mathbb V}_{u, v}$. Thus, $a$ is of the same order as $s_{i}^{-1} K^{1}$ and so $a^{-1} \leq s_i$. Recall the definition of $\hat h$-process as in \eqref{eq:WND_decomposition-stationary}. 
 By an  argument similar to that in the proof of \eqref{eq-coupling-hat-h-eta-1} and \eqref{eq-coupling-hat-h-eta-2},
we have  that with high probability, 
\begin{equation}\label{eq-coupling-hat-h-eta-tilde}
 \max_{x\in \tilde {\mathbb V}_{z, z'}} \max_{n\geq 0, 2^{-n} \leq \epsilon^* s_i^*} |\hat h_{2^{-n}}^{s_i} (x) - \check \eta_{ 2^{-n}}^{\mathsf B, s_i }( x) |+ \max_{x\in \tilde {\mathbb V}_{z, z'}} \max_{n\geq 0, a2^{-n}\leq 1} |\hat h_{ a 2^{-n}}^1(\theta x) - \eta_{ a2^{-n}}(\theta x)| \leq (\log \delta^{-1})^{0.90}\,.
 \end{equation}
Next we need to control $ \hat h_{2^{-n}}^{s_i}(x) - \hat h_{2^{-n}}^{a^{-1}}(x) = \hat h^{s_i}_{a^{-1}}(x)$.  Let $\mathfrak C$ be a maximal collection of points in $\tilde {\mathbb V}_{z, z'}$ such that the pairwise distance is at least $a^{-1}$. Then,  $|\mathfrak C| \leq O(a^2)$. By \eqref{eq-hat-h-continuity} and Lemma~\ref{lem-ferniquecriterion}, we have that $\E \max_{y: |y-x| \leq 4a^{-1}} |\hat h^{s_i}_{a^{-1}}(x) - \hat h^{s_i}_{a^{-1}}(y)| = O(1)$, for all $x\in \mathfrak C$. Thus, by Lemma~\ref{Lem.concentration}, we have that with high probability, 
 $$ \max_{x\in \mathfrak C}\max_{y: |y-x| \leq 4a^{-1}} |\hat h^{s_i}_{a^{-1}}(x) - \hat h^{s_i}_{a^{-1}}(y)| \leq (\log \delta^{-1})^{0.90}\,.$$
In addition, since $ \Var(\hat h^{s_i}_{a^{-1}}(x)) \leq O(1) \log \delta^{-1}$ for all $x\in \mathfrak C$, a  union bound gives that with high probability $\max_{x\in \mathfrak C} |\hat h^{s_i}_{a^{-1}}(x) |\leq (\log \delta^{-1})^{0.90}$. Altogether, this gives that with high probability $\max_{x\in \tilde {\mathbb V}_{z, z'}} |\hat h^{s_i}_{a^{-1}}(x)| \leq 2(\log \delta^{-1})^{0.90}$. Combined with \eqref{eq-coupling-hat-h-eta-tilde}, we have
 that with high probability 
$$\max_{x\in \tilde {\mathbb V}_{z, z'}} \max_{n\geq 0, 2^{-n} \leq a^{-1}} |\hat h_{2^{-n}}^{a^{-1}}(x) - \check \eta_{ 2^{-n}}^{\mathsf B, s_i}( x) |+ \max_{x\in \tilde {\mathbb V}_{z, z'}} \max_{n\geq 0, a2^{-n}\leq 1} |\hat h_{ a 2^{-n}}^1(\theta x) - \eta_{ a2^{-n}}(\theta x)| \leq (\log \delta^{-1})^{0.91}\,.
$$
Combined with  the translation invariance and scaling invariance property of $\hat h$-process, we finally conclude the proof of \eqref{eq-scaling-invariance-approximate}.
\end{proof}

Finally, we prove Lemma \ref{lem-exponent-point-to-boundary}, where we will crucially used Proposition~\ref{prop-concentration} and Lemma~\ref{lem-existence-exponent}.
 \begin{figure}[ht]
 \vspace{-2cm}
  \includegraphics[width=17cm]{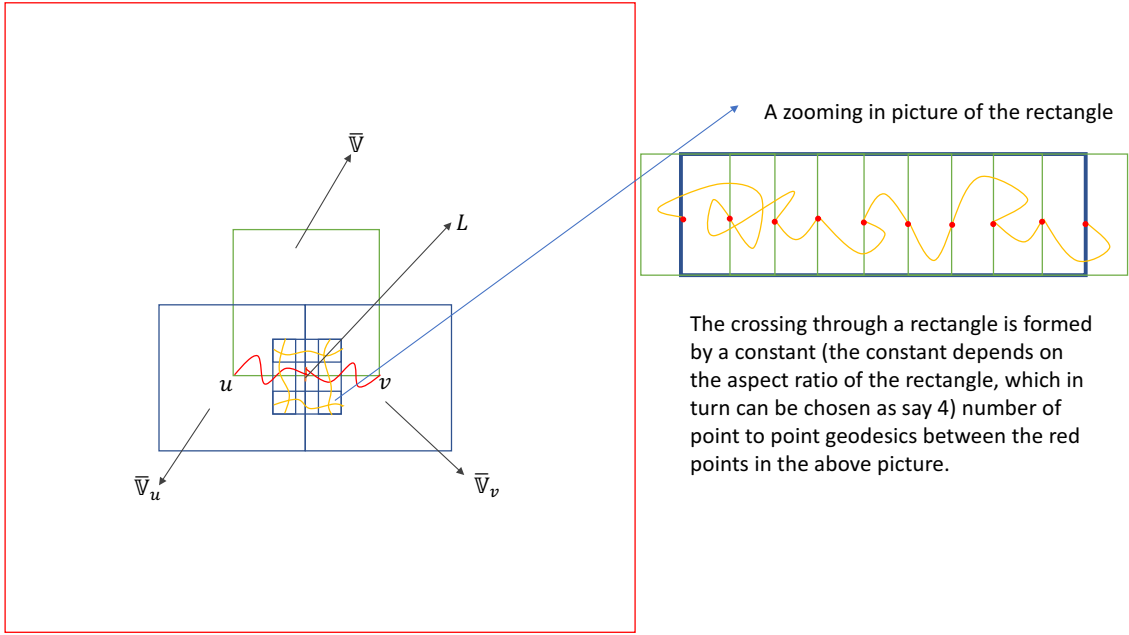}
\\ \vspace{-3cm} \caption{On the left, the big  box is $\mathbb V$ and the inside is an illustration of how we join $u$ and $v$ using geodesics from $u$, $v$ to $L$ as well as an annulus enclosing $L$. On the right is an illustration for the 
crossing in the small rectangle.}
\label{fig-glue} \end{figure}

\begin{proof}[Proof of Lemma~\ref{lem-exponent-point-to-boundary}]
Fix an arbitrarily small $0<\iota<C_{\mathrm{Mc}}/6$.  Let $u, v$ be the left bottom and right bottom corners of  $\bar {\mathbb V}$, respectively (such choice of $u, v$ is somewhat arbitrary). By Lemma~\ref{lem-existence-exponent} there exists $\delta_0$ depending on $(\gamma, \iota)$ such that for all $\delta \leq \delta_0$
\begin{equation}\label{eq-delta_0} 
(\chi - \iota/10)\log \delta^{-1}\leq  \E \log \tilde D_{\gamma, \delta, \eta}(u, v) \leq (\chi + \iota/10)\log \delta^{-1}\,.
\end{equation}
Recall $\lambda = \frac 1 {20}$. We denote $\bar {\mathbb V}_u = \mathbb V_{u, \lambda}$ and $\bar D^{u, 2\lambda}_{\gamma, \delta, \eta}$ by $\bar D_{\gamma, \delta, \eta}$ for brevity. We claim that for any line segment  $ L_\delta \subseteq  \partial \bar {\mathbb V}_u$ with length in $[\delta^{2\iota}/2, \delta^{2\iota}]$, we have
\begin{equation}\label{eq-point-to-segment}
\E \log \min_{x\in  L_\delta} \bar D_{\gamma, \delta, \eta} (u, x) \geq (\chi - 2\iota)\log \delta^{-1}\,.
\end{equation}
Suppose \eqref{eq-point-to-segment} does not hold. We assume without loss (by symmetry) that there exists an $ L_\delta$ on the right vertical boundary of $\bar {\mathbb V}_u$ so that \eqref{eq-point-to-segment} fails. Then, we  give an upper bound on the distance between $u$ and $v$ by gluing the geodesics from $u$ to $ L_\delta$, $v$ to $ L_\delta$ as well as four short crossings through four rectangles (with dimension $10| L_\delta| \times 40 | L_\delta|$)  which altogether form a contour enclosing $ L_\delta$ (see Figure~\ref{fig-glue} for an geometric illustration) --- we remark that each of the four 
rectangle crossings can be formed by a constant number of point to point geodesics thanks to the restriction to $\tilde {\mathbb V}_{x, y}$ in the definition of $\tilde D_{\gamma, \delta, \eta}(x, y)$. With high probability, the balls intersecting both $\partial \mathbb V_{u, \lambda}$ and $\partial \mathbb V_{u, 2 \lambda}$ (respectively,  $\partial \mathbb V_{v, \lambda}$ and $\partial \mathbb V_{v, 2 \lambda}$) have LQG measure larger than $2 \delta^2$ (and thus similar equalities to \eqref{eq-geodesic-range} hold). On this event, one has
 $$
 \tilde D_{\gamma, \delta, \eta}(u, v) \leq  \min_{x\in  L_\delta} \bar D_{\gamma, \delta, \eta} (u, x) +  \min_{x\in  L_\delta} \bar D^{v, 2 \lambda}_{\gamma, \delta, \eta} (v, x) + \sum_{(x, y)} \tilde D_{\gamma, \delta, \eta}(x, y) ,
  $$
where in the third term on the right hand side, the sum is over all pairs of neighboring red points on the right hand side of Figure~\ref{fig-glue} (for each such pair $(x,y)$ we have $|x - y| \leq  10 | L_\delta|$). Thus by \eqref{eq-delta_0}  and a similar scaling argument as in the proof of \eqref{eq-z-open} we have that with probability tending to 1 as $\delta \to 0$
$$\tilde D_{\gamma, \delta, \eta}(x, y) \leq \delta^{-\chi + \iota} \mbox{ for all such } (x, y)\,.$$
Combined with our assumption that \eqref{eq-point-to-segment} fails for $ L_\delta$, we then deduce that $\tilde D_{\gamma, \delta, \eta} (u,v) \leq \delta^{-\chi + \iota/2}$ with probability tending to 1 as $\delta \to 0$, contradicting with \eqref{eq-delta_0} and Proposition~\ref{prop-concentration}. Thus, we have shown that \eqref{eq-point-to-segment} holds.

Next, note that 
$$\min_{x\in \partial \bar {\mathbb V}_u} \bar D_{\gamma, \delta, \eta} (u, x) = \min_{ L_\delta} \min_{x\in  L_\delta} \bar D_{\gamma, \delta, \eta} (u, x)\,,$$
where the minimization is over $4 \delta^{-2\iota}$ many disjoint segments $ L_\delta$ of length $\delta^{2\iota}$. 
Combined with Proposition~\ref{prop-concentration} ( note that $\{(u, L_\delta)\}$ forms a sequence of admissible pairs as required for applying Proposition~\ref{prop-concentration}), this implies that
$$\E \log (\min_{x\in \partial \bar {\mathbb V}_u} \bar D_{\gamma, \delta, \eta} (u, x)) \geq (\chi - 2\iota - C\iota^{1/2}) \log \delta^{-1}$$
for some constant $C>0$. Since we can choose $\iota>0$ arbitrarily small, this
completes the proof of the lemma.
\end{proof}

\section{Appendix}

In this appendix, we record, for use in subsequent work, a few lemmas that can be readily deduced from the techniques employed in this paper; these lemmas are not used in the paper.
Let $\lambda = \frac 1 {20}$ as in Lemma~\ref{lem-exponent-point-to-boundary}. Denote $ \bar {\mathbb V}_u = {\mathbb V}_{u,  \lambda}$ and $ \bar {\mathbb V}_{u, \alpha} = {\mathbb V}_{u, \alpha \lambda}$  for $\alpha \in (0,1)$.
\begin{lemma}\label{lem-exponent-boundary-to-boundary}
Fix  $\alpha \in (0, 1)$.
 Let $\chi$ be as in Lemma~\ref{lem-existence-exponent}. Then, for any $u\in \bar {\mathbb V}$,
\begin{equation}\label{eq-boundary-to-boundary}
\lim_{\delta\to 0} \frac{\E \log (\min_{x\in  \partial \bar {\mathbb V}_{u, \alpha},  y\in \partial \bar {\mathbb V}_u}D_{\gamma, \delta}(x, y))}{\log \delta^{-1}} = \lim_{\delta\to 0} \frac{\E \log (\min_{x\in  \partial \bar {\mathbb V}_{u, \alpha},  y\in \partial \bar {\mathbb V}_u}D_{\gamma, \delta, \eta}(x, y))}{\log \delta^{-1}} = \chi\,.
\end{equation}
\end{lemma}
\begin{proof}
The first equality holds due to Lemma~\ref{lem-LGD-two-fields}  and the main task is to prove the second equality.
By Lemma~\ref{lem-exponent-point-to-boundary} and a similar derivation to \eqref{eq-distance-point-to-boundary-bound}, we get that
for any $\kappa>0$, $v\in \mathbb V$ 
\begin{equation}\label{eq-point-to-boundary-kappa}
\E \log (\min_{ y\in \partial {\mathbb V}_{v, \kappa}}  D_{\gamma, \delta, \eta}(v, y)) = (\chi +o(1)) \log \delta^{-1}\,.
\end{equation}
Thus it suffices to prove a lower bound in \eqref{eq-boundary-to-boundary}. The proof is  similar to that of Lemma~\ref{lem-exponent-point-to-boundary}. 

By Proposition~\ref{prop-concentration}, it suffices to show that for any fixed
 $\iota>0$ and any segment $ L_\delta\subseteq \partial \bar {\mathbb V}_{u, \alpha}$ with length in $[\delta^{2\iota}/2, \delta^{2\iota}]$ we have
 $$
\E \log (\min_{x\in  L_\delta,  y\in \partial \bar {\mathbb V}_u} D_{\gamma, \delta, \eta}(x, y)) \geq  (\chi - 2\iota) \log \delta^{-1}\,.
$$
Suppose the preceding statement fails for some $ L_\delta$. Let $v_{ L_\delta}$ be an arbitrary point on $ L_\delta$. As shown in Figure~\ref{fig-glue} employed in the proof of Lemma~\ref{lem-exponent-point-to-boundary}, we can construct four short crossings through four rectangles (with dimension $10| L_\delta| \times 40 | L_\delta|$)  which altogether form a contour enclosing $ L_\delta$. Consequently,  the union of these short crossings, the geodesic between $ L_\delta$ and $\partial \bar {\mathbb V}_u$, as well as the geodesic between $v_{ L_\delta}$ and $\partial \bar {\mathbb V}_u$ contains a path between $v_{ L_\delta}$ and $\partial \bar {\mathbb V}_u$.
Therefore, by the same argument as in Lemma~\ref{lem-exponent-point-to-boundary}, we get that 
$$\E \log ( \min_{y\in \partial \bar {\mathbb V}_u}D_{\gamma, \delta, \eta}(v_{ L_\delta}, y)) \leq  (\chi - \iota) \log \delta^{-1}\,.$$
This contradicts with \eqref{eq-point-to-boundary-kappa}. Thus, we complete the proof of the lemma by contradiction.
\end{proof}

Fix $\xi>0$ through out the appendix. 
For any Euclidean ball $B$, we denote by $2B$ a Euclidean ball concentric with $B$, whose radius is double that of $B$.
For $\delta>0$ and any two distinct points $u, v\in {\mathbb V}^\xi$, we define a variation of Liouville graph distance $D^{(\mathsf 2)}_{\gamma, \delta, \xi}(u, v)$ to be the minimal $d$ such that there exist Euclidean balls $B_1, \ldots, B_d \subseteq \mathbb V^\xi$ with rational centers
and $M_\gamma(2B_i)\leq \delta^2$ for $1\leq i\leq d$, whose union contains a path from $u$ to $v$.

For an Euclidean ball $B$ with radius $r$ centered at $z$, we define its circle-average-approximate-LQG measure by $M^\circ_{\gamma}(B) = r^{2+\gamma^2/2} e^{\gamma h_r(z)}$,
compare with \eqref{eq-limit-LQG}. For $\delta>0$ and any two distinct points $u, v\in {\mathbb V}^\xi$, we define another variation of Liouville graph distance $D^{\circ}_{\gamma, \delta, \xi}(u, v)$ to be the minimal $d$ such that there exist Euclidean balls $B_1, \ldots, B_d \subseteq \mathbb V^\xi$ with rational centers
and $M^\circ_\gamma(B_i)\leq \delta^2$ for $1\leq i\leq d$, whose union contains a path from $u$ to $v$.

We define $D'_{\gamma, \delta, \xi}(x, y)$ to be a version of the approximate Liouville graph distance where we restrict to cells in $\mathbb V^\xi$. One can verify that our proofs for  Lemmas~\ref{lem-approximate-LGD-two-deltas}, \ref{lem-LGD-compare}, \ref{lem-LGD-two-fields} and Corollary~\ref{cor-LGD-two-deltas} as well as Proposition~\ref{prop-concentration} extend automatically to $D'_{\gamma, \delta, \xi}$. Recall $C_{\mathrm{Mc}}$ as specified in Lemma~\ref{lem-partition-minimal-cell}.

\begin{prop}\label{prop-approximate-LGD-Ewain}
For any fixed $0<\xi< C_{\mathrm{Mc}}/3$ there exists a constant $c=c(\gamma, \xi)$ so that for any fixed $\iota>0$ and any sequence of $\xi$-admissible pairs  $(A_\delta, B_\delta )$, 
$$ \min_{x\in A_\delta, y \in B_\delta}D_{\gamma, \delta}(x, y) \cdot \delta^{\iota} \leq \min_{x\in A_\delta, y\in B_\delta}D^{(\mathsf 2)}_{\gamma, \delta, \xi}(x, y) \leq   \min_{x\in A_\delta, y\in B_\delta} D_{\gamma, \delta}(x, y) \cdot \delta^{-\iota}\,,
$$
with $(c\cdot \iota^2)$-high probability. The preceding statement remains true if we replace $D^{(\mathsf 2)}_{\gamma, \delta, \xi}$ by $D^{\circ}_{\gamma, \delta, \xi}$.
\end{prop}
\begin{proof}
 By  Lemma~\ref{lem-exponent-boundary-to-boundary} and Proposition~\ref{prop-concentration}, we have that with $(c\cdot \iota^2)$-high probability 
 $$ \min_{x\in A_\delta, y\in B_\delta}D'_{\gamma, \delta}(x, y) \cdot \delta^{\iota} \leq \min_{x\in A_\delta, y\in B_\delta}D'_{\gamma, \delta, \xi}(x, y) \leq   \min_{x\in A_\delta,\
  y\in B_\delta} D'_{\gamma, \delta}(x, y) \cdot \delta^{-\iota}\,.$$
Combined with Proposition~\ref{prop-approximate-LGD}, it implies that Proposition~\ref{prop-approximate-LGD-Ewain} follows provided that with $(c\cdot \iota)$-high probability 
 \begin{equation}\label{eq-approximate-Ewain}
 \begin{split}
 & \min_{x\in A_\delta, y\in B_\delta}D'_{\gamma, \delta, \xi}(x, y) \cdot \delta^{\iota} \leq \min_{x\in A_\delta, y\in  B_\delta}D^{(\mathsf 2)}_{\gamma, \delta, \xi}(x, y) \leq   \min_{x\in A_\delta, y\in  B_\delta} D'_{\gamma, \delta, \xi}(x, y) \cdot \delta^{-\iota}\,,\\
   & \min_{x\in A_\delta, y\in  B_\delta}D'_{\gamma, \delta, \xi}(x, y) \cdot \delta^{\iota} \leq \min_{x\in A_\delta, y\in  B_\delta}D^{\circ}_{\gamma, \delta, \xi}(x, y) \leq   \min_{x\in 
   A_\delta, y\in  B_\delta} D'_{\gamma, \delta, \xi}(x, y) \cdot \delta^{-\iota}.
  \end{split}
 \end{equation}
The proof of \eqref{eq-approximate-Ewain} is  similar to that of Proposition~\ref{prop-approximate-LGD}. Thus, we only briefly discuss how to adapt the proof of Proposition~\ref{prop-approximate-LGD}. 

For $D^{(\mathsf 2)}_{\gamma, \delta, \xi}$, since $D^{(\mathsf 2)}_{\gamma, \delta, \xi} \geq D_{\gamma, \delta, \xi}$, it remains to bound $D^{(\mathsf 2)}_{\gamma, \delta, \xi}$ by $D'_{\gamma, \delta, \xi}$ from above. We repeat the proof of Proposition~\ref{prop-approximate-LGD}, but with the following change:  we will now define a new version of $\Phi_{B, \delta}$ (similar to that in Definition~\ref{def-E-delta-B-prime}) to be the minimal number of Euclidean balls 
$\mathsf B$ with  $M_\gamma(2\mathsf B) \leq \delta^2$ that covers $\partial B$.
(The only difference is that we used $M_{\gamma}(2\mathsf B)$ in the preceding definition as opposed to $M_{\gamma}(\mathsf B)$  as in Definition~\ref{def-E-delta-B-prime}.) One can then just repeat the arguments  with this version of $\Phi_{B, \delta}$ to conclude the proof on the upper bound --- the only place that needs to be changed
is in the proof of \eqref{eq-B-percolation-Psi} and \eqref{eq-B-good-Psi}, where the required change is noting but enlarging a few constants which have been absorbed by much larger terms in the earlier proof.

Next, we consider $D^{\circ}_{\gamma, \delta, \xi}$. By \cite[Proposition 3.2]{DG16} (which states that the circle average process and our $\hat h$-process are close to each other) and Lemma \ref{lem-hat-h-eta}, we get that with high probability 
$$\max_{j: 2^{-j} \geq \delta^{C_{\mathrm{mc}}+10}}\max_{x\in \mathbb V^\xi} |\eta_{2^{-j}}(x) - h_{2^{-j}}(x)| = O(\sqrt{\log \delta^{-1}})\,.$$
This, together with Lemma~\ref{lem-neighboring-cell}, implies that with high probability
$$ \min_{x\in A_\delta, y\in B_\delta}D'_{\gamma, \delta e^{(\log \delta^{-1})^{0.6}}, \xi}(x, y)  \leq \min_{x\in A_\delta, y\in B_\delta}D^{\circ}_{\gamma, \delta, \xi}(x, y) \leq   \min_{x\in A_\delta, y\in B_\delta} D'_{\gamma, \delta e^{-(\log \delta^{-1})^{0.6}}, \xi}(x, y)\,.$$
Combining Lemma~\ref{lem-approximate-LGD-two-deltas}, we complete the proof of \eqref{eq-approximate-Ewain}, and thus the proof of the proposition.
\end{proof}

\def\cprime{$'$}

\end{document}